%% file: main-new.tex
\documentclass[a4paper]{amsart}

\usepackage[left = 1.5in, right = 1.5in, top=1in, bottom=1in]{geometry}

\usepackage{theoremref}
\usepackage{tikz-cd}
\usepackage{quiver}
\usetikzlibrary{nfold}
\usepackage{upgreek}

\usepackage{braket}
\PassOptionsToPackage{disable}{todonotes}
\usepackage{todonotes}
\usepackage{mathtools}
\usepackage{comment}
\usepackage{graphicx}

\usepackage{esint}
\usepackage{stackengine,scalerel}

\usepackage{xcolor}
\definecolor{acid}{HTML}{479107}
\definecolor{crimson}{RGB}{177,12,12}
\definecolor{teal}{HTML}{008080}
\definecolor{orange}{HTML}{ef9e8b}
\usepackage{hyperref}
\hypersetup{
	colorlinks=true,
	linktocpage=true, 
	breaklinks=true, 
	pageanchor=true,
	pdfpagemode=UseNone,
	urlcolor=orange, 
	linkcolor=teal, 
	citecolor=acid,
	pdfcreator={pdfLaTeX},
	pdfproducer={LaTeX with hyperref}
}

\usepackage{mathpazo} 
\renewcommand{\phi}{\varphi}
\renewcommand{\epsilon}{\varepsilon}

\renewcommand{\Set}{{\bf Set}}
\newcommand{\CAT}{{\bf CAT}}
\newcommand{\CoLaxAlg}{{\rm ColaxAlg}}
\newcommand{\Bcat}{\mathcal{B}}
\newcommand{\Acat}{\mathcal{A}}
\newcommand{\id}{\operatorname{id}}

\renewcommand{\implies}{\Rightarrow}
\newcommand{\op}[1]{{#1}^{\operatorname{op}}}
\newcommand{\funcat}[2]{\left[#1, #2\right]}
\newcommand{\Hom}[3]{{#1}({#2}, {#3})}
\newcommand{\kl}[2]{{#1}^{#2}}
\DeclareMathOperator*\opcolim{colim^{oplax}}
\newcommand{\iso}{\cong}
\let\u\underline
\renewcommand\tilde\widetilde
\newcommand{\Mod}[1]{\operatorname{Mod}({#1})}

\newcommand{\Pos}{\mathbf{Pos}}
\renewcommand{\leq}{\leqslant}
\DeclareMathOperator*{\colim}{colim}
\NewDocumentCommand{\lend}{e{_^}}{
	\newintop{\hspace{-0.09em}\leftarrow}[#1][#2]%
}%
\NewDocumentCommand{\oplend}{e{_^}}{
	\newintop{\hspace{-0.09em}\rightarrow}[#1][#2]%
}%
\NewDocumentCommand{\psend}{e{_^}}{
	\newintop{\sim}[#1][#2]%
}%
\NewDocumentCommand{\newintop}{moo}{%
	\ThisStyle{\ensurestackMath{%
			\stackengine{0pt}{%
				\SavedStyle\int%
				\IfValueT{#2}{_{#2}}%
				\IfValueT{#3}{^{#3}}%
			}{%
				\SavedStyle#1%
			}{O}{l}{F}{F}{L}}}}%

\newtheorem{theorem}{Theorem}[section]
\newtheorem{proposition}[theorem]{Proposition}
\newtheorem{lemma}[theorem]{Lemma}
\newtheorem{corollary}[theorem]{Corollary}

\newtheorem*{theorem*}{Theorem}
\newtheorem*{proposition*}{Proposition}
\newtheorem*{lemma*}{Lemma}
\newtheorem*{corollary*}{Corollary}
\newtheorem*{claim*}{Claim}

\theoremstyle{definition}
\newtheorem{definition}[theorem]{Definition}
\newtheorem{remark}[theorem]{Remark}

\newtheorem{notation}[theorem]{Notation}
\newtheorem{example}[theorem]{Example}

\title{Ultracategories via Kan extensions of relative monads}
\author{Umberto Tarantino}
\author{Joshua Wrigley}

\date{}


\thanks{The research is supported by the Agence Nationale de la Recherche (ANR), project ANR-23-CE48-0012-01.}

\address[Umberto Tarantino]{Université Paris Cité, CNRS, IRIF, F-75013, Paris, France.}
\email{tarantino@irif.fr}
\urladdr{https://utarantino.github.io/}

\address[Joshua Wrigley]{Université Paris Cité, CNRS, IRIF, F-75013, Paris, France.}
\email{wrigley@irif.fr}
\urladdr{https://jlwrigley.github.io/}

\begin{document}
	\begin{abstract}
		Many structured categories of interest are most naturally described as algebras for a relative monad, but turn out nonetheless to be algebras for an ordinary monad.  We show that, under suitable hypotheses, the left oplax Kan extension of a relative 2-monad on categories yields a pseudomonad having the same category of colax algebras.  In particular, we apply this to the study of ultracategories to recover the `ultracompletion' pseudomonad.
	\end{abstract}
	
	\maketitle
	
	\section*{Introduction}
	Monad theory allows the category theorist to recognise algebraic structure in surprising places, such as compact Hausdorff spaces, which are famously the algebras for the \emph{ultrafilter monad} on sets \cite{kennison-gildenhuys,manes,leinster-codensity}.  The ultrafilter monad is an example of a \emph{codensity monad}, whose underlying functor is the right Kan extension of the inclusion of finite sets into sets along itself.  Our aim in this paper is to make \emph{ultracategories} emerge just as naturally, as algebras for a universally induced monad on categories.
	
	Ultracategories offer a solution to a recurring problem in categorical logic: on its own, the category of models and model homomorphisms of a first-order theory does not determine the theory up to any useful notion of logical equivalence (see e.g.\ \cite{lascar} and \cite[Remark D3.5.4]{elephant}). To reconstruct the theory we must endow the category of models with additional abstract structure, such as an \emph{ultrastructure}. Ultrastructure was introduced by Makkai \cite{makkai} with the aim of capturing the formal aspects of forming \emph{ultraproducts} of models. The canonical ultrastructure on a category of models, given by the familiar notion of an ultraproduct \cite[\S 4]{chang-keisler}, provides the necessary extra structure to recover the theory up to \emph{Morita equivalence} (see \cite[Theorem 4.1]{makkai} and \cite[Theorem 3.2]{makkai-scc}). In this sense, ultrastructure plays the role of the usual Stone topology on the set of models of a propositional theory -- allowing for a Stone-like duality for first-order logic -- and thus behaves as a kind of topological gadget on a category. More specifically, ultracategories can be seen as a categorification of compact Hausdorff spaces: indeed, the datum of an ultrastructure on a set $X$ coincides with that of a compact Hausdorff topology on $X$ \cite[Theorem 3.1.5]{lurie}. 
	
	At its core, an {ultrastructure} on a category $C$ consists of, for each $X$-indexed sequence of objects $h \colon X \to C$ and each ultrafilter $\nu$ on $X$, a choice of an object $h^C(\nu) \in C$ intended to play the role of the ultraproduct. This datum is called a \emph{pre-ultracategory} in \cite{makkai}, and its `algebraic' flavour was already recognised in \cite{marmolejo-thesis}. Therein, Marmolejo shows that a class of ultracategories appears by first considering a 2-monad on $\CAT$, the 2-category of categories, and then taking a second 2-monad on the category of algebras. The question of whether Makkai's ultracategories can be characterised as algebras for a monad on $\CAT$ is however still open. Later, Lurie \cite{lurie} introduced a new axiomatisation of ultracategories and their morphisms which captures sufficient data so as to recover, and improve, Makkai's reconstruction result \cite[Theorems 2.2.2 and 2.3.1]{lurie}. Lurie's ultracategories were proved in \cite{hamad} to be colax algebras for a pseudomonad on $\CAT$, which had been dubbed the \emph{ultracompletion} pseudomonad in \cite{rosolini-ultracompletion}. As much as Lurie's axiomatisation of an ultrastructure is significantly simpler than Makkai's, as noted in \cite{ivan-ultra}, the origin of the axioms still seems obscure. 
	
	The main novelty of our approach is to embed the theory of ultracategories within the field of \emph{relative monad theory}. Starting from the `no-iteration' presentation of monads described in \cite{manes,marmolejo-wood}, relative monads were introduced in \cite{ACU-conference,ACU-relative} as a generalisation of monads where the underlying functor need not be an endofunctor. In essence, just as a monad $T \colon C \to C$ can be specified by choosing a lift $T c' \to T c$ for each arrow $c' \to T c$ in $C$, a relative monad $T \colon D \to C$ over a fixed \emph{root} functor $J \colon D \to C$ is given by choosing a lift $T d \to T c$ for each arrow $J d \to T c$ in $C$. Relative monads -- and their 2-dimensional analogues, the \emph{relative pseudomonads} of \cite{pseudomonads-fiore} -- turn out to be a flexible notion, incorporating many examples of functors which are `almost' a monad but cannot be iterated: most notably, the presheaf construction on small categories (see e.g.\ \cite{arkor-saville-slattery}). The key observation making this paper possible is that the datum of a pre-ultracategory resembles that of an algebra for a relative monad over the inclusion $\Set \hookrightarrow\CAT$ of sets as discrete categories. Specifically, the functor underlying the monad is the composite of the ultrafilter monad $\beta \colon \Set \to \Set$ with the inclusion $\Set \hookrightarrow \CAT$ itself.
	
	Since their introduction, relative monads were linked with the theory of Kan extensions since we often want to extend, universally, a relative monad to an ordinary monad. Indeed, in \cite{ACU-relative} the authors show that, for a suitable ``well-behaved'' root $J \colon D \to C$, the algebras for a $J$-relative monad $T \colon D \to C$ can be realised as algebras for a monad on $C$ whose underlying functor is the left Kan extension of $T$ along $J$. A naïve 2-categorical generalisation of the same ``well-behavedness'' assumptions would not be satisfied by the inclusion $\Set \hookrightarrow \CAT$ involved in our motivating example.  In particular, one of these assumptions is that the root functor $J \colon D \to C$ is \emph{dense}, i.e.\ every object of $C$ is a colimit of objects in the image of $J$. This would not be satisfied by the inclusion $\Set \hookrightarrow \CAT$, since a colimit of discrete categories is itself discrete. 
	
	Our intention in this paper is to take the first step towards a more general framework to `unrelativise' relative 2-monads in a similar fashion to \cite{ACU-relative}, with the aim to realise (a notion of) ultracategories as algebras for a monad on $\CAT$. As in \cite{rosolini-ultracompletion,hamad}, our strategy to achieve this will be to construct, starting from a category $C$, a category $\tilde\beta C$ of \emph{formal ultraproducts} in $C$ given by triples of a set $X$, an $X$-indexed sequence of objects $h \colon X \to C$, and an ultrafilter $\nu\in \beta X$. The intuition here is that $C$ is an ultracategory if such formal ultraproducts can be `realised' in $C$ by means of a functor $\tilde \beta C \to C$. In particular, if $C$ is the category of models of a first-order theory $\mathbb T$, a formal ultraproduct in $C$ corresponds to a family $\set{ M_x }_{x\in X}$ of models of $\mathbb T$ together with an ultrafilter $\nu\in\beta X$, which can be mapped to the `actual' ultraproduct $\prod_{\nu} M_\bullet$ (see Remark \ref{rem:algebra-map}). Crucially, our construction of $\tilde \beta C$ does not depend on any specific property of ultrafilters or ultraproducts, allowing us to generalise to other contexts.
	
	\subsection*{Our contribution}We single out a notion of ultracategories emerging naturally as algebras for a pseudomonad on $\CAT$ universally induced by the ultrafilter monad. Starting from the observation that $\beta \colon \Set \to \Set$ lifts to a relative 2-monad $\beta \colon \Set \to \CAT$ over the inclusion $\Set \hookrightarrow \CAT$, we will define \emph{weak ultracategories} as colax algebras for this relative 2-monad. Concretely, weak ultracategories are defined by the same data as Lurie's ultracategories, but with a slightly weaker axiomatisation. Their morphisms, however, coincide precisely with Lurie's \emph{ultrafunctors} \cite[Definition 1.4.1]{lurie}. Thus, weak ultracategories can still serve as a framework for reconstruction theorems in logic (\cite[Theorem 4.1 and Corollary 8.3]{makkai}, \cite[Theorem 2.2.2]{lurie}). 
	
	To arrive at the (pseudo)monadicity of weak ultracategories over $\CAT$, we will extend the results of \cite[\S 3, 4]{ACU-relative} by introducing \emph{left oplax Kan extensions} (Definition \ref{def:left-oplax-kan-extension}), analogues in the lax setting of (pointwise) left Kan extensions, and we will give an explicit construction of left oplax Kan extensions for $\CAT$-valued 2-functors. The main results of the paper can be summed up in the following theorem.
	
	\begin{theorem*}[\protect{\ref{thm:tilde-t-is-a-pseudomonad} and \ref{thm:algebras_preserved}}]
		Let $T \colon \Bcat \to \CAT$ be a relative 2-monad over a 2-functor $J \colon \Bcat \to \CAT$ and let $\Tilde{T} \colon \CAT \to \CAT$ denote the left oplax Kan extension of $T$ along $J$.  Under the assumptions that: 
		\begin{enumerate}
			\item $\Bcat$ has a terminal object $1_\Bcat$ that is preserved by $J$,    
			\item $\Bcat$ admits all oplax colimits for diagrams of shape $\op{(Jb)}$ for $b\in \Bcat$ and they are preserved by $J$,
			\item and $J$ is 2-fully faithful,
		\end{enumerate}
		$\Tilde{T}$ carries the structure of a pseudomonad on $\CAT$ whose 2-category of colax algebras is isomorphic to that of $T$.
	\end{theorem*}
	
	Pseudomonadicity of weak ultracategories will then follow by applying the previous theorem to the \emph{relative ultrafilter 2-monad} $\beta \colon \Set \to \CAT$. Consequently, we conclude that the \emph{weak ultracompletion} pseudomonad $\tilde \beta \colon \CAT \to \CAT$ is obtained via a combination of Kan extensions and oplax Kan extensions (Corollary \ref{cor:ultracategories-universally}). Thus, if ``ultraproducts are categorically inevitable'' \cite{leinster-codensity}, so too are weak ultracategories.
	
	
	\subsection*{Overview}
	The paper proceeds as follows.
	\begin{enumerate}
		\item First, we recall some preliminaries on relative 2-monads and their 2-categories of colax algebras.
		\item Second, we describe our motivating example of colax algebras for a relative 2-monad: {weak ultracategories}. We also deduce that, in the relative monadic approach, it automatically follows that \emph{ultrasets}, i.e.\ sets endowed with (weak) ultrastructure, are precisely the compact Hausdorff spaces (Remark \ref{rem:ultrasets}, cf.\ \cite[Theorem 3.1.5]{lurie}).
		\item We will `unrelativise' a relative 2-monad on $\CAT$ by taking its \emph{left oplax Kan extension}, which is described in Section \ref{sec:kan-ext}.
		\item Next, in Section \ref{sec:absolution}, under appropriate hypotheses on the root 2-functor, we endow the left oplax Kan extension of a relative 2-monad on $\CAT$ with the structure of a pseudomonad. In order to do this, we will first introduce a `relative composition' operation on the category of $\CAT$-valued 2-functors and oplax transformations, making it into a \emph{skew-monoidal category}.
		\item In Section \ref{sec:algebras} we show that {colax} algebras of a relative 2-monad on $\CAT$ coincide with those of its left oplax Kan extension, seen as a pseudomonad on $\CAT$ as in Section \ref{sec:absolution}.
		\item Finally, we combine our earlier results to deduce that weak ultracategories are pseudomonadic over $\CAT$, where the underlying 2-functor of the pseudomonad is obtained by taking Kan extensions and oplax Kan extensions. As another application of our framework, we introduce \emph{prime categories} as an `ordered' analogue of weak ultracategories induced by the \emph{prime upper filter monad} on the category of posets.
	\end{enumerate}
	\subsection*{Acknowledgements}  We would like to thank Sam van Gool for his detailed comments while preparing this work.
	
	\section{Relative monads and their algebras}\label{sec:prelims}
	
	We begin by reviewing the relevant notions from the theory of relative monads and their algebras, both for the convenience of the reader and so as to fix notations.
	
	\subsection{2-categories}
	We will assume basic knowledge of 2-category theory.  By \emph{2-category}, \emph{2-functor}, \emph{2-monad} etc., we mean the strictest 2-categorical notion, e.g.\ a 2-functor preserves identity and composition on the nose.  In particular, by a \emph{2-fully faithful} 2-functor $J \colon \Bcat\to \Acat$ we mean a 2-functor inducing {isomorphisms} $\Hom {\Bcat} {b}{b'} \iso \Hom{\Acat}{Jb}{Jb'}$ between hom-categories, rather than equivalences.  We write $\CAT$ for the 2-category of (large) categories, with $[C,D]$ denoting the category of functors $C \to D$.  We refer the reader to, e.g., \cite{2dimensionalcategories} for more details. For a category $C$, we will systematically abuse notation by writing $c\in C$ and $f \colon c\to c' \in C$ in place of $c \in \operatorname{Ob} (C)$ and $f \colon c \to c' \in \operatorname{Arr}(C)$ respectively.
	
	\subsection{Relative 2-monads}Recall from \cite{manes,marmolejo-wood} that a monad on a category $C$ can be equivalently described as an \emph{extension system} $\braket{T,\eta,(-)^\ast}$, that is:
	\begin{enumerate}
		\item for each $c \in C$, an object $T c \in C$ and a \emph{unit arrow} $\eta_c \colon c \to T c \in C$,
		\item for each pair $d,c \in C$, an \emph{extension function} $(-)^\ast \colon \Hom C d {Tc}  \to \Hom C {Td} {Tc}$, 
	\end{enumerate}
	satisfying equational axioms (see Definition \ref{def:relative-2-monad} below).	In particular, we recover the usual description in terms of an endofunctor $T \colon C \to C$ by letting $T f \coloneqq ( \eta_{c'} \circ f )^*$ for an arrow $f \colon c \to c' \in C$.
	
	In \cite{ACU-relative}, this definition is extended to that of a \emph{relative monad}, which allows for an underlying functor between different categories. The \emph{relative pseudomonads} of \cite{pseudomonads-fiore} are then a further 2-dimensional generalisation, which we here recall in the particular case which will be used in this paper: namely, that of a \emph{relative 2-monad} in the sense of \cite{arkor-relative-monads} -- i.e.\ where the root pseudofunctor is a 2-functor and where all structure isomorphisms are taken to be the identity.
	
	\begin{definition}\label{def:relative-2-monad}
		Let $J \colon \Bcat \to \Acat$ be a 2-functor between 2-categories, which we will refer to as the \emph{root}.  A \emph{$J$-relative 2-monad} $\braket{T,\eta,(-)^\ast}$ is the datum of:
		\begin{enumerate}
			\item for each $b \in \Bcat$, an object $Tb \in \Acat$ and a \emph{unit 1-cell} $\eta_b \colon Jb \to Tb \in \Acat$;
			\item for each pair $a,b \in \Bcat$, an \emph{extension functor} $(-)^\ast \colon \Hom \Acat {Ja} {Tb} \to \Hom\Acat {Ta}{Tb}$,
		\end{enumerate}
		satisfying:
		\begin{enumerate}
			\item[(a)] $\eta_b^* = \id_{ T b }$ for each object $b\in \Bcat$;
			\item[(b)] $f^* \circ \eta_c = f $ for each 1-cell $ f \colon Jc \to Tb \in \Acat$;
			\item[(c)] $(f^*\circ g)^* = f^*\circ g^* $ for each pair of 1-cells $ f \colon  Jc \to Tb, g \colon J d \to T c\in \Acat$.
		\end{enumerate}
		
		Under this definition, $T$ can be extended to a 2-functor $\Bcat\to \Acat$, where the action on an arrow $f \colon b \to b' \in \Bcat$ is given by $(\eta_{b'} \circ Jf)^* \colon Tb \to Tb'$, while $\eta$ and $(-)^*$ become 2-natural transformations $J \implies T$ and $\Hom{\Acat} {J-} {T-} \implies \Hom{\Acat} {T-} {T-}$ respectively (cf.\ \cite[Proposition 4.7]{pseudomonads-fiore}).
	\end{definition}

	\begin{example}\label{ex:relative_over_identity}
		Relative 2-monads over the identity 2-functor $\id \colon \Acat\to\Acat$ coincide with 2-monads on $\Acat$ in the usual sense (cf.\ \cite[Remarks 3.4, 3.5]{arkor-saville-slattery}).
	\end{example}
	\begin{example}\label{ex:Jff-and-monad-yields-relative}
		For a 2-monad $\braket{T,\eta,(-)^\ast}$ on $\Bcat$ and a 2-fully faithful 2-functor $J \colon \Bcat \to \Acat$, the 2-functor $J T \colon \Bcat\to \Acat$ carries the structure of a $J$-relative 2-monad where:
		\begin{enumerate}
			\item for $b \in \Bcat$, the unit is given by $J \eta_b \colon J b \to J T b$ ;
			\item for $b,c\in \Bcat$, the extension functor is given by 
			\[\begin{tikzcd}
				{\Hom{\Acat}{Jc}{J T b}} \iso {\Hom{\Bcat}{c}{T b}} & {\Hom{\Bcat}{Tc}{Tb}} \iso {\Hom{\Bcat}{JTc}{JTb}}
				\arrow["{(-)^*}", from=1-1, to=1-2]
			\end{tikzcd}\]
			which we still denote as $(-)^*$, so that for a 1-cell $f \colon c \to Tb$ in $\Bcat$ we can unambiguously write $J f^*$ for $(Jf)^* = J(f^*)$.
		\end{enumerate}
		We omit the routine verification that these data satisfy Definition \ref{def:relative-2-monad}.
	\end{example}
	
	\subsection{Algebras for a relative 2-monad}
	Relative 2-monads admit notions of algebras which extend the usual notion of algebra for a (2-)monad. The 2-dimensional setting, in particular, allows us to distinguish between \emph{lax}, \emph{colax}, \emph{pseudo} and \emph{strict} algebras depending on the 2-cells involved in the definitions, and similarly for their morphisms. We recall here the definitions that we will use in this paper, namely those regarding \emph{colax algebras}; for a more extensive treatment, we refer the reader to \cite{arkor-saville-slattery}.
	
	The construction of the 2-category of colax algebras given in Definition \ref{def:colax-algebra} below is not concocted in a vacuum.  It is intended to capture a 2-dimensional generalisation of the universal property of the category of algebras for a relative monad amongst all resolutions \cite[Theorem 2.12]{ACU-relative}, itself a generalisation of the universal property of the category of Eilenberg-Moore algebras for a monad (see \cite[Theorem 2.2]{eilenberg-moore}). However, the necessary literature regarding `colax-resolutions' of relative 2-monads is not yet extant, the closest parallel being given for pseudo-resolutions in \cite[\S 6]{arkor-saville-slattery}.  We hope that the example of weak ultracategories (Definition \ref{def:ultracategory}) serves as the catalyst for further research generalising \cite{ACU-relative,arkor-saville-slattery} to the colax setting -- see also Example \ref{ex:free_algebras} and \cite[Remark 3.4]{arkor-saville-slattery}.

	\begin{definition}\label{def:colax-algebra}
		Let $J \colon  \Bcat\to \Acat$ be a 2-functor and let $\braket{T,\eta,(-)^\ast}$ be a $J$-relative 2-monad, which we abbreviate to $T$.  
		
		A \emph{colax algebra} for $T$ is a tuple $\braket{A,(-)^A, \Gamma,\Delta}$ consisting of:
		\begin{enumerate}
			\item an object $A \in \Acat$;
			\item for each $b\in \Bcat$, a functor $(-)^A \colon \Hom{\Acat}{Jb}{A} \to \Hom{\Acat}{Tb}{A}$, called the \emph{extension operator};
			\item a family of 2-cells $ \Gamma_{h} \colon \kl h A \circ \eta_b \implies h$ for each $h \colon Jb \to A \in \Acat$, natural in $h$;
			\item a family of 2-cells $\Delta_{h,k} \colon \kl h A \circ k^\ast \implies (\kl h A \circ k)^A $ for each $ h \colon J b \to A, k \colon Jc \to Tb \in \Acat $, natural in $h$ and $k$,
		\end{enumerate}
		satisfying the commutativity equations
		\[(a) \begin{tikzcd}[sep = large]
			{\kl h A \eta_b^*} & {\kl {(\kl h A \eta_b)} A} \\
			& {\kl h A}
			\arrow["{{\Delta_{h,\eta_b}}}", from=1-1, to=1-2]
			\arrow[equals, from=1-1, to=2-2]
			\arrow["{{\kl {(\Gamma_h)}A}}", from=1-2, to=2-2]
		\end{tikzcd} \qquad (b) \begin{tikzcd}[sep = large]
			{\kl h A k^* \eta_c} & {\kl { ( \kl h A k ) } A \eta_c} \\
			& {\kl h A k}
			\arrow["{{\Delta_{h,k}*\eta_c}}", from=1-1, to=1-2]
			\arrow[equals, from=1-1, to=2-2]
			\arrow["{{\Gamma_{\kl h A k}}}", from=1-2, to=2-2]
		\end{tikzcd}\]
		\[ (c) \quad \begin{tikzcd}[sep = large]
			{\kl h A (k^* l)^*} & {\kl h A k^* l^*} & {\kl {(\kl h A k)} A l^*} \\
			{\kl {(\kl h A k^*l)} A} && {\kl{( \kl{( \kl h A k)}A l)}A}
			\arrow[equals, from=1-1, to=1-2]
			\arrow["{{\Delta_{h,k^*l}}}"', from=1-1, to=2-1]
			\arrow["{{\Delta_{h,k}*l^*}}", from=1-2, to=1-3]
			\arrow["{{\Delta_{\kl h A k, l}}}", from=1-3, to=2-3]
			\arrow["{{\kl{(\Delta_{h,k}*l)}A}}"', from=2-1, to=2-3]
		\end{tikzcd}\]
		for every $h \colon J b \to A, k \colon J c \to T b, l \colon Jd \to T c \in \Acat$.  We say that $A$ is \emph{strict} if each 2-cell $\Gamma_h$ and each 2-cell $\Delta_{h,k}$ is the identity.
		
		A \emph{colax morphism} $\braket{U, \upsilon} \colon {\braket{A,(-)^A, \Gamma, \Delta} \to\braket{A',(-)^{A'}, \Gamma', \Delta'}}$ of colax algebras for $T$ consists of:
		\begin{enumerate}
			\item a 1-cell $U \colon A \to A' \in \Acat$;
			\item a family of $2$-cells $\upsilon_h \colon  U\kl h A \implies \kl {(Uh)} {A'}$ for each $h \colon Jb \to A \in \Acat$, natural in $h$,
		\end{enumerate}
		satisfying the commutativity equations
		\[(a)\begin{tikzcd}[sep = large]
			{U \kl h A \eta_b} & {\kl{(Uh)} {A'} \eta_b} \\
			& Uh
			\arrow["{\upsilon_h*\eta_b}", from=1-1, to=1-2]
			\arrow["{U*\Gamma_h}"', from=1-1, to=2-2]
			\arrow["{\Gamma'_{Uh}}", from=1-2, to=2-2]
		\end{tikzcd} \quad  (b)\begin{tikzcd}[row sep = large]
			{U \kl h A k^*} && {\kl{(Uh)}{A'} k^*} \\
			{U\kl{(\kl h A k)}{A}} & {\kl{(U\kl h A k)}{A'}} & {\kl{ (\kl{(Uh)}{A'} k) }{A'}}
			\arrow["{\upsilon_h*k^*}", from=1-1, to=1-3]
			\arrow["{U*\Delta_{h,k}}"', from=1-1, to=2-1]
			\arrow["{\Delta'_{Uh,k}}", from=1-3, to=2-3]
			\arrow["{\upsilon_{\kl h A k}}"', from=2-1, to=2-2]
			\arrow["{\kl{(\upsilon_h*k)}{A'}}"', from=2-2, to=2-3]
		\end{tikzcd}\] 
		for every $h \colon  J b \to A, k \colon J c \to T b \in \Acat$. A \emph{lax morphism} is defined analogously, reversing the direction of each 2-cell $\upsilon_h$. In particular, a (co)lax morphism $\braket{U,\upsilon}$ is a \emph{pseudomorphism} if each $\upsilon_h$ is invertible, and it is a \emph{strict morphism} if each $\upsilon_h$ is the identity. 
		
		Let $\braket{U, \upsilon}, \braket{U', \upsilon'} \colon {\braket{A,(-)^A, \Gamma, \Delta} \to\braket{A',(-)^{A'}, \Gamma', \Delta'}}$ be colax morphisms of colax algebras for $T$. A \emph{transformation} $\braket{ U, \upsilon}\implies \braket{U', \upsilon'}$ is a 2-cell $\alpha \colon U \implies U'$ such that, for each $h \colon J b \to A \in \Acat$, the square
		\[\begin{tikzcd}
			{U \kl h A } & {U' \kl h A} \\
			{\kl{(Uh)}{A'}} & {\kl{(U'h)}{A'}}
			\arrow["{\alpha*\kl h A }", from=1-1, to=1-2]
			\arrow["{\upsilon_h}"', from=1-1, to=2-1]
			\arrow["{\upsilon'_h}", from=1-2, to=2-2]
			\arrow["{\kl{(\alpha*h)}{A'}}"', from=2-1, to=2-2]
		\end{tikzcd}\]
		commutes. Transformations between lax morphisms are defined analogously.
	\end{definition}
	
	\begin{proposition}[\protect{\cite[\S 3.9-11]{arkor-saville-slattery}}]
		Colax algebras, their morphisms and transformations form 2-categories:
		\begin{enumerate}
			\item $\CoLaxAlg_{co}^J(T)$, with colax morphisms;
			\item $\CoLaxAlg_{lax}^J(T)$, with lax morphisms;
			\item $\CoLaxAlg^J_{ps}(T)$, with pseudomorphisms;
			\item $\CoLaxAlg_{str}^J(T)$, with strict morphisms.
		\end{enumerate} 
	\end{proposition}
	
	
	\begin{notation}\label{nota:omit_morphisms}
		For $\star \in \set{ co, lax, ps, str}$, we will write $\CoLaxAlg^J_\star(T)$ to refer to any of the above 2-categories.
	\end{notation}
	
	\begin{example}\label{ex:free_algebras}
		Let $J \colon  \Bcat\to \Acat$ be a 2-functor and let $\braket{T,\eta,(-)^\ast}$ be a $J$-relative 2-monad. For each object $b \in \Bcat$, the object $Tb\in \Acat$ admits a (strict) algebra structure determined by the extension functors $(-)^* \colon \Hom\Acat{Ja}{Tb} \to \Hom\Acat{Ta}{Tb}$. The assignment $b \mapsto \braket{Tb,(-)^\ast,\id,\id}$ extends to a 2-functor $T' \colon \Bcat \to \CoLaxAlg^J_\star(T)$ which acts as $T$ on the underlying objects, 1-cells and 2-cells of $\Bcat$.  That is to say, post-composing $T'$ with the evident forgetful 2-functor $F \colon \CoLaxAlg^J_\star(T) \to \Acat$, that sends a colax algebra $\braket{A,(-)^A,\Gamma,\Delta}$ to $A$, returns the 2-functor $T$.
		
		Moreover, the 2-functor $T'$ yields a \emph{$J$-relative colax 2-adjunction}
		\[\begin{tikzcd}
			& {\CoLaxAlg_{co}^J(T)} \\
			\Bcat && {\Acat\ .}
			\arrow[""{name=0, anchor=center, inner sep=0}, "F", from=1-2, to=2-3]
			\arrow[""{name=1, anchor=center, inner sep=0}, "{T'}", from=2-1, to=1-2]
			\arrow["J"', from=2-1, to=2-3]
			\arrow["\dashv"{anchor=center}, draw=none, from=1, to=0]
		\end{tikzcd}\]
		By this we mean that, for each $b \in \Bcat$ and each colax algebra $\braket{A,(-)^A,\Gamma,\Delta}$ for $T$, there is a suitably natural adjunction
		\[\begin{tikzcd}
			{\Acat(Jb,A)} && \CoLaxAlg_{co}^J(T)\big(\braket{Tb,(-)^\ast,\id,\id},\braket{A,(-)^A,\Gamma,\Delta}\big)
			\arrow[""{name=0, anchor=center, inner sep=0}, "{(-)^A}"', shift right=2, from=1-1, to=1-3]
			\arrow[""{name=1, anchor=center, inner sep=0}, "{F(-) \circ \eta_b}"', shift right=2, from=1-3, to=1-1]
			\arrow["\dashv"{anchor=center, rotate=-90}, draw=none, from=1, to=0]
		\end{tikzcd}.\]
		The counit of the adjunction, at a 1-cell $h \colon Jb \to A$, is given by $\Gamma_h \colon h^A \circ \eta_b \Rightarrow h$, while the unit, at a  colax morphism $\braket{U,\upsilon} \colon \braket{Tb,(-)^\ast,\id,\id} \to \braket{A,(-)^A,\Gamma,\Delta}$, is the transformation defined by the 2-cell
		\[
		\begin{tikzcd}
			{U = U\circ \eta^*_b} & {(U\circ \eta_b)^A \ .}
			\arrow["{{\upsilon_{\eta_b}}}", from=1-1, to=1-2]
		\end{tikzcd}
		\]
		In this sense, $\CoLaxAlg_{co}^J(T)$ is a \emph{colax-resolution} for the relative 2-monad $\braket{T,\eta,(-)^*}$.
	\end{example}
	
	\begin{example}\label{ex:Jff-and-monad-colax-algebras-embed}
		Let $T$ be a 2-monad on $\Bcat$ and $J \colon \Bcat \to \Acat$ a 2-fully faithful 2-functor, so that $JT$ is a $J$-relative 2-monad as in Example \ref{ex:Jff-and-monad-yields-relative}. Then, by the 2-fully faithfulness of $J$, we have evident 2-fully faithful embeddings $\CoLaxAlg^{\id}_\star ( T )\hookrightarrow \CoLaxAlg^J_\star ( J T )$ whose images are spanned by those algebras having their carrier object in the essential image of $J$.
	\end{example}


	\section{Weak ultracatagories}\label{sec:weak-ultra}
	Having introduced the colax algebras for a relative 2-monad, we describe our motivating example: \emph{ultracategories}. Let $J \colon \Set \hookrightarrow \CAT$ denote the the inclusion of sets into categories viewing a set as a discrete category, and let $\braket{\beta,\eta,(-)^\ast}$ denote the \emph{ultrafilter monad} on $\Set$, i.e.\ $\beta$ sends a set $X$ to the set of ultrafilters on $X$, presented as an extension system. Recall from Example \ref{ex:Jff-and-monad-yields-relative} that, since $J$ is 2-fully faithful, we obtain a $J$-relative 2-monad $\braket{J\beta, J\eta, (-)^\ast}$ which we dub the \emph{relative ultrafilter 2-monad}. 
	
	\begin{definition}\label{def:ultracategory}
		A \emph{weak ultracategory} is a colax algebra for the relative ultrafilter 2-monad. For $\star \in \set{ co, lax, ps, str}$, we denote with $\mathbf{WUlt}_{\star}$ the 2-category $\CoLaxAlg^J_\star(J\beta)$.
	\end{definition}
	
	We will now spell out what this means, leaving $J$ implicit. First, to provide some intuition, recall by Manes' theorem \cite[Proposition 5.5]{manes-comphaus} that a compact Hausdorff topology on a set $S$ can be specified by the structure of a $\beta$-algebra on $S$. Presenting $\beta$ as the extension system $\braket{\beta, \eta,(-)^*}$, this means that a compact Hausdorff topology on a set $S$ can be given by specifying, for each function $h \colon X \to S$, an {extension} $h^S \colon \beta X \to S$. Thinking of $h$ as an $X$-indexed sequence of points in $S$, $h^S$ can then be understood as mapping each ultrafilter $\nu \in \beta X$ to the unique `limit' of the sequence $h$ `with respect' to $\nu$. The description in terms of a $\beta$-algebra $\beta S \to S$ can be recovered by considering the extension of $\id \colon S \to S$.
		
	With these ideas in mind, a weak ultracategory is a category $M$ endowed with:
	\begin{enumerate}
		\item for each set $X$, a functor $(-)^{ M}\colon \funcat X M  \to \funcat {\beta X} M$, which maps an $X$-indexed sequence $h 
		\colon X \to M$ of objects in $M$ to the functor $h^M$ sending each ultrafilter $\nu\in\beta X$ to the limit of the sequence $h$ with respect to $\nu$;
		\item a family of natural transformations $\Gamma_h \colon h^M \circ \eta_X \implies h$ for every sequence $h \colon X \to M$ of objects in $M$, natural in $h$;
		\item a family of natural transformations $\Delta_{h,k} \colon h^M \circ k^* \implies (h^M\circ k)^M$ for every sequence $h \colon X \to M$ of objects in $M$ and every sequence $k \colon Y \to \beta X$ of ultrafilters on $X$, natural in $h$ and $k$,
	\end{enumerate}
	satisfying the equations
	\[(a) \begin{tikzcd}[sep = large]
		{\kl h M \eta_X^*} & {\kl {(\kl h M \eta_X)} M} \\
		& {\kl h M}
		\arrow["{{\Delta_{h,\eta_X}}}", from=1-1, to=1-2]
		\arrow[equals, from=1-1, to=2-2]
		\arrow["{{\kl {(\Gamma_h)}M}}", from=1-2, to=2-2]
	\end{tikzcd} \qquad (b) \begin{tikzcd}[sep = large]
		{\kl h M k^* \eta_Y} & {\kl { ( \kl h M k ) } M \eta_Y} \\
		& {\kl h M k}
		\arrow["{{\Delta_{h,k}*\eta_Y}}", from=1-1, to=1-2]
		\arrow[equals, from=1-1, to=2-2]
		\arrow["{{\Gamma_{\kl h M k}}}", from=1-2, to=2-2]
	\end{tikzcd}\]
	\[ (c) \quad \begin{tikzcd}[sep = large]
		{\kl h M (k^* l)^*} & {\kl h M k^* l^*} & {\kl {(\kl h M k)} M l^*} \\
		{\kl {(\kl h M k^*l)} M} && {\kl{( \kl{( \kl h M k)}M l)}M}
		\arrow[equals, from=1-1, to=1-2]
		\arrow["{{\Delta_{h,k^*l}}}"', from=1-1, to=2-1]
		\arrow["{{\Delta_{h,k}*l^*}}", from=1-2, to=1-3]
		\arrow["{{\Delta_{\kl h M k, l}}}", from=1-3, to=2-3]
		\arrow["{{\kl{(\Delta_{h,k}*l)}M}}"', from=2-1, to=2-3]
	\end{tikzcd}\]
	for every sequence $h \colon X \to M$ of objects in $M$, every sequence $k \colon Y \to \beta X $ of ultrafilters on $X$, and every sequence $l \colon Z \to \beta Y$ of ultrafilters on $Y$. 
	
	\begin{example}
		The category $\Set$ carries the archetypal structure of an ultracategory. For a function $h \colon X \to \Set$, i.e.\ a tuple of sets $\{M_x\}_{x \in X}$, the extension $h^\Set \colon \beta X \to \Set$ acts by sending an ultrafilter $\nu\in \beta X$ to the \emph{ultraproduct} $\prod_\nu M_\bullet \coloneqq \colim_{U \in \op\nu}\left( \prod_{x \in U}M_x \right)$.
	\end{example}
	
	\begin{remark}\label{rem:algebra-map}
		In the case of a compact Hausdorff space $S$, the $\beta$-algebra map $\beta S \to S$ witnessing the topology can be recovered as the extension of $\id \colon S \to S$. In the case of a weak ultracategory $M$, instead, we cannot obtain an `algebra map' simply by ``extending the identity''. Our strategy to turn $M$ into an algebra for a (pseudo)monad on $\CAT$ will be to define a category $\tilde \beta M$ whose objects are triples $(X \in \Set, h \colon X \to M, \nu \in \beta X)$. The algebra functor $\tilde \beta M \to M$ will then map such a triple to the limit $h^M(\nu)\in M$ of the sequence $h$ with respect to $\nu$. 
		
		In particular, in the case of a (weak) ultracategory of the form $\Mod {\mathbb{T}}$, the category of models for a first-order theory $\mathbb{T}$, this means that the algebra functor sends a triple $(X \in \Set, h \colon X \to \Mod {\mathbb{T}}, \nu \in \beta X)$, i.e.\ a sequence of models $\set{ M_x }_{x\in X}$ together with an ultrafilter $\nu \in \beta X$, to their ultraproduct with respect to $\nu$. In this sense, we can see the category $\tilde \beta (\Mod {\mathbb{T}})$ as the category of \emph{formal ultraproducts} of sequences of models of $\mathbb{T}$ (cf.\ \cite{rosolini-ultracompletion}).
	\end{remark}
	
	Lurie's ultracategories can be immediately framed within weak ultracategories. It is evident that -- up to currying functors and natural transformations -- every ultracategory \emph{à la} Lurie \cite[Definition 1.3.1]{lurie} is endowed with the same structure as that of a colax algebra for $\braket{J\beta, J \eta, (-)^*}$. The only difference is that our axiomatisation is weaker.
	
	\begin{definition}[\protect{\cite[Definition 1.3.1]{lurie}}]\label{def:luries-axioms}
		A \emph{Lurie-ultracategory} is a weak ultracategory $\braket{M, (-)^M, \Gamma,\Delta}$ such that:
		\begin{enumerate}
			\item each natural transformation $\Gamma$ is an isomorphism;
			\item for each sequence $h \colon X \to M$ of objects in $M$ and each injection $i \colon Y \hookrightarrow X$, the natural transformation $\Delta_{h, \eta_X  i}$ is an isomorphism.
		\end{enumerate}
	\end{definition}
	
	\begin{remark}
		In detail, our axioms (b) and (c) above precisely correspond to axioms (A) and (C) of \cite[Definition 1.3.1]{lurie}, while axiom (a) is recovered making use of the two extra assumptions in \cite[Corollary 1.3.6]{lurie}.
	\end{remark}
	
	Although our axiomatisation of ultracategories is strictly weaker than Lurie's (see Example \ref{ex:free-ultracat-on-1}), the \emph{left/right ultrafunctors}, \emph{ultrafunctors}, and \emph{transformations} from \cite[Definitions 1.4.1-2]{lurie} coincide exactly with the \emph{colax/lax morphisms}, \emph{pseudomorphisms}, and \emph{transformations} from Definition \ref{def:colax-algebra}. In other words, Lurie's 2-categories $\mathbf{ Ult }$ and $\mathbf{ Ult}^{\text L}$ 2-fully-faithfully embed into $\mathbf{WUlt}_{ps}$ and $\mathbf{WUlt}_{co}$ respectively. This immediately entails that the logical properties of Lurie's ultracategories carry over to weak ultracategories, which thus satisfy the desiderata for an axiomatisation of the notion of ultracategories:
	\begin{enumerate}
		\item for a small pretopos $P$, its category of models $\Mod P \coloneqq \Hom{\mathbf{Pretop}}P \Set$ is a weak ultracategory, and the evaluation functor 
		\[ \mathsf{ev} \colon P \to \Hom{\mathbf{WUlt}_{ps}}{\Mod P}\Set \subseteq \Hom{\mathbf{WUlt}_{co}}{\Mod P}\Set \]
		induces an equivalence between $\Hom{\mathbf{WUlt}_{co}}{\Mod P}\Set$ and the topos $\mathsf{Sh}(P)$ as in \cite[Theorem 2.2.2]{lurie};
		
		\item moreover, the functor	$\mathsf{ev}$ is an equivalence $P \simeq \Hom{\mathbf{WUlt}_{ps}}{\Mod P}\Set$ itself as in \cite[Theorem 4.1]{makkai} and \cite[Theorem 2.3.1]{lurie}, which implies in particular that for each pair of pretoposes $P, P'$ there is an equivalence of categories 
		\[ \Hom{\mathbf{Pretop}} P {P'} \simeq \Hom{\mathbf{WUlt}_{ps}}{\Mod {P'}}{\Mod P}\]
		as in \cite[Corollary 8.3]{makkai} and \cite[Corollary 2.3.3]{lurie}.
	\end{enumerate}	
	As we will see, the main advantage of our weakening of Lurie's axioms will be that weak ultracategories coincide with algebras for a universally obtained pseudomonad on $\CAT$ (Corollary \ref{cor:ultracategories-universally}; cf.\ also the discussion in \cite{ivan-ultra}). 
	
	\begin{remark}
		As already for Lurie's ultracategories, it is not entirely clear whether Makkai's ultracategories are weak ultracategories, and moreover if the 2-category $\mathbf{UC}$ of Makkai's ultracategories embeds into $\mathbf{WUlt}_{ps}$. However, denoting with $\mathbf{WUlt}_0, \mathbf{UC}_0$ and $\mathbf{Ult}_0$ respectively the full 2-subcategories of $\mathbf{WUlt}_{ps}, \mathbf{UC}$ and $\mathbf{Ult}$ spanned by those ultracategories of the form $\Mod P$ for a small pretopos $P$, it follows that 
		\[ \mathbf{WUlt}_0 \simeq \mathbf{UC}_0 \simeq \mathbf{Ult}_0 \]
		as they are all equivalent to $\op{\mathbf{Pretop}}$ via the 2-functor $\Mod {-}$.
	\end{remark}
	
	\begin{remark}[Ultrasets]\label{rem:ultrasets}
		Call a weak ultracategory $\braket{M, (-)^M, \Gamma,\Delta}$ an \emph{ultraset} if the carrier category $M$ is actually a set. Then since $\Set$ is discrete on 2-cells, all ultrasets are strict algebras: in particular they are Lurie-ultracategories, and the notion coincides with that of \cite[Definition 3.1.1]{lurie}. Moreover, all morphisms of ultrasets are strict, and the only transformations between them are identities, so we can unambiguously define the category $\mathbf{UltSet}$ of ultrasets. By Example \ref{ex:Jff-and-monad-colax-algebras-embed}, it follows that $\mathbf{UltSet}$ is isomorphic to the category of $\beta$-algebras. Applying Manes' theorem \cite[Proposition 5.5]{manes-comphaus}, we deduce that $\mathbf{UltSet}$ is isomorphic to the category of compact Hausdorff spaces, thus immediately recovering \cite[Theorem 3.1.5]{lurie}.
	\end{remark}
	
	\begin{example}\label{ex:free-ultracat-on-1}
		Before moving on, we describe a weak ultracategory that does not satisfy either of the stronger conditions of Definition \ref{def:luries-axioms}.  Consider the \emph{category of elements} for the ultrafilter functor $\beta \colon \Set \to \Set$, i.e.\ the category whose objects are pairs $(X,\nu)$ consisting of a set $X$ and an ultrafilter $\nu \in \beta X$, while an arrow $f \colon (X,\nu) \to (X',\nu')$ consists of a function $f \colon X' \to X$ such that $\beta(f)(\nu') = \nu$. Note that this coincides with the category $\mathcal{U}\mathcal{E}$ of ultrafilters defined in \cite[Definition 4]{garner-ultra}. We denote this category by $\tilde \beta (1)$: indeed, it follows from the results of Section \ref{sec:ultracompletion} that $\tilde \beta(1)$ is the free weak ultracategory on the trivial category.
		
		Explicitly, $\tilde \beta (1)$ is endowed with the following ultrastructure.
		\begin{enumerate}
			\item The extension functor is constructed as follows. 
			
			Let $h \colon Y \to \tilde \beta(1)$ be a function. For each $y\in Y$ let $(Ry , \nu_y)$ be the image $h(y)$, and let $\coprod R$ denote the coproduct of the family $\set{ R y | y \in Y }$, with coprojections $c_y \colon Ry \hookrightarrow \coprod R$ which for simplicity we identify with subset inclusions. Consider then the function:
			\[ \textstyle q \colon Y \to \beta\left( \coprod R \right), \qquad y \mapsto \beta(c_y)(\nu_y)\] 
			which, using the monad structure for $\beta$, lifts to a function $q^* \colon \beta Y \to  \beta\left( \coprod R \right)$. Explicitly, for an ultrafilter $\nu\in \beta Y$, $q^*\nu$ is the ultrafilter on $\coprod R $ characterised by
			\[ S \in q^*\nu \iff \set{ y \in Y | S\cap Ry \in \nu_{y} } \in \nu .\]
			Thus, we can define the algebra extension of $h$ as the function:
			\[ \textstyle h^{\tilde \beta (1)} \colon \beta Y  \to \tilde \beta(1) , \qquad \nu \mapsto \left(\coprod R , q^\ast \nu\right).\]
			
			Given two functions $h,h' \colon Y \to \tilde \beta (1)$, let $\sigma \colon h \implies h'$ be a natural transformation. For each $y \in Y$, the component $\sigma_y \colon (Ry, \nu_y) \to (R'y, \nu'_y)$ is in particular a function $\sigma_y \colon R'y \to Ry$, so there is an induced function $\coprod \sigma \colon \coprod R' \to \coprod R$. Thus, we can define the algebra extension $\sigma^{\tilde \beta(1)} \colon h^{\tilde \beta(1)} \implies h'^{\tilde \beta(1)}$ of $\sigma$ as the natural transformation whose component at $\nu\in\beta Y$ is the arrow
			\[  \textstyle \coprod \sigma \colon \left( \coprod R, q^*\nu\right) \to \left(\coprod R', q'^*\nu\right) \]
			in $\tilde \beta(1)$.
			
			\item For a function $h \colon Y \to \tilde \beta(1)$, the natural transformation $\Gamma_h \colon h^{\tilde \beta (1)} \circ \eta_Y \Rightarrow h$ is defined at $y \in Y$ by the arrow
			\[
			\textstyle c_y \colon h^{\tilde \beta (1)} \circ \eta_Y(y) = (\coprod R, q^\ast \eta_Y(y) ) \to (Ry,\nu_y) = h(y)
			\]
			Indeed, $q^*\eta_Y(y)$ is the ultrafilter on $\coprod R$ defined by:
			\[ S \in q^*\eta_Y(y) \iff \set{ y' \in Y | S \cap R y' \in \nu_{y'} } \in \eta_Y (y) \iff S \cap R y \in \nu_y , \] 
			meaning that $q^*\eta_Y(y) = \beta(c_y)(\nu_y)$. Note that this arrow is not invertible, and hence $\Gamma_h$ is not an isomorphism, whenever there is more than one non-empty $Ry$.
			
			\item For a pair of functions $h \colon Y \to \tilde \beta (1)$ and $k \colon Z \to \beta Y$, the natural transformation $\Delta_{h,k} \colon h^{\tilde \beta (1)} \circ k^* \implies (h^{\tilde \beta (1)} \circ k )^{\tilde \beta (1)}$ is defined as follows.
			
			Let $\nu \in \beta Z$. By definition, $h^{\tilde \beta (1)}(k^* \nu)$ is the pair $(\coprod R, q^* k^*\nu )$. On the other hand note that, for each $z\in Z$, $h^{\tilde \beta (1)} (k(z)) = (\coprod R, q^*k(z))$, where for $S \subseteq \coprod R$:
			\[ S \in q^*k(z)\iff \set{ y \in Y | S \cap Ry \in \nu_y } \in k(z)\]
			Therefore, by definition:
			\[ \textstyle (h^{\tilde \beta (1)} \circ k)^{\tilde \beta (1)} ( \nu ) = ( \coprod_{z\in Z} (\coprod R) , \bar q ^* \nu ) \]
			where, for convenience of notation, we define:
			\[ \textstyle \coprod R \coloneqq \set{ (y,x) | y \in Y, x \in R y }\]
			\[\textstyle \coprod_{z\in Z}(\coprod R) \coloneqq \set{ (z,y,x) | z \in Z, y \in Y, x \in R y} \]
			so that, for $T \subseteq \coprod_{z\in Z}(\coprod R)$, we can write:
			\begin{align*}
				T \in \bar q ^* \nu & \iff \set{ z \in Z | \set{ (y,x) | y \in Y , x \in R y : (z,y,x) \in T} \in q^*k(z) } \in \nu 
			\end{align*}
			With this notation, let $\pi \colon \coprod_{z\in Z}(\coprod R) \to \coprod R$ be the obvious projection, and note that $q^*k^*\nu = \beta(\pi)(\bar q ^* \nu )$. The component of $\Delta_{h,k}$ at $\nu \in \beta Z$ is given by the arrow
			\[ \textstyle \pi \colon h^{\tilde \beta(1)}(k^*\nu) = (\coprod R, q^* k^* \nu) \to (\coprod_{z\in Z}(\coprod R), \bar q^*\nu ) = (h^{\tilde \beta (1)} \circ k)^{\tilde \beta (1)} ( \nu ) .\]
			Note that this arrow need not be invertible, and hence $\Delta_{h,k}$ need not be an isomorphism, even if $k = \eta_Y \circ i$ for some injective function $i \colon Z \hookrightarrow Y$. For a concrete counterexample, take $Y$ having at least three distinct elements, $h$ such that each $Ry$ is finite, and $Z \subseteq Y$ consisting of two distinct elements.
		\end{enumerate}
	\end{example}
	

	\section{Oplax Kan extensions}\label{sec:kan-ext}
	
	The 2-endofunctor underlying the `unrelativisation' of a relative 2-monad in $\CAT$, which we construct in the next section, will be a \emph{left oplax Kan extension} -- a generalisation of the notion of a \emph{$\CAT$-enriched left Kan extension} in the sense of \cite[\S 4.20]{kelly2005}.  In this section, we describe the universal property defining left oplax Kan extensions and give an explicit construction for $\CAT$-valued 2-functors.
	
	\subsection{Oplax Kan extensions}We begin by recalling the definition of an oplax transformation between 2-functors.
	
	\begin{definition}[\protect{\cite[Definitions 1]{street_lax,2dimensionalcategories}}]\label{def:oplax-transformation}
		Let $F, G \colon \Bcat \to \Acat$ be a pair of 2-functors.  An \emph{oplax transformation} $\beta \colon F \Rightarrow G$ consists of the data of an arrow $\beta_b \colon Fb \to Gb$ for each $b \in \Bcat$, and a 2-cell $\beta_f \colon \beta_b \circ F(f) \Rightarrow G(f) \circ \beta_{b'}$ for each arrow $f \colon b' \to b$ in $\Bcat$.  Moreover, these data must satisfy the coherence and naturality conditions:
		\begin{enumerate}
			\item the 2-cells $\beta_f$, for $f \colon b' \to b \in \Bcat$, define a natural transformation 
			\[\begin{tikzcd}
				{\Bcat(b',b)} && {\Acat(Fb',Gb);}
				\arrow[""{name=0, anchor=center, inner sep=0}, "{\beta_{b} \ast F(-)}", curve={height=-20pt}, from=1-1, to=1-3]
				\arrow[""{name=1, anchor=center, inner sep=0}, "{G(-)\ast\beta_{b'}}"', curve={height=20pt}, from=1-1, to=1-3]
				\arrow["{\beta_{(-)}}", shorten <=5pt, shorten >=5pt, Rightarrow, from=0, to=1]
			\end{tikzcd}\]
			\item $\beta_{\id_b} $ is the identity 2-cell on $\beta_b$;
			\item for each pair $f \colon b'\to b, g\colon b''\to b'$, the pasting on the left coincides with the natural transformation on the right:
			\[ \begin{tikzcd}[row sep = 2.25em]
				{Fb''} & {Fb'} & Fb \\
				{Gb''} & {Gb'} & {Gb ,}
				\arrow["Fg", from=1-1, to=1-2]
				\arrow["{{\beta_{b''}}}"', from=1-1, to=2-1]
				\arrow["Ff", from=1-2, to=1-3]
				\arrow["{{\beta_g}}"', Rightarrow, from=1-2, to=2-1]
				\arrow["{{\beta_{b'}}}"{description}, from=1-2, to=2-2]
				\arrow["{{\beta_f}}"', Rightarrow, from=1-3, to=2-2]
				\arrow["{{\beta_b}}", from=1-3, to=2-3]
				\arrow["Gg"', from=2-1, to=2-2]
				\arrow["Gf"', from=2-2, to=2-3]
			\end{tikzcd} \qquad 
			\begin{tikzcd}[row sep = 2.25em]
				{Fb''} & {Fb'} & Fb \\
				{Gb''} & {Gb'} & {Gb .}
				\arrow["Fg", from=1-1, to=1-2]
				\arrow["{{\beta_{b''}}}"', from=1-1, to=2-1]
				\arrow["Ff", from=1-2, to=1-3]
				\arrow["{{\beta_{f\circ g}}}"', shorten <=5pt, shorten >=5pt, Rightarrow, from=1-3, to=2-1]
				\arrow["{{\beta_{b}}}", from=1-3, to=2-3]
				\arrow["Gg"', from=2-1, to=2-2]
				\arrow["Gf"', from=2-2, to=2-3]
			\end{tikzcd} 
			\]
		\end{enumerate}
		
		A \emph{lax transformation} $\beta \colon F \implies G$ is defined in the same way but reversing the direction of each 2-cell $\beta_{f}$. A \emph{strict transformation} $\beta \colon F \implies G$ is then one where each $\beta_{f}$ is the identity, i.e.\ simply a 2-natural transformation $F \implies G$.
		
		We write ${\mathbf{Oplax}}\funcat{\Bcat}{\Acat}$ for the 2-category of 2-functors $\Bcat \to \Acat$, oplax transformations, and \emph{modifications} (see e.g.\ \cite{leinster-bicategories} for the definition); similarly we define ${\mathbf{Lax}}\funcat{\Bcat}{\Acat}$ and ${\mathbf{Str}}\funcat{\Bcat}{\Acat}$.
	\end{definition}
	
	\begin{example}\label{ex:extension-operators-define-lax-trans}
		Let $J \colon \Bcat\to\Acat$ be a 2-functor, $\braket{T, \eta, (-)^*}$ be a $J$-relative 2-monad, and $\braket{A, (-)^A, \Gamma,\Delta}$ be a colax algebra for $\braket{T, \eta,(-)^*}$. Then, the family of extension operators $\set{ (-)_b^A \colon \Hom\Acat{ J b }A \to \Hom\Acat{Tb}A | b\in\Bcat}$ determines a lax transformation
		\[{(-)^A \colon \Hom\Acat{ J - }A \implies \Hom\Acat{T-}A }\]
		where, for $f \colon b' \to b \in \Bcat$, the natural transformation 
		\[ \begin{tikzcd}[sep = 2.25em]
			{\Hom\Acat{Jb}A} & {\Hom\Acat{Tb}A} \\
			{\Hom\Acat{Jb'}A} & {\Hom\Acat{Jb'}A}
			\arrow["{(-)^A_b}", from=1-1, to=1-2]
			\arrow["{-\circ Jf}"', from=1-1, to=2-1]
			\arrow["{(-)^A_f}"', shorten <=4pt, shorten >=4pt, Rightarrow, from=1-2, to=2-1]
			\arrow["{-\circ T f}", from=1-2, to=2-2]
			\arrow["{(-)^A_{b'}}"', from=2-1, to=2-2]
		\end{tikzcd} \]
		is defined, for each $h \colon J b \to A$, the composite
		\[ \begin{tikzcd}[sep = large]
			{h^A \circ T f = h^A\circ ( \eta_b \circ J f)^*} & {(h^A\circ \eta_b\circ J f)^A} & {(h\circ Jf)^A}
			\arrow["{\Delta_{h,\eta_bJf}}", from=1-1, to=1-2]
			\arrow["{(\Gamma_h*Jf)^A}", from=1-2, to=1-3]
		\end{tikzcd} .\]
		Indeed, the algebra axioms ensure that the requirements of Definition \ref{def:oplax-transformation} are met (cf.\ \cite[Lemma 3.22]{arkor-saville-slattery}). 
	\end{example}
	
	Now let $J, T \colon \Bcat\to\Acat$ be a pair of 2-functors. To motivate the appearance of oplax Kan extensions in the context of 'unrelativising' algebras, consider the following situation.  Generalising the notion of an algebra for an endofunctor, and in the light of Example \ref{ex:extension-operators-define-lax-trans}, we can define a \emph{($J$-relative) algebra} for $T$ as an object $a \in \Acat$ together with a lax transformation $(-)^a \colon \Hom\Acat{ J - }a \implies \Hom\Acat{T-}a $. If a 2-functor $\tilde T \colon \Acat \to \Acat$ is to have the same algebras as $T$, then the datum of such a lax transformation must correspond to that of a 1-cell $\tilde T a \to a$ in $\Acat$. Thus, we arrive at the following definition, which lifts that of an enriched left Kan extension from \cite[\S 4.20]{kelly2005} to the lax setting.
	
	\begin{definition}\label{def:left-oplax-kan-extension}
		Let $J, T \colon \Bcat \to \Acat$ be a pair of 2-functors. A 2-functor $\tilde T \colon \Acat \to \Acat$ is a \emph{left oplax Kan extension} of $T$ along $J$, if, for any pair $a,a'\in \Acat$, there is an isomorphism of categories
		\[ \Hom{\Acat}{\tilde T a }{a'} \iso \Hom{\mathbf{Lax}\funcat{\op\Bcat}{\CAT}}{\Hom{\Acat}{J-}{a}}{\Hom\Acat{T-}{a'}} \] 
		which is natural in $a$ and $a'$.
	\end{definition}
	
	At first glance, the above definition does not seem reminiscent of the familiar universal property of a left Kan extension from 1-dimensional category theory (see \cite[\S X.3]{maclane}).  To recover an analogue of this property, and to justify the name `\emph{oplax} Kan extension' even though it is \emph{lax} transformations which appear in Definition \ref{def:left-oplax-kan-extension}, we need to speak of \emph{lax ends}. The notion of a lax end was introduced in \cite{bozapalides77, bozapalides80}, generalising ends to the lax setting. The theory of lax ends is then furthered in \cite{hirata2022noteslaxends} by developing a lax end/coend calculus parallel to the ordinary end/coend one (see e.g.\ \cite{loregian2021}). We refer the reader to Appendix \ref{app:oplax-coends} for the relevant definitions.
	
	In terms of lax ends, by \cite[\S 3.1]{hirata2022noteslaxends}, the left oplax Kan extension of $T \colon \Bcat\to\Acat$ along $J\colon\Bcat\to\Acat$ is defined, for each pair $a,a'\in \Acat$, by a natural isomorphism of categories
	\[ \Hom{\Acat}{\tilde T a }{a'} \iso \oplend_{b\in\Bcat} \funcat{\Hom\Acat{Jb}a}{\Hom\Acat{Tb}{a'}}   \]
	where the right-hand term is the oplax end of the 2-functor
	\[ \op\Bcat \times \Bcat \to \CAT\qquad (b_0,b_1) \mapsto \funcat{\Hom\Acat{Jb_1}a}{\Hom\Acat{Tb_0}{a'}}.\]
	The lax end calculus of \cite{hirata2022noteslaxends} then allows us to deduce the following property of left oplax Kan extensions, involving \emph{oplax} transformations, which generalises the universal property of a 1-dimensional Kan extension.  
	
	\begin{proposition}\label{prop:weak-left-oplax-kan-extension}
		If $\tilde T \colon \Acat\to\Acat$ is a left oplax Kan extension of $T \colon \Bcat\to \Acat$ along $J \colon \Bcat\to \Acat$, then for each 2-functor $S\colon \Acat\to \Acat$ there is an isomorphism of categories:
		\[  \mathbf{Str}\funcat{\Acat}{\Acat} ( \tilde T , S ) \iso \mathbf{Oplax}\funcat{\Bcat}{\Acat} \left( T , S \circ J \right)  \]
		which is natural in $S$. 
		\begin{proof}
			A proof of this is given in Proposition \ref{prop:app.weak-left-oplax-kan-extension}.
		\end{proof}
	\end{proposition}
	
	\begin{corollary}\label{cor:oplax-trans-uniquely-factor-through-zeta}
		If $\tilde T \colon \Acat\to\Acat$ is a left oplax Kan extension of $T \colon \Bcat\to \Acat$ along $J \colon \Bcat\to \Acat$, then there is an oplax transformation $\zeta \colon  T \implies \tilde T \circ J $ such that, for every other 2-functor $S \colon \Acat \to \Acat$ together with an oplax transformation $\epsilon \colon T \implies S \circ J$, there exists a unique strict transformation $\u\epsilon \colon \tilde T \implies S$ such that $\epsilon = (\u \epsilon * J ) \circ \zeta$, i.e.\ such that the natural transformation on the left coincides with the pasting on the right:
		\[  \begin{tikzcd}[column sep = 20pt]
			\Bcat &&& \Acat \\
			\\
			\Acat
			\arrow["T", from=1-1, to=1-4]
			\arrow["J"', from=1-1, to=3-1]
			\arrow[""{name=0, anchor=center, inner sep=0}, "S"', curve={height=18pt}, from=3-1, to=1-4]
			\arrow["\epsilon"', shorten <=11pt, shorten >=11pt, Rightarrow, from=1-1, to=0]
		\end{tikzcd} \qquad  \begin{tikzcd}[column sep = 20pt]
			\Bcat &&& \Acat \\
			\\
			\Acat
			\arrow["T", from=1-1, to=1-4]
			\arrow["J"', from=1-1, to=3-1]
			\arrow["{{{\tilde T }}}"{description}, curve={height=-18pt}, from=3-1, to=1-4]
			\arrow[""{name=0, anchor=center, inner sep=0}, "S"', curve={height=18pt}, from=3-1, to=1-4]
			\arrow[""{name=1, anchor=center, inner sep=0}, "{{{\u\epsilon}}}"'{pos=0.7}, shorten <=28pt, shorten >=8pt, Rightarrow, from=1-1, to=0]
			\arrow["\zeta"'{pos=0.4}, shorten <=3pt, shorten >=8pt, Rightarrow, from=1-1, to=1]
		\end{tikzcd}  \] 
		\begin{proof}
			Take $\zeta$ to be the unique oplax transformation corresponding to the strict transformation $\id \colon \tilde T \implies \tilde T$.
		\end{proof}
	\end{corollary}
	
	\begin{remark}
		The property expressed by Corollary \ref{cor:oplax-trans-uniquely-factor-through-zeta} makes it evident that left oplax Kan extensions are unique up to isomorphism whenever they exist.
	\end{remark}
	
	\subsection{ \texorpdfstring{$\CAT$-valued left oplax Kan extensions}{\textbf{CAT}-valued left oplax Kan extensions}}\label{ssec:cat-valued-left-oplax-kan-extensions}In the case of $\CAT$-valued 2-functors, left oplax Kan extensions can be described explicitly. To provide intuition on how we will show this, recall from \cite[\S X]{maclane} that the (pointwise) left Kan extension $L \colon \Set \to \Set$ of a functor $T \colon D \to \Set$ along a functor $J \colon D \to \Set$ exists if and only if, for each set $X$, the coend $\int^{d \in D} \Hom{\Set}{Jd}{X}\times Td$ is a set. In that case, this coend is the value of $L$ at $X$, while the action of $L$ on arrows is uniquely determined by the universal property of coends.\footnote{For a more general argument in the enriched case, see \cite[\S 4.25]{kelly2005}.} An analogous relation holds between left oplax Kan extensions and oplax coends in $\CAT$: the left oplax Kan extension $\tilde T \colon \CAT \to \CAT$ of a 2-functor $T \colon \Bcat \to \CAT$ along a 2-functor $J \colon \Bcat \to \CAT$ is uniquely determined by setting $\tilde T C$, for a category $C$, to be the oplax coend $\oplend^{b\in\Bcat} \funcat{J b}{C} \times T b$ (Proposition \ref{prop:left-oplax-kan-ext-determined-by-oplax-coends}). In this section we will compute this oplax coend explicitly, obtaining a description of the left oplax Kan extensions of $T$ along $J$.
	
	\begin{definition}\label{def:Ttilde-of-C}
		Fix a pair of 2-functors $J, T  \colon \Bcat \to \CAT$.  For a category $C$, we let $\tilde T C$ be the following category.
		\begin{enumerate}
			\item An object of $\tilde T C$ is a triple $(b, h\colon Jb \to C, \nu)$ consisting of an object $b \in \Bcat$, a functor $h \colon J b \to C$, and an object $\nu \in T b$.
			
			\item An arrow $(b , h \colon J b\to C , \nu ) \to (b', h'\colon J b' \to C , \nu')$ in $\tilde T C$ is an equivalence class of triples $(f,\alpha,\phi)$ consisting of a 1-cell $f \colon b' \to b$ in $\Bcat$, a natural transformation $\alpha \colon h \circ Jf \implies h'$, and an arrow $\phi \colon \nu  \to Tf (\nu')$ in the category $Tb$, with obvious identities and compositions.
			
			We identify two such triples $(f, \alpha_f,\phi_f) , (g, \alpha_g , \phi_g) \colon ( b,h,\nu) \to (b',h',\nu')$ up to the equivalence relation generated by letting  $(f, \alpha_f,\phi_f) \sim (g, \alpha_g , \phi_g)$ if there exists a 2-cell $\sigma \colon f \implies g$ in $\Bcat$ such that 
			\begin{enumerate}
				\item $\alpha_g \circ (h * J \sigma) = \alpha_f$, i.e.\ the pasting on the left coincides with the natural transformation on the right:
				\[ \begin{tikzcd}[row sep = 28pt]
					{Jb'} && {Jb,} \\
					& C
					\arrow[""{name=0, anchor=center, inner sep=0}, "Jf", curve={height=-18pt}, from=1-1, to=1-3]
					\arrow[""{name=1, anchor=center, inner sep=0}, "Jg"{description}, curve={height=18pt}, from=1-1, to=1-3]
					\arrow["{h'}"', curve={height=6pt}, from=1-1, to=2-2]
					\arrow["h", curve={height=-6pt}, from=1-3, to=2-2]
					\arrow["{J\sigma}"', shorten <=5pt, shorten >=5pt, Rightarrow, from=0, to=1]
					\arrow["{\alpha_g}"'{pos=0.55}, shorten <=6pt, Rightarrow, from=1, to=2-2]
				\end{tikzcd} \qquad \begin{tikzcd}[row sep = 28pt]
					{Jb'} && {Jb;} \\
					& C
					\arrow[""{name=0, anchor=center, inner sep=0}, "Jf", curve={height=-18pt}, from=1-1, to=1-3]
					\arrow["{h'}"', curve={height=6pt}, from=1-1, to=2-2]
					\arrow["h", curve={height=-6pt}, from=1-3, to=2-2]
					\arrow["{\alpha_f}"', shorten <=5pt, Rightarrow, from=0, to=2-2]
				\end{tikzcd} \]
				\item and $\phi_g = (T\sigma)_{\nu'} \circ \phi_f$, i.e.\ the triangle 
				\[ \begin{tikzcd}[row sep = small]
					&& {Tf(\nu')} \\
					\nu \\
					&& {Tg(\nu')}
					\arrow["{(T\sigma)_{\nu'}}", from=1-3, to=3-3]
					\arrow["{\phi_f}", curve={height=-6pt}, from=2-1, to=1-3]
					\arrow["{\phi_g}"', curve={height=6pt}, from=2-1, to=3-3]
				\end{tikzcd} \] 
				commutes.
			\end{enumerate}
		\end{enumerate} 
		This assignment on categories extends to a 2-functor $\tilde T \colon \CAT \to \CAT$.  For a functor $F \colon C \to D$, we set $\tilde T F \colon \tilde T C \to \tilde T D$ to be the functor:
		\[\begin{tikzcd}
			{(b, h \colon J b\to C, \nu)} & {(b, F\circ h \colon J b\to D, \nu)} \\
			{(b', h'\colon J b'\to C, \nu')} & {(b', F\circ h' \colon J b'\to C, \nu')};
			\arrow[maps to, from=1-1, to=1-2]
			\arrow["{(f,\alpha,\phi)}"', from=1-1, to=2-1]
			\arrow["{(f, F*\alpha, \phi)}", from=1-2, to=2-2]
			\arrow[maps to, from=2-1, to=2-2]
		\end{tikzcd}\]
		and for a natural transformation $\gamma \colon F \Rightarrow G$, we set $\tilde T \gamma \colon \tilde T F \Rightarrow \tilde T G$ to be the natural transformation whose component at $(b,h \colon Jb \to C,\nu) \in \tilde T C$ is the arrow
		\[
		(\id_b, \gamma \ast h , \id_\nu) \colon (b,F  h \colon J b \to D, \nu ) \to (b,G h\colon J b \to D, \nu).
		\]
	\end{definition}
	
	\begin{notation}
		We will systematically abuse notation by identifying an equivalence class with one of its representatives where possible.
	\end{notation}
	
	\begin{remark}
		In our main example of weak ultracategories, $\Bcat$ will be the category $\Set$ seen as a locally discrete 2-category. In this case, the datum of an arrow $(b , h \colon J b \to C, \nu) \to (b', h'\colon J b' \to C, \nu')$ in $\tilde T C$ reduces to a pair of a function $f \colon b' \to b$ such that $\nu = T f (\nu')$ and a natural transformation $\alpha \colon h \circ J f \implies h'$. In particular, the equivalence relation in Definition \ref{def:left-oplax-kan-extension} reduces to the identity.
	\end{remark}
	
	\begin{proposition}\label{prop:left-oplax-Kan-extension-in-CAT}
		The 2-functor  $\tilde T \colon \CAT \to \CAT$ of Definition \ref{def:Ttilde-of-C} is the left oplax Kan extension of $T$ along $J$.
		\begin{proof}
			For a category $C$, a proof of the fact that the category $\tilde T C$ is the oplax coend of the 2-functor $\op\Bcat \times \Bcat \to \CAT$ defined by $(b,b')\mapsto \funcat{Jb}C \times T b$ is in Proposition \ref{prop:oplax-coend-in-CAT}. 	The 2-functor $\tilde T \colon \CAT \to \CAT$ described in Definition \ref{def:Ttilde-of-C} is then precisely the unique extension of the assignment $C \mapsto \tilde T C$  induced by the universal properties of oplax coends, and hence it is the left oplax Kan extension of $T$ along $J$ (see Proposition \ref{prop:left-oplax-kan-ext-determined-by-oplax-coends}).
		\end{proof}
	\end{proposition}
	
	\begin{remark}\label{rem:oplax-trans-correspond-to-strict-trans}
		By Proposition \ref{prop:weak-left-oplax-kan-extension}, $\tilde T$ is equipped with a universal oplax transformation $\zeta \colon T \implies \tilde T \circ J$, which is defined:
		\begin{enumerate}
			\item for an object $b \in \Bcat$, by the functor $\zeta_b \colon T b \to \tilde T Jb$ defined as
			\[\begin{tikzcd}
				\nu & {(b, \id_{Jb}\colon J b \to J b , \nu)} \\
				\nu' & {(b, \id_{Jb}\colon J b \to J b , \nu');}
				\arrow[maps to, from=1-1, to=1-2]
				\arrow["\phi"', from=1-1, to=2-1]
				\arrow["{(\id_b, \id, \phi)}", from=1-2, to=2-2]
				\arrow[maps to, from=2-1, to=2-2]
			\end{tikzcd}\]
			\item for a 1-cell $f \colon b' \to b \in \Bcat$, by the natural transformation $\zeta_f \colon \zeta_{b} \circ T f \implies \tilde T J f \circ \zeta_{b'}$ which, in turn, is defined at an object $\nu'\in T b'$ by the arrow
			\[ (f, \id, \id) \colon (b, \id_{Jb}\colon J b \to J b , Tf(\nu')) \to ( b', Jf \colon J b' \to Jb, \nu')\]
			in $\tilde T J b$.
		\end{enumerate}
		
		In particular, every oplax transformation $\epsilon \colon T \implies S \circ J$, for a 2-functor $S\colon \CAT \to \CAT$ factors uniquely through $\zeta$ via a strict transformation $\u\epsilon \colon \tilde T \implies S$, which we can describe explicitly. For each category $C$, its component $\u\epsilon_C \colon \tilde T C \implies S C$ is the functor defined:
		\begin{enumerate}
			\item on objects, by mapping $(b, h \colon Jb \to C , \nu)$ to the image of $\nu \in Tb$ under the functors $(S h) \circ \epsilon_b \colon Tb \to S J b \to S C$.
			\item on arrows, by mapping $(f, \alpha , \phi) \colon ( b, h , \nu) \to (b', h', \nu')$ to the composite
			\[ \begin{tikzcd}[column sep = small]
				{(Sh\circ\epsilon_b)(\nu)} && {(Sh\circ\epsilon_b \circ T f)(\nu')} && {(Sh \circ SJf \circ \epsilon_{b'})(\nu')} && {(Sh'\circ\epsilon_{b'})(\nu') .}
				\arrow["{(Sh\circ\epsilon_b) (\phi)}", from=1-1, to=1-3]
				\arrow["{(Sh*\epsilon_f)_{\nu'}}", from=1-3, to=1-5]
				\arrow["{(S\alpha*\epsilon_{b'})_{\nu'}}", from=1-5, to=1-7]
			\end{tikzcd}  \]
		\end{enumerate}
	\end{remark}
	
	\begin{remark}\label{rem:size-issues}
		By taking $\CAT$ to be the (very large) 2-category of large categories, the above construction does not raise size issues. However, note that if $\Bcat$ is locally small as a 2-category and $Jb$ is a small category for each $b\in\Bcat$, then $\tilde T \colon \CAT \to \CAT$ preserves local smallness, so that we can replace $\CAT$ with the (large) 2-category of locally small categories.
	\end{remark}
	
	\section{Unrelativising relative 2-monads}\label{sec:absolution}
	
	We now arrive at the main construction of the paper.  Throughout, $J \colon \Bcat \to \CAT$ is a fixed 2-functor and $\braket{T, \eta, (-)^{*}}$ is a fixed $J$-relative 2-monad.  Moreover, we make the following assumptions:
	\begin{enumerate}
		\item $\Bcat$ has a terminal object $1_\Bcat$ and it is preserved by $J$, where we use $\bullet$ to denote the unique object in $J1_\Bcat = 1_\CAT$ and $! = \id_{\bullet}$ the unique arrow;
		\item\label{enum:assumption_on_oplax_colimits} $\Bcat$ admits all oplax colimits for diagrams of shape $\op{(Jb)}$ for $b\in \Bcat$ (see Definition \ref{def:laxcolimits} below) and they are preserved by $J$;
		
		\item $J$ is 2-fully faithful.
	\end{enumerate}
	Under these assumptions, we will show that the left oplax Kan extension $\tilde T \colon \CAT\to \CAT$ of $T$ along $J$ constructed in the previous section carries the structure of a \emph{pseudomonad} (see e.g.\ \cite{lack-pseudomonads} for a definition including coherence axioms).   Later, in Section \ref{sec:algebras}, we will show that algebras for $T$ coincide with algebras for $\tilde T$ with respect to this pseudomonad structure.
	
	\begin{remark}\label{rem:all_small_oplax_colimits_suffices}
		Note that condition (2) is automatically satisfied under the stronger hypotheses that $J \colon \Bcat \to \CAT$ factors through small categories, $\Bcat$ admits all small oplax colimits and these are preserved by $J$.
	\end{remark}
	
	With the aim of turning a relative monad with root $J\colon D \to C$ into an ordinary monad on $C$, in \cite{ACU-relative}, the authors endow the category $\funcat{D}{C}$ with a \emph{skew-monoidal} structure in the sense of \cite{Szlach_nyi_2012}, hence identifying $J$-relative monads with \emph{skew-monoids} in $\funcat{D}{C}$. This draws a parallel with the familiar case of (ordinary) monads on a category $C$, which are monoids in the monoidal category $\funcat{C}{C}$. In the case of a ``well-behaved'' root functor in the sense of \cite[Definition 4.1]{ACU-relative}, this skew-monoidal structure is then proved to be actually monoidal, which is exploited to define a monad on $C$ with the same category of algebras. 
	
	In this section, we will proceed similarly. We will begin by endowing the category $\mathbf{Oplax}\funcat\Bcat\CAT$, where our relative 2-monad lives, with a `$J$-relative' composition operation making it into a skew-monoidal category. The skew-monoidal structure will be defined essentially by making use of the universal properties of left oplax Kan extensions. In contrast to \cite{ACU-relative}, by exploiting the above assumptions on $J$ we will define additional structure for the $J$-relative composition, which we will use to endow $\tilde T$ with the structure of a pseudomonad. Crucially, our assumptions -- which are not direct generalisations of the ``well-behavedness'' conditions of \cite[Definition 4.1]{ACU-relative} -- will not be strong enough to make the skew-monoidal structure actually monoidal (see Remark \ref{rem:comparison-with-ACU}).
	
	\begin{remark}\label{rem:tilde-is-functorial}
		The operation of taking left oplax Kan extensions along $J$ is functorial on $\mathbf{Oplax}\funcat{\Bcat}{\CAT}$, where given a pair of 2-functors $F, G \colon \Bcat\to\CAT$ and an oplax transformation $\sigma \colon F \implies G$, we define $\tilde \sigma \colon \tilde F \implies \tilde G$ as the unique strict transformation universally induced by the the oplax transformation $\zeta^G \circ \sigma \colon F \implies G \implies \tilde G \circ F$.	
		
		In particular, when $\sigma$ is a \emph{strict} transformation, its component at a category $C$ is the functor $\tilde \sigma _C \colon \tilde F (C) \to \tilde G (C)$ defined as
		\[\begin{tikzcd}
			{(b, h \colon J b \to C, \nu \in F b )} & {(b, h \colon J b \to C, \sigma_b(\nu)\in Gb)} \\
			{(b', h' \colon J b' \to C, \nu' \in F b' )} & {(b', h' \colon J b' \to C, \sigma_{b'}(\nu') \in G b' )}
			\arrow[maps to, from=1-1, to=1-2]
			\arrow["{(f,\alpha,\phi)}"', from=1-1, to=2-1]
			\arrow["{(f,\alpha, \sigma_b(\phi))}", from=1-2, to=2-2]
			\arrow[maps to, from=2-1, to=2-2]
		\end{tikzcd}\]
		which is well-defined since $\sigma_b(F f(\nu')) = G f(\sigma_{b'}(\nu'))$. 
	\end{remark}
	
	\subsection{Oplax colimits}As we will make extensive use of the universal properties of oplax colimits, we recall here their definition for the convenience of the reader.
	\begin{definition}\label{def:laxcolimits}
		Let $\Acat$ be a 2-category, $I$ a (1-)category, and $F \colon I \to \Acat$ a (1-)functor.  
		\begin{enumerate}
			\item A \emph{(strict) oplax cocone} for the diagram $F$ consists of the data of
			\begin{enumerate}
				\item an object $\ell \in \Acat$ called the \emph{vertex},
				\item 1-cells $c_i \colon Fi \to \ell$ for each $i \in I$ called the \emph{coprojections},
				\item and 2-cells $\lambda_m \colon c_j \circ Fm \Rightarrow c_i$ for each $m \colon i \to j \in I$
			\end{enumerate} 
			satisfying $\lambda_{\id_i} = \id_{c_i}$ and $\lambda_m \circ (\lambda_n \ast Fm) = \lambda_{n \circ m}$ for each composable pair of arrows $m \colon i \to j, n\colon j \to k \in I$, i.e.\ the pasting on the left coincides with the natural transformation on the right:
			\[
			\begin{tikzcd}
				Fi & Fj & Fk \\
				& {\ell,}
				\arrow["Fm", from=1-1, to=1-2]
				\arrow[""{name=0, anchor=center, inner sep=0}, "{c_i}"', curve={height=6pt}, from=1-1, to=2-2]
				\arrow[""{name=1, anchor=center, inner sep=0}, "Fn", from=1-2, to=1-3]
				\arrow[""{name=2, anchor=center, inner sep=0}, "{c_j}"{description}, from=1-2, to=2-2]
				\arrow["{c_k}", curve={height=-6pt}, from=1-3, to=2-2]
				\arrow["{\lambda_m}"', pos=0.4, shorten >=3pt, Rightarrow, from=1-2, to=0]
				\arrow["{\lambda_n}"', pos=0.4, shift left=4, shorten <=2pt, shorten >=6pt, Rightarrow, from=1, to=2]
			\end{tikzcd}
			\qquad 
			\begin{tikzcd}
				Fi & Fj & Fk \\
				& {\ell.}
				\arrow["Fm", from=1-1, to=1-2]
				\arrow[""{name=0, anchor=center, inner sep=0}, "{c_i}"', curve={height=6pt}, from=1-1, to=2-2]
				\arrow[""{name=1, anchor=center, inner sep=0}, "Fn", from=1-2, to=1-3]
				\arrow["{c_k}", curve={height=-6pt}, from=1-3, to=2-2]
				\arrow["{\lambda_{n \circ m}}"', pos=0.4, shift left=4, shorten <=7pt, shorten >=15pt, Rightarrow, from=1, to=0]
			\end{tikzcd}
			\]
			
			\item A \emph{(strict) oplax colimit} for the diagram $F$ is an oplax cocone $(\ell , c_i, \lambda_m)_{i, m \in I}$ for $F$ with the following universal properties.   
			\begin{enumerate}
				\item For any other oplax cocone $(k,d_i,\kappa_m)_{i , m \in I}$ for $F$, there exists a unique 1-cell $g \colon  \ell \to k$ such that $g\circ c_i =d_i$ and $g \ast \lambda_m = \kappa_m$ for each $m \colon i \to j \in I$.
				\item Moreover, given any other 1-cell $h \colon \ell \to k$ and 2-cells $\beta_i \colon h \circ c_i \Rightarrow g \circ c_i = d_i$ such that $(g \ast \lambda_m) \circ( \beta_j \ast Fm ) = \beta_i \circ (h \ast \lambda_m)$ for each $m \colon i \to j \in I$, i.e.\ the pastings
				\[
				\begin{tikzcd}[column sep=small]
					Fi && Fj \\
					& \ell & \ell \\
					& {\ell'}
					\arrow[""{name=0, anchor=center, inner sep=0}, "Fm", from=1-1, to=1-3]
					\arrow[""{name=1, anchor=center, inner sep=0}, "{c_i}"', from=1-1, to=2-2]
					\arrow["{c_j}"{description}, from=1-3, to=2-2]
					\arrow["{c_j}", from=1-3, to=2-3]
					\arrow["g"', from=2-2, to=3-2]
					\arrow["{\beta_j}", Rightarrow, from=2-3, to=2-2]
					\arrow["h", curve={height=-6pt}, from=2-3, to=3-2]
					\arrow["{\lambda_m}"', shift left=5, shorten >=5pt, Rightarrow, from=0, to=1]
				\end{tikzcd}
				\qquad \begin{tikzcd}[column sep=small]
					Fi && Fj \\
					\ell & \ell \\
					& {\ell'}
					\arrow[""{name=0, anchor=center, inner sep=0}, "Fm", from=1-1, to=1-3]
					\arrow["{c_i}"', from=1-1, to=2-1]
					\arrow[""{name=1, anchor=center, inner sep=0}, "{c_i}"{description}, from=1-1, to=2-2]
					\arrow["{c_j}", from=1-3, to=2-2]
					\arrow["g"', curve={height=6pt}, from=2-1, to=3-2]
					\arrow["{\beta_i}", Rightarrow, from=2-2, to=2-1]
					\arrow["h", from=2-2, to=3-2]
					\arrow["{\lambda_m}"', shift left=5, shorten >=5pt, Rightarrow, from=0, to=1]
				\end{tikzcd}
				\]
				coincide for each $m \colon i \to j \in I$, then there exists a unique 2-cell $\beta \colon h \Rightarrow g$ such that $\beta_i = \beta \ast c_i$.
			\end{enumerate}
		\end{enumerate}
	\end{definition}
	
	\begin{remark}\label{rem:iso between oplax colimits}
		As with ordinary colimits, note that if $k \in \Acat$ is (the vertex of) an oplax cocone for a diagram $F \colon I \to \Acat$ for which $\ell_0,\ell_1 \in \Acat$ are both (vertices of) oplax colimits, so that in particular there exists a canonical isomorphism $f \colon \ell_0 \to \ell_1$, then the uniquely determined 1-cells $g_0 \colon \ell_0 \to k$, $g_1 \colon \ell_1 \to k$ satisfy $g_1\circ f = g_0$. 
	\end{remark}
	\begin{example}\label{ex:oplax-colimit-of-itself}
		Let $I$ be a category, and let $R \colon \op{I} \to \CAT$ be the functor that sends each object of $I$ to the terminal category $1_\CAT$ and each morphism of $I$ to the identity on $1_\CAT$.  In this case, the oplax colimit is given by $I$ itself, with the corresponding 1-cell of the cocone $c_i \colon 1_\CAT \to I$ at an object $i \in \op{I}$ given by the functor that picks out the object $i \in I$, while the 2-cell $\lambda_m \colon c_j \circ \id_{1_\CAT} \Rightarrow c_i$ at an arrow $m \colon i \to j \in \op{I}$ is the natural transformation whose (only) component is $m \colon j \to i \in I$. 
	\end{example}
	
	\subsection{Skew-monoidal structure}\label{ssec:structure-on-Bcat-to-CAT}
	
	We define the \emph{$J$-relative composition} operation
	\[ \circ^J \colon {\mathbf{Oplax}}\funcat {\Bcat}\CAT \times {\mathbf{Oplax}}\funcat {\Bcat}\CAT \to {\mathbf{Oplax}}\funcat {\Bcat}\CAT \]
	by letting, on objects, $F \circ^J G \coloneqq \tilde F \circ G$, which clearly extends to a functor making use of Remark \ref{rem:tilde-is-functorial}. This operation can be endowed with the following structural maps.
	\begin{enumerate}
		\item We define a \emph{left unit}, a family of oplax transformations $\mathfrak{l}_F \colon J \circ^J F \implies F$ for each 2-functor $F \colon \Bcat\to\CAT$, natural in $F$, as follows.
		
		First, consider the strict transformation $\id_{J} \colon J \implies J$. By the universal property of $\tilde J$, it uniquely factors through $\zeta^J \colon J \implies \tilde J \circ J$ via a strict transformation $\bar{\mathfrak{l}} \colon \tilde J \implies \id_{\CAT}$, from which we define $\mathfrak{l}_F \coloneq \bar{\mathfrak{l}} * F$. Explicitly, this means that for each object $c \in \Bcat$, the component $(\mathfrak{l}_F)_c \colon \tilde J(C) \to C$ is the functor:
		\[ \begin{tikzcd}
			{(b, h \colon J b \to Fc, x \in J b)} & {h(x)} \\
			{(b', h'\colon J b' \to Fc, x' \in J b')} & {h'(x')}
			\arrow[maps to, from=1-1, to=1-2]
			\arrow["{(f,\alpha,g)}"', from=1-1, to=2-1]
			\arrow["{\alpha_{x'}\circ h(g)}", from=1-2, to=2-2]
			\arrow[maps to, from=2-1, to=2-2]
		\end{tikzcd} \]
		
		\item We define a \emph{right unit}, a family of oplax transformations $\mathfrak{r}_F \colon F \implies F \circ^J J$ for each 2-functor $F \colon \Bcat\to\CAT$, natural in $F$, by considering for each $b\in \Bcat$ the functor
		\[ (\mathfrak{r}_F)_b \coloneqq (\iota_{F, J b})_b (\id_{Jb}) \colon F b \to \tilde F ( J b ) \]
		explicitly defined as
		\[ \begin{tikzcd}
			\nu & {(b , \id_{Jb}\colon J b \to J b, \nu \in F b)} \\
			{\nu'} & {(b , \id_{Jb}\colon J b \to J b, \nu' \in F b)}
			\arrow[maps to, from=1-1, to=1-2]
			\arrow["\phi"', from=1-1, to=2-1]
			\arrow["{(\id_b, \id, \phi)}", from=1-2, to=2-2]
			\arrow[maps to, from=2-1, to=2-2]
		\end{tikzcd}\]
		and, for each 1-cell $f \colon b' \to b\in\Bcat$, the natural transformation $(\mathfrak{r}_F)_f \colon (\mathfrak{r}_F)_b \circ Ff \implies \tilde F J f \circ (\mathfrak{r}_F)_{b'}$ defined at an object $\nu'\in F b'$ by the arrow
		\[ ((\mathfrak{r}_F)_f)_{\nu'} \coloneqq ( f, \id, \id) \colon (b, \id_{Jb}, Ff(\nu')) \to (b', J f, \nu')\]
		in $\tilde F (Jb)$.

		\item We define an \emph{associator}, a family of oplax transformations $\mathfrak{a}_{F,G,H} \colon (F \circ^J G)\circ^J H \implies F \circ^J ( G \circ^J H )$ for each triple of 2-functors $F,G,H \colon \Bcat\to\CAT$, natural in $F$, $G$ and $H$, as follows.
		
		First, consider the oplax transformation $\tilde F * \zeta^G \colon \tilde F \circ G \implies \tilde F \circ \tilde G \circ J$. By the universal property of $\tilde{ \tilde F \circ G }$, it uniquely factors through $\zeta^{\tilde F \circ G} \colon \tilde F \circ G \implies \tilde {( \tilde F \circ G)} \circ J$ via a strict transformation $\bar{\mathfrak{a}}_{F,G} \colon \tilde { \tilde F \circ G }  \implies \tilde F \circ \tilde G$, from which we define $\mathfrak{a}_{F,G,H} \coloneqq \bar{\mathfrak{a}}_{F,G} * H$. Explicitly, this means that for each object $c \in \Bcat$, the component $({\mathfrak{a}}_{F,G,H})_c \colon \tilde { \tilde F \circ G } (H c) \to \tilde F(\tilde G(H c))$ is the functor defined:
		\begin{enumerate}
			\item on objects, by mapping $(b, h \colon J b \to H c, ( d , k \colon J d \to Gb, \nu \in Fd))$ to the triple $(d , K \colon J d \to \tilde G (Hc), \nu \in F d)$ where $K$ is the functor
			\[ \begin{tikzcd}[column sep = 30pt]
				Jd & Gb & {\tilde G(Jb)} & {\tilde G (H c)}
				\arrow["k", from=1-1, to=1-2]
				\arrow["{\zeta^G_b}", from=1-2, to=1-3]
				\arrow["{\tilde G(h)}", from=1-3, to=1-4]
			\end{tikzcd} \]
			i.e.\ such that $K(x) \coloneqq ( b , h \colon J b \to Hc, k(x) \in G b)$;
			
			\item on arrows, by mapping $(f,\alpha, ( g, \beta, \psi ) ) \colon ( b , h , ( d, k , \nu)) \to (b', h', (d', k', \nu') )$ to the triple of $g \colon d'\to d\in \Bcat$, the natural transformation $J g \circ K \implies K'$ given by the composite
			\[ \begin{tikzcd}[column sep = 1.75 em, row sep = 1em]
				Jd && Gb && {\tilde G (Jb)} \\
				&&&&&& {\tilde G (H c)} \\
				{Jd'} && {Gb'} && {\tilde G ( J b')}
				\arrow["k", from=1-1, to=1-3]
				\arrow["\beta", shorten <=11pt, shorten >=11pt, Rightarrow, from=1-1, to=3-3]
				\arrow["{\zeta^G_b}", from=1-3, to=1-5]
				\arrow["{\zeta^G_f}", shorten <=11pt, shorten >=11pt, Rightarrow, from=1-3, to=3-5]
				\arrow["{\tilde G (h)}", curve={height=-12pt}, from=1-5, to=2-7]
				\arrow["Jg", from=3-1, to=1-1]
				\arrow["{k'}"', from=3-1, to=3-3]
				\arrow["Gf"{description}, from=3-3, to=1-3]
				\arrow["{\zeta^G_{b'}}"', from=3-3, to=3-5]
				\arrow["{\tilde G (Jf)}"{description}, from=3-5, to=1-5]
				\arrow[""{name=0, anchor=center, inner sep=0}, "{\tilde G (h')}"', curve={height=12pt}, from=3-5, to=2-7]
				\arrow["{\tilde G ( \alpha)}", shorten <=9pt, shorten >=9pt, Rightarrow, from=1-5, to=0]
			\end{tikzcd} \]
			and $\psi \colon \nu \to Fg(\nu') \in F d$.
			
		\end{enumerate}
		
		\todo{check naturality in arguments}
	\end{enumerate}
	
	\begin{remark}
		Note how the transformations $\mathfrak{l}_F$ and $\mathfrak{a}_{F,G,H}$ are actually strict rather than merely oplax.
	\end{remark}
	
	\begin{proposition}
		The category $\mathbf{Oplax}\funcat{\Bcat}{\CAT}$, together with the composition $\circ^J$ and the families $\mathfrak{l},\mathfrak{r}, \mathfrak{a}$, is a skew-monoidal category.
		\begin{proof}
			The proof follows exactly as in \cite[Theorem 3.1]{ACU-relative}.
		\end{proof}
	\end{proposition}
	
	\begin{remark}\label{rem:comparison-with-ACU}
		In \cite{ACU-relative}, the hypotheses of ``well-behavedness'' on the root functor ensures that the families $\mathfrak{l},\mathfrak{r}, \mathfrak{a}$ are invertible, i.e.\ that the skew-monoidal structure is monoidal. If our assumptions on the root 2-functor were to imply the same, our motivating example of weak ultracategories would be left out. For instance, it is easy to see that $\bar{\mathfrak{l}}$ is not invertible for $J \colon \Set \hookrightarrow \CAT$ being the inclusion of sets as discrete categories: in other words, the left oplax Kan extension of $J$ along itself is not (isomorphic to) the identity. In a way, our assumptions on $J$ are intended to do without these inverses.
	\end{remark}
	
	In light of the previous remark, we now use our assumptions on $\Bcat$ and on the root $J \colon \Bcat \to \CAT$ to define additional structure for the $J$-relative composition.
	\begin{enumerate}
		\item[(4)]    We define a strict transformation $\mathfrak{d} \colon \id_{\CAT} \implies \tilde J$ by setting $\mathfrak{d}_C \colon C \to \tilde J (C)$, for each category $C$, to be the functor:
		\[ \begin{tikzcd}
			c & {(1_{\Bcat}, c \colon 1_{\CAT} \to C, \bullet)} \\
			{c'} & {(1_{\Bcat}, c ' \colon 1_{\CAT} \to C, \bullet)}
			\arrow[maps to, from=1-1, to=1-2]
			\arrow["f"', from=1-1, to=2-1]
			\arrow["{(\id, f, \id)}", from=1-2, to=2-2]
			\arrow[maps to, from=2-1, to=2-2]
		\end{tikzcd}
		\] 
		where by assumption $J 1_{\Bcat} = 1_{\CAT}$ is the terminal category.     Clearly, $\bar{\mathfrak {l}} \circ \mathfrak{d} = \id$.
		
		\item[(5)] For each pair of 2-functors $F, G \colon \Bcat \to \CAT$ and for each category $C$, let
		\[ (\mathfrak{s}_{F,G})_C \colon \tilde F ( \tilde G (C)) \to \tilde {\tilde F \circ G}(C) \]
		be the functor defined as follows.\todo{maybe it needs to be explained more in detail? i wrote less than what we had in the first draft}
		
		Consider an object  $( b, h \colon J b \to \tilde G C, \nu \in Fb) \in \tilde F ( \tilde G C)$. For each object $x \in Jb$ and each arrow $g \colon x \to y\in J b$, we will use the following notation to denote their image under $h$:
		\[
		{h(x) \coloneqq ( Rx , a_x \colon J Rx \to C , \nu_x \in G R x )} , \quad {{h(g) \coloneqq \left( Rg , \gamma_g , \psi_g \right)}}.
		\]
		Note that this yields a functor $R \coloneqq \pi\circ \op h\colon \op{(Jb)} \to \Bcat$. Then, we define the action of $(\mathfrak{s}_{F,G})_C$ on $(b, h, \nu)$ as the tuple $(\ell, a \colon J \ell \to C, (b, Q \colon Jb \to G \ell, \nu \in Fb ) )$ where:
		\begin{enumerate}
			\item $\ell \in \Bcat$ is the oplax colimit of the diagram $R$, witnessed by families
			\[ \set{ c_x \colon Rx \to \ell | x \in J b } \mbox{ and } \set{ \lambda_g \colon c_x \circ R g \implies c_y | g \colon x \to y \in J b },\]
			\item $a \colon J \ell \to C$ is the functor uniquely induced by seeing $C$ as an oplax cocone on the diagram $J \circ R$ through the families
			\[ \set{ a_x \colon J Rx \to C | x \in J b } \mbox{ and } \set{ \gamma_g \colon a_x \circ J R g \implies a_y | g \colon x \to y \in J b },\]
			\item and $Q \colon J b \to G\ell$ is the functor extending $x \mapsto Gc_x(\nu_x)$ by setting $Qg$, for an arrow $g \colon x \to y \in J b$, to be the composite:
			\[ \begin{tikzcd}
				{Gc_x(\nu_x)} && {Gc_x(GRg (\nu_y)) = G(c_x\circ Rg)(\nu_y)} && {Gc_y(\nu_y).}
				\arrow["{Gc_x(\psi_x)}", from=1-1, to=1-3]
				\arrow["{G(\lambda_g)_{\nu_y}}", from=1-3, to=1-5]
			\end{tikzcd} \]
		\end{enumerate}
		
		Now consider an arrow $(f, \alpha,\phi) \colon ( b, h , \nu ) \to ( b', h' , \nu' ) \in \tilde F (\tilde G C)$. Building on the previous notation, we will denote the component of $\alpha \colon h\circ J f \implies h'$ at $x \in Jb'$ as the triple
		\[
		\left( w_x , \delta_x, \chi_x\right) \colon { h ( Jf(x))  =( RJf(x) , a_{Jf(x)} , \nu_{Jf(x)} )} \to {( R'x , a_x'  , \nu_x' ) = h'(x)}.
		\]
		Then, we define the action of $(\mathfrak{s}_{F,G})_C$ on $( f , \alpha, \phi )$ as the tuple $\left( \bar f , \bar \alpha , ( f, \theta, \phi ) \right)$ where:
		\begin{enumerate}
			\item $\bar f \colon \ell' \to \ell \in \Bcat$ is the 1-cell uniquely induced by seeing $\ell$ as an oplax cocone on $R'$ through the families
			\[ \set{ c_{Jf(x)} \circ w_x | x \in J b ' } \mbox{ and } \set{ \lambda_{Jf(g)} * w_y | g \colon x \to y \in J b '}  ,\]
			\item $\bar\alpha \colon a \circ \bar f \implies a'$ is the natural transformation uniquely induced by the universal property of $J\ell'$ as the oplax colimit of $J \circ R'$ through the family
			\[ \set { \delta_x \colon a \circ f^F \circ J c'_x \implies a'_x | x \in J b' },\]
			\item and $\theta \colon Q \circ J f \implies G \bar f \circ Q'$ is the natural transformation defined at $x \in J b'$ by the arrow
			\[ \begin{tikzcd}[column sep = small]
				QJf(x) = Gc_{Jf(x)}(\nu_{Jf(x)}) \ar{rr}{G c_{Jf(x)}(\chi_x)}  && {G c_{Jf(x)} ( G w_x (\nu_x')) } = {G \bar f ( G c_x'(\nu_x'))} = { G\bar f (Q' x).}
			\end{tikzcd}  \]	
		\end{enumerate}
		
		We omit the routine verification that the above functors are well-defined on equivalence classes and define a strict transformation 
		$\mathfrak{s}_{F,G} \colon \tilde F \circ \tilde G \implies \tilde { \tilde F \circ G }$, and that the family $\set{ \mathfrak{s}_{F,G} | F, G \colon \Bcat \to \CAT}$ is natural in both arguments.
	\end{enumerate}
	
	\begin{proposition}\label{prop:structure-on-Bcat-to-CAT}
		For each 2-functor $F \colon \Bcat\to\CAT$, the following diagrams of strict transformations commute up to a canonical invertible modification:
		\[(a) \quad \begin{tikzcd}
			{\tilde J\circ \tilde F } && {\tilde { \tilde J \circ F }} \\
			& {\tilde F }
			\arrow["{\mathfrak{s}_{J,F}}", from=1-1, to=1-3]
			\arrow["{\mathfrak{d}*\tilde F }", from=2-2, to=1-1]
			\arrow["{\tilde{\mathfrak{d}*F}}"', from=2-2, to=1-3]
		\end{tikzcd} \qquad (b)\quad  \begin{tikzcd}
			{\tilde F \circ \tilde J } && {\tilde {\tilde F \circ J }} \\
			& {\tilde F }
			\arrow["{\mathfrak{s}_{F,J}}", from=1-1, to=1-3]
			\arrow["{\tilde F * \mathfrak{d}}", from=2-2, to=1-1]
			\arrow["{\tilde {\mathfrak{r}_F}}"', from=2-2, to=1-3]
		\end{tikzcd}\]
		\[ (c)\quad  \begin{tikzcd}
			{\tilde F ^3} &&&& {\tilde F \circ \widetilde{\widetilde F \circ F}} \\
			{\widetilde {\widetilde F \circ F} \circ \tilde F} && {\widetilde {\widetilde {\widetilde F \circ F} \circ F}} && {\widetilde{ \tilde F ^2 \circ F}}
			\arrow["{\tilde F *\mathfrak{s}_{F,F}}", from=1-1, to=1-5]
			\arrow["{\mathfrak{s}_{F,F}* \tilde F }"', from=1-1, to=2-1]
			\arrow["{\mathfrak{s}_{F, \tilde F \circ F}}", from=1-5, to=2-5]
			\arrow["{\mathfrak{s}_{\tilde F \circ F, F}}"', from=2-1, to=2-3]
			\arrow["{\widetilde{\bar{{{\mathfrak{a}}}}_{F,F}* F}}"', from=2-3, to=2-5]
		\end{tikzcd}\]
		\begin{proof}
			We consider each diagram in turn.
			
			\begin{enumerate}
				\item[(a)] Let $C$ be a category and let $(b, h\colon J b \to C, \nu \in F b) \in \tilde F C$, which $(\tilde {\mathfrak{d} * F})_C$ maps to the triple $(b, h , \mathfrak{d}_{F b}(\nu))$ i.e.\ $(b , h , (1_{\Bcat}, \nu \colon 1_{\CAT}\to F b , \bullet))$. On the other side, $\mathfrak{d}_{\tilde F C}$ maps $(b,h,\nu)$ to the tuple $(1_{\Bcat}, (b,h , \nu) \colon 1_{\CAT} \to \tilde F C, \bullet)$, which is then sent by $(\mathfrak{s}_{J,F})_C$ to a tuple $( \ell, a \colon J \ell \to C , ( 1_{\Bcat}, Q \colon 1_{\CAT} \to F \ell , \bullet ) )$, where by definition $\ell \in \Bcat$ is the oplax colimit of the diagram $R \colon \op{(1_{\CAT})} \to \Bcat$ that picks out the object $b\in \Bcat$. Evidently, $b$ itself is an oplax colimit of $R$, so there is a canonically induced invertible 1-cell $u\colon b \to \ell \in \Bcat$ yielding the invertible arrow:
				\[ ( u , \id, \id ) \colon ( \ell , a, (1_{\Bcat}, Q, \bullet)) \to ( b , h , ( 1_{\Bcat}, \nu , \bullet))\] 
				in $(\tilde{\tilde J \circ F})(C)$, which is the component of an invertible modification $\mathfrak{s}_{J,F} \circ (\mathfrak{d}*\tilde F) \Rrightarrow \tilde{\mathfrak{d}*F}$. \todo{check}
				
				\item[(b)]  Let $C$ be a category and let $(b, h\colon J b \to C, \nu \in F b) \in \tilde F C$, which $(\tilde {\mathfrak{r}_F})_C$ maps to the triple $(b, h , (\mathfrak{r}_F)_b(\nu))$ i.e.\ $(b , h , ( b, \id_{Jb} , \nu ))$. On the other side, $\tilde F (\mathfrak{d}_C)$ maps $(b,h, \nu)$ to the triple $(b , \mathfrak{d}_C \circ h , \nu)$, which $(\mathfrak{s}_{F,J})_C$ then maps to a tuple $(\ell, a \colon J \ell \to C , ( b, Q \colon J b \to J \ell , \nu ))$ where by definition $\ell \in \Bcat$ is the oplax colimit of the constant diagram $R \colon \op{(Jb)} \to \Bcat$ of value $1_{\Bcat}$. By Example \ref{ex:oplax-colimit-of-itself} and 2-fully-faithfulness of $J$, $b\in\Bcat$ is also an oplax colimit of $R$, so there is a canonically induced invertible 1-cell $v\colon b \to \ell \in \Bcat$ yielding the invertible arrow:
				\[ ( v , \id, \id ) \colon ( \ell , a, (b, Q, \nu)) \to ( b , h , ( b, \id_{Jb} , \nu))\] 
				in $(\tilde{\tilde F \circ J})(C)$, which is the component of an invertible modification $\mathfrak{s}_{F,J} \circ (\tilde F * \mathfrak{d}) \Rrightarrow \tilde{\mathfrak{r}_F}$. \todo{check}
				
				\item[(c)]Let $C$ be a category and let $(b , h \colon J b \to \tilde F^2 C, \nu \in F b) \in \tilde F ^ 3 C$. By definition, $\tilde F ((\mathfrak{s}_{F,F})_C)$ maps it to the triple $( b , (\mathfrak{s}_{F,F})_C\circ h , \nu )$, which $(\mathfrak{s}_{F, \tilde F \circ F})_C$ then maps to a tuple $(\bar\ell, \bar a \colon J \bar \ell \to C, ( b , \bar Q \colon J b \to \tilde F (F \bar \ell ), \nu))$ where
				\[ \bar \ell \coloneqq \opcolim_{x \in \op{(Jb)}} \left( \opcolim_{y \in \op{(JRx)}} \left( \pi a_x(y)\right) \right) \]
				Since oplax colimits commute with oplax colimits (cf.\ \cite{bonart2024}), there is a canonical isomorphism: 
				\[ \opcolim_{x \in \op{(Jb)}} \left( \opcolim_{y \in \op{(J R x)}} \left( \pi a_x (y) \right) \right) \iso \opcolim_{(x,y) \in \op{ \left(\opcolim_{x' \in \op{(J b)} } (J R x') \right) }} \left( \pi a_x (y) \right) \]
				where, in the right-hand side, $(x,y)$ denotes the image of $y \in J R x$ along the coprojection functor $J R x \to \opcolim_{x' \in \op{(J b)} } (J R x' )$. 
				Chasing $(b, h ,\nu)$ along the other side of the diagram, instead, we have that:
				\begin{enumerate}
					\item[(1)] $(\mathfrak{s}_{F,F})_{\tilde F C}$ maps it to the tuple 
					\[(\ell, a \colon J \ell \to \tilde F C, ( b , Q \colon J b \to F \ell , \nu))\] 
					where $\ell \in \Bcat$ is the oplax colimit of the diagram $R \colon \op{(Jb)}\to \Bcat$ -- with coprojections $\set{ c_x \colon Rx \to \ell | x \in J b}$; 
					
					\item[(2)] $(\mathfrak{s}_{\tilde F \circ F , F})_C$ maps (1) to the tuple 
					\[ ( \ell' , a ' \colon J \ell ' \to C, ( \ell , Q ' \colon J \ell \to F \ell' , ( b , Q \colon J b \to F \ell , \nu ))) \] 
					where $\ell' \in \Bcat$ is the oplax colimit of the diagram $\pi \circ \op{a} \colon \op{(J\ell)} \to \Bcat$;
					
					\item[(3)] $(\tilde{\bar{{\mathfrak{a}}}_{F,F} * F})_C$ maps (2) to the tuple
					\[ ( \ell' , a ' \colon J \ell ' \to C, (\bar{{\mathfrak{a}}}_{F,F})_{F\ell'} ( \ell , Q ' \colon J \ell \to F \ell' , ( b , Q \colon J b \to F \ell , \nu ))) \] 
					i.e.\ $( \ell' , a ' \colon J \ell ' \to C, ( b , K \colon J b \to \tilde F (F\ell'), \nu ))$, where $K \coloneqq \tilde F ( Q' ) \circ \zeta_{\ell'} \circ Q$. 
				\end{enumerate}
				Note therefore that, by our assumptions, $J \ell$ is the oplax colimit of $J \circ R$, with coprojections $\set{ J c_x \colon J R x \to J \ell | x \in J b }$. Therefore, recalling also that $a \circ J c_x = a_x$ by definition of $a \colon J \ell \to \tilde F C$, we have that:
				\begin{align*}
					\opcolim_{(x,y) \in \op{ \left(\opcolim_{x' \in \op{(J b)} } (J R x') \right) }} \left( \pi a_x (y) \right) & = \opcolim_{z = J c_x (y)  \in \op{ ( J \ell ) }} \left( \pi a_x (y) \right)\\
					&= \opcolim_{z = J c_x (y)  \in \op{ ( J \ell ) }} \left( \pi a (Jc_x (y)) \right)\\
					&= \opcolim_{z \in \op{ ( J \ell ) }} \left( \pi a (z) \right) = \ell' ,
				\end{align*}
				and hence there is a canonically induced invertible 1-cell $w \colon \ell'\to \bar \ell \in \Bcat$ yielding the invertible arrow:
				\[ ( w, \id, \id ) \colon (\bar\ell, \bar a , ( b , \bar Q , \nu)) \to ( \ell' , a' , ( b, K, \nu)) .\]
				in $(\tilde{\tilde F^2\circ F})(C)$, which is the component of an invertible modification \[\mathfrak{s}_{F,\tilde F \circ F} \circ (\tilde F * \mathfrak{s}_{F,F}) \Rrightarrow \tilde{\bar{{\mathfrak{a}}}_{F,F} * F} \circ \mathfrak{s}_{\tilde F \circ F , F } \circ (\mathfrak{s}_{F,F}*\tilde F).\]
			\end{enumerate}
		\end{proof}
	\end{proposition}
	
	\begin{remark}
		Given 2-functors $F, G , H \colon \Bcat \to \CAT$, the same proof of (c) can be generalised to show that a similar diagram going from $\tilde F \circ \tilde G \circ \tilde H$ to $\tilde{\tilde F \circ \tilde G \circ H}$ also commutes up to a canonical invertible modification, but we will not need this in the rest of the paper.
	\end{remark}

	\subsection{Pseudomonad structure}We are finally ready to endow $\tilde T$ with the structure of a pseudomonad on $\CAT$, that is, a unit $\eta^\sharp$ and a multiplication $\mu^\sharp$ together with appropriate coherent invertible modifications expressing unitality and associativity. Following \cite{ACU-relative}, we first define a `$J$-relative multiplication' for $T$ -- i.e.\ a strict transformation $\mu \colon T \circ^J T \implies T$ -- by considering, for each $b\in B$, the functor $\mu_b \colon \tilde T (T b) \to T b$ uniquely corresponding to the strict transformation $(-)^* \colon \funcat{J-}{ Tb} \implies \funcat{T-}{Tb}$ by the universal property of the left oplax Kan extension. Explicitly, $\mu_b$ acts as follows:
	\[ \begin{tikzcd}
		{(d, k \colon J d \to Tb, \nu \in T d)} & {k^*(\nu)} \\
		{(d', k' \colon J d' \to Tb', \nu' \in T d')} & {k'^*(\nu').}
		\arrow[maps to, from=1-1, to=1-2]
		\arrow["{(f,\alpha,\phi)}"', from=1-1, to=2-1]
		\arrow["{\alpha^*_{\nu'}\circ k^*(\phi)}", from=1-2, to=2-2]
		\arrow[maps to, from=2-1, to=2-2]
	\end{tikzcd} \]
	
	\begin{proposition}
		The following diagrams of oplax transformations commute:
		\[ (a)\quad \begin{tikzcd}[cramped]
			{\tilde J \circ T } && {\tilde T \circ T} \\
			\\
			T && T
			\arrow["{\tilde \eta*T}", from=1-1, to=1-3]
			\arrow["\mu", from=1-3, to=3-3]
			\arrow["{\mathfrak{d}*T}", from=3-1, to=1-1]
			\arrow[equals, from=3-1, to=3-3]
		\end{tikzcd} \qquad (b)\quad \begin{tikzcd}[cramped]
			{\tilde T \circ J} && {\tilde T \circ T} \\
			\\
			T && T
			\arrow["{\tilde T * \eta}", from=1-1, to=1-3]
			\arrow["\mu", from=1-3, to=3-3]
			\arrow["{\mathfrak{r}_T}", from=3-1, to=1-1]
			\arrow[equals, from=3-1, to=3-3]
		\end{tikzcd}  \]
		\[ (c) \quad \begin{tikzcd}
			& {\tilde T \circ \tilde T \circ T} & {\tilde T \circ T } \\
			{\tilde{\tilde T \circ T }\circ T} \\
			{\tilde T \circ T} && T
			\arrow["{{\tilde T *\mu}}", from=1-2, to=1-3]
			\arrow["\mu", from=1-3, to=3-3]
			\arrow["{{\bar{\mathfrak{a}}_{T,T}*T}}", from=2-1, to=1-2]
			\arrow["{{\tilde \mu * T }}"', from=2-1, to=3-1]
			\arrow["\mu"', from=3-1, to=3-3]
		\end{tikzcd} \] 
		\begin{proof}
			The proof carries over from that of \cite[Theorem 3.4]{ACU-relative} -- for diagram (a), also making use of the fact that $\mathfrak{l}_T \circ (\mathfrak{d}* T) = \id$. 
		\end{proof}
	\end{proposition}
	
	\begin{definition}
		We endow $\tilde T \colon \CAT \to \CAT$ with the following pseudomonad structure.
		\begin{enumerate}
			\item First, we define a unit $\eta^\sharp \colon \id_{\CAT} \implies \tilde T$ as the strict transformation:
			\[ \begin{tikzcd}
				{\id_{\CAT}} & {\tilde J} & {\tilde T}
				\arrow["\mathfrak{d}", from=1-1, to=1-2]
				\arrow["{\tilde \eta}", from=1-2, to=1-3]
			\end{tikzcd} \]
			
			\item Second, we define a multiplication $\mu^\sharp \colon \tilde T \circ \tilde T \implies \tilde T$ as the strict transformation:
			\[ \begin{tikzcd}
				{\tilde T \circ\tilde T} & {\tilde{\tilde T \circ T}} & {\tilde T}
				\arrow["{\mathfrak{s}_{T,T}}", from=1-1, to=1-2]
				\arrow["{\tilde\mu}", from=1-2, to=1-3]
			\end{tikzcd} \]
			
			\item We define a \emph{left unit} $\mathfrak{p} \colon \mu^\sharp \circ (\eta^\sharp * \tilde T) \Rrightarrow \id$ as the invertible modification obtained by pasting:
			\[\begin{tikzcd}[cramped, column sep = small, row sep = small]
				&&& {\tilde T \circ \tilde T} \\
				& {\tilde J \circ \tilde T} &&&& {\tilde {\tilde T \circ T}} \\
				&&& {\tilde {\tilde J \circ T}} \\
				{\tilde T } &&&&&& {\tilde T }
				\arrow["{{{\mathfrak{s}_{T,T}}}}", from=1-4, to=2-6]
				\arrow["{\tilde\eta*\tilde T}", from=2-2, to=1-4]
				\arrow["{\mathfrak{s}_{J,T}}", from=2-2, to=3-4]
				\arrow["{\tilde \mu}", from=2-6, to=4-7]
				\arrow["{\tilde{\tilde \eta * T }}"', from=3-4, to=2-6]
				\arrow["{\mathfrak{d}*\tilde T }", from=4-1, to=2-2]
				\arrow[""{name=0, anchor=center, inner sep=0}, "{\tilde{\mathfrak{d} * T }}"{description}, from=4-1, to=3-4]
				\arrow[equals, from=4-1, to=4-7]
				\arrow["\iso"', shorten <=3pt, shorten >=6pt, Rightarrow, from=2-2, to=0]
			\end{tikzcd}\]
			
			\item We define a \emph{right unit} $\mathfrak{q} \colon \id \Rrightarrow \mu^\sharp \circ ( \tilde T * \eta^\sharp)$ as the invertible modification obtained by pasting:
			\[ \begin{tikzcd}[cramped, column sep = small, row sep = small]
				&&& {\tilde T \circ \tilde T } \\
				& {\tilde T \circ \tilde J } &&&& {\tilde {\tilde T \circ T}} \\
				&&& {\tilde {\tilde T \circ J}} \\
				{\tilde T } &&&&&& {\tilde T }
				\arrow["{{\mathfrak{s}_{T,T}}}", from=1-4, to=2-6]
				\arrow["{{\tilde T * \tilde \eta}}", from=2-2, to=1-4]
				\arrow["{{\mathfrak{s}_{T,J}}}", from=2-2, to=3-4]
				\arrow["{{\tilde \mu}}", from=2-6, to=4-7]
				\arrow["{{\tilde{\tilde T *\eta}}}"', from=3-4, to=2-6]
				\arrow["{{\tilde T * \mathfrak{d}}}", from=4-1, to=2-2]
				\arrow[""{name=0, anchor=center, inner sep=0}, "{{\tilde{\mathfrak{r}_T}}}"{description}, from=4-1, to=3-4]
				\arrow[equals, from=4-7, to=4-1]
				\arrow["\iso", shorten <=6pt, shorten >=3pt, Rightarrow, from=0, to=2-2]
			\end{tikzcd} \]
			
			\item Finally, we define an \emph{associator} $\mathfrak{x} \colon \mu^\sharp \circ ( \tilde T * \mu^\sharp ) \Rrightarrow \mu^\sharp \circ (\mu^\sharp * \tilde T )$ as the invertible modification obtained by pasting:
			\[ 
			\begin{tikzcd}
				{\tilde T \circ \tilde T\circ \tilde T} && {\tilde T\circ \tilde{\tilde T \circ T }} &&& {\tilde T\circ \tilde T} \\
				{\tilde{\tilde T \circ T}\circ \tilde T } &&& {\tilde{\tilde T \circ \tilde T \circ T}} \\
				&& {\tilde{\tilde{\tilde{T}\circ T}\circ T}} &&& {\tilde{\tilde T\circ T }} \\
				{\tilde T \circ \tilde T} &&& {\tilde{\tilde T \circ T}} && {\tilde T}
				\arrow["{\tilde T*\mathfrak{s}_{T,T}}", from=1-1, to=1-3]
				\arrow["{\mathfrak{s}_{T,T}*\tilde T}"', from=1-1, to=2-1]
				\arrow["{\tilde T*\tilde \mu}", from=1-3, to=1-6]
				\arrow["{\mathfrak{s}_{T, \tilde T\circ T}}", from=1-3, to=2-4]
				\arrow["{\mathfrak{s}_{T,T}}", from=1-6, to=3-6]
				\arrow[""{name=0, anchor=center, inner sep=0}, "{\mathfrak{s}_{\tilde T \circ T , T }}"', from=2-1, to=3-3]
				\arrow["{\tilde \mu*\tilde T}"', from=2-1, to=4-1]
				\arrow["{\tilde{\tilde T * \mu}}", from=2-4, to=3-6]
				\arrow["{\tilde{\bar{\mathfrak{a}}_{T,T}*T}}"', from=3-3, to=2-4]
				\arrow["{\tilde{\tilde\mu* T}}", from=3-3, to=4-4]
				\arrow["{\tilde\mu}", from=3-6, to=4-6]
				\arrow["{\mathfrak{s}_{T,T}}"', from=4-1, to=4-4]
				\arrow["{\tilde\mu}"', from=4-4, to=4-6]
				\arrow["\iso", shorten <=8pt, shorten >=8pt, Rightarrow, from=1-3, to=0]
			\end{tikzcd}
			\]
		\end{enumerate}
	\end{definition}
	
	\begin{theorem}\label{thm:tilde-t-is-a-pseudomonad}
		The triple $\braket{\tilde T , \eta^\sharp, \mu^\sharp}$, together with the modifications $\mathfrak{p}$, $\mathfrak{q}$, and $\mathfrak{x}$, is a pseudomonad on $\CAT$.
		\begin{proof}
			The two coherence conditions spelled out in \cite{lack-pseudomonads} are satisfied since all arrows describing them are uniquely induced by the universal properties of oplax colimits.
		\end{proof}
	\end{theorem}
	
	\section{Preservation of algebras}\label{sec:algebras}
	In the previous section, we constructed the pseudomonad $\braket{\tilde T,\eta^\sharp,\mu^\sharp}$ obtained by unrelativising a $J$-relative 2-monad $\braket{T , \eta, (-)^*}$. This section is dedicated to proving that their colax algebras coincide. Before formally stating the theorem, we here recall the definitions regarding colax algebras for a pseudomonad on $\CAT$. For practical purposes, we restrict ourselves to the context of the pseudomonad $\tilde T$: that is, in the equational axioms, we will assume that the underlying pseudofunctor is actually a 2-functor rather than a pseudofunctor, and that the unit and the multiplication are strict transformations rather than pseudonatural transformations. For a more general treatment, we refer the reader to \cite{stepan-thesis}.
	
	\begin{definition}[\protect{\cite[\S 2]{street74}, \cite[\S 2.4]{stepan-thesis}}]\label{def:colax-algebras-for-pseudomonad}
		A \emph{colax algebra} for $\tilde T$ is a tuple $\braket{C, \Phi, \mathfrak{i},\mathfrak{m}}$ consisting of:
		\begin{enumerate}
			\item a category $C$;
			\item a functor $\Phi \colon \tilde T C \to C$;
			\item a \emph{counitor} $\mathfrak{i} \colon \Phi \circ \eta_C^\sharp \implies \id_C$;
			\item a \emph{coassociator} $\mathfrak{m} \colon \Phi \circ \mu^\sharp_C \implies \Phi \circ \tilde T \Phi $,
		\end{enumerate}
		such that:
		\begin{enumerate}
			\item[(a)]both of the following pastings coincide with the identity transformation on $\Phi$:
			\[ \begin{tikzcd}
				{\tilde T C} &&&& {\tilde T C} \\
				&& {\tilde T ^ 2 C} \\
				&& {\tilde T C} \\
				C &&&& {C,}
				\arrow[""{name=0, anchor=center, inner sep=0}, equals, from=1-1, to=1-5]
				\arrow["{{\eta^\sharp_{\tilde T C}}}"'
				, from=1-1, to=2-3]
				\arrow["\Phi"', from=1-1, to=4-1]
				\arrow["\Phi", from=1-5, to=4-5]
				\arrow[""{name=1, anchor=center, inner sep=0}, "{{\mu^\sharp_C}}"{description}, from=2-3, to=1-5]
				\arrow["{{\tilde T \Phi}}"', from=2-3, to=3-3]
				\arrow[""{name=2, anchor=center, inner sep=0}, "\Phi"{description}, from=3-3, to=4-5]
				\arrow["{{\eta^\sharp_C}}"
				, from=4-1, to=3-3]
				\arrow[""{name=3, anchor=center, inner sep=0}, equals, from=4-1, to=4-5]
				\arrow["{{\mathfrak{p}_C^{-1}}}", shorten <=3pt, shorten >=3pt, Rightarrow, from=0, to=2-3]
				\arrow["{{\mathfrak{m}}}", shorten <=9pt, shorten >=9pt, Rightarrow, from=1, to=2]
				\arrow["{{\mathfrak{i}}}", shorten <=2pt, shorten >=3pt, Rightarrow, from=3-3, to=3]
			\end{tikzcd} \qquad \begin{tikzcd}
				&&& {\tilde T C} \\
				{\tilde T C} && {\tilde T ^2 C} && {C;} \\
				&&& {\tilde T C}
				\arrow["\Phi", from=1-4, to=2-5]
				\arrow["{{{{\mathfrak{m}}}}}"', shorten <=6pt, shorten >=6pt, Rightarrow, from=1-4, to=3-4]
				\arrow[""{name=0, anchor=center, inner sep=0}, curve={height=-20pt}, equals, from=2-1, to=1-4]
				\arrow["{{{{\tilde T \eta^\sharp_C}}}}"{description}, from=2-1, to=2-3]
				\arrow[""{name=1, anchor=center, inner sep=0}, curve={height=20pt}, equals, from=2-1, to=3-4]
				\arrow["{{{{\mu^\sharp_C}}}}"{description}, from=2-3, to=1-4]
				\arrow["{{{{\tilde T \Phi}}}}"{description}, from=2-3, to=3-4]
				\arrow["\Phi"', from=3-4, to=2-5]
				\arrow["{{{{\mathfrak{q}_C}}}}"', shorten <=4pt, Rightarrow, from=0, to=2-3]
				\arrow["{{{{\tilde T \mathfrak{i}}}}}"', shorten <=4pt, shorten >=4pt, Rightarrow, from=2-3, to=1]
			\end{tikzcd} \] 
			\item[(b)]the following pastings coincide:
			\[ \begin{tikzcd}[column sep = 14pt, row sep = 20 pt]
				& {\tilde T^2C} && {\tilde T C} \\
				{\tilde T ^3 C} && {\tilde T ^2 C} && {C,} \\
				& {\tilde T ^2 C} && {\tilde T C}
				\arrow["{\mu^\sharp_C}", from=1-2, to=1-4]
				\arrow["{\mathfrak{x}_C^{-1}}", shorten <=4pt, shorten >=4pt, Rightarrow, from=1-2, to=2-3]
				\arrow["\Phi", from=1-4, to=2-5]
				\arrow["{\mathfrak{m}}", shorten <=6pt, shorten >=6pt, Rightarrow, from=1-4, to=3-4]
				\arrow["{\mu^\sharp_{\tilde T C}}", from=2-1, to=1-2]
				\arrow["{\tilde T \mu^\sharp_C}"{description}, from=2-1, to=2-3]
				\arrow["{\tilde T^2 \Phi}"', from=2-1, to=3-2]
				\arrow["{ \mu^\sharp_C }"{description}, from=2-3, to=1-4]
				\arrow["{\tilde T \mathfrak{m}}", shorten <=4pt, shorten >=4pt, Rightarrow, from=2-3, to=3-2]
				\arrow["{\tilde T \Phi}"{description}, from=2-3, to=3-4]
				\arrow["{\tilde T \Phi}"', from=3-2, to=3-4]
				\arrow["\Phi"', from=3-4, to=2-5]
			\end{tikzcd} \qquad \begin{tikzcd}[column sep = 14pt, row sep = 20 pt]
				& {\tilde T^2C} && {\tilde T C} \\
				{\tilde T ^3 C} && {\tilde T C} && {C.} \\
				& {\tilde T ^2 C} && {\tilde T C}
				\arrow["{\mu^\sharp_C}", from=1-2, to=1-4]
				\arrow["{\tilde T \Phi}"'
				, from=1-2, to=2-3]
				\arrow["{\mathfrak{m}}"', shorten <=4pt, shorten >=4pt, Rightarrow, from=1-4, to=2-3]
				\arrow["\Phi", from=1-4, to=2-5]
				\arrow["{\mu^\sharp_{\tilde T C}}", from=2-1, to=1-2]
				\arrow["{\tilde T^2 \Phi}"', from=2-1, to=3-2]
				\arrow["\Phi"{description}, from=2-3, to=2-5]
				\arrow["{\mathfrak{m}}"', shorten <=4pt, shorten >=4pt, Rightarrow, from=2-3, to=3-4]
				\arrow["{\mu^\sharp_C}"
				, from=3-2, to=2-3]
				\arrow["{\tilde T \Phi}"', from=3-2, to=3-4]
				\arrow["\Phi"', from=3-4, to=2-5]
			\end{tikzcd}\]
		\end{enumerate}
		
		A \emph{colax morphism} $\braket{U, \mathfrak{u}}\colon \braket{C, \Phi, \mathfrak{i}, \mathfrak{m}} \to \braket{C', \Phi', \mathfrak{i}', \mathfrak{m}'}$ of colax algebras for $\tilde T$ consists of:
		\begin{enumerate}
			\item a functor $U \colon C \to C'$;
			\item a natural transformation $\mathfrak{u} \colon U \circ \Phi \implies \Phi' \circ \tilde T U$,
		\end{enumerate}
		such that:
		\begin{enumerate}
			\item[(a)]the following pastings coincide:
			\[ \begin{tikzcd}[row sep = 28pt]
				& {\tilde T C} && C \\
				C & {\tilde T C'} && {C',} \\
				{C'}
				\arrow["\Phi", from=1-2, to=1-4]
				\arrow["{{\tilde T U}}"'
				, from=1-2, to=2-2]
				\arrow["{{\mathfrak{u}}}"', shorten <=9pt, shorten >=9pt, Rightarrow, from=1-4, to=2-2]
				\arrow["U", from=1-4, to=2-4]
				\arrow["{{\eta_C^\sharp}}", from=2-1, to=1-2]
				\arrow["U"', from=2-1, to=3-1]
				\arrow["{\Phi'}"{description}, from=2-2, to=2-4]
				\arrow["{{\eta_{C'}^\sharp}}"{description}, from=3-1, to=2-2]
				\arrow[""{name=0, anchor=center, inner sep=0}, curve={height=20pt}, equals, from=3-1, to=2-4]
				\arrow["{{\mathfrak{i}'}}", shorten <=4pt, shorten >=4pt, Rightarrow, from=2-2, to=0]
			\end{tikzcd} \qquad \begin{tikzcd}[row sep = 28pt]
				& {\tilde T C} && C \\
				C &&& {C'} \\
				{C';}
				\arrow["\Phi", from=1-2, to=1-4]
				\arrow["U", from=1-4, to=2-4]
				\arrow["{{\eta_C^\sharp}}", from=2-1, to=1-2]
				\arrow[""{name=0, anchor=center, inner sep=0}, curve={height=20pt}, equals, from=2-1, to=1-4]
				\arrow["U"', from=2-1, to=3-1]
				\arrow[curve={height=20pt}, equals, from=3-1, to=2-4]
				\arrow["{{\mathfrak{i}}}"', shorten <=4pt, shorten >=4pt, Rightarrow, from=1-2, to=0]
			\end{tikzcd}\]
			\item[(b)]the following pastings coincide:
			\[ \begin{tikzcd}[column sep = 14pt, row sep = 20 pt]
				& {\tilde T C} && C \\
				{\tilde T^2 C} && {\tilde T C} && {C',} \\
				& {\tilde T^2 C'} && {\tilde T C'}
				\arrow["\Phi", from=1-2, to=1-4]
				\arrow["{{\mathfrak{m}}}", shorten <=4pt, shorten >=4pt, Rightarrow, from=1-2, to=2-3]
				\arrow["U", from=1-4, to=2-5]
				\arrow["{{{\mathfrak{u}}}}", shorten <=6pt, shorten >=6pt, Rightarrow, from=1-4, to=3-4]
				\arrow["{{\mu^\sharp_C}}", from=2-1, to=1-2]
				\arrow["{{\tilde T \Phi}}"{description}, from=2-1, to=2-3]
				\arrow["{{\tilde T ^2 U}}"', from=2-1, to=3-2]
				\arrow["\Phi"{description}, from=2-3, to=1-4]
				\arrow["{{\tilde T \mathfrak{u}}}", shorten <=4pt, shorten >=4pt, Rightarrow, from=2-3, to=3-2]
				\arrow["{{{\tilde T U}}}"{description}, from=2-3, to=3-4]
				\arrow["{{\tilde T \Phi'}}"', from=3-2, to=3-4]
				\arrow["{{{\Phi'}}}"', from=3-4, to=2-5]
			\end{tikzcd} \quad \begin{tikzcd}[column sep = 14pt, row sep = 20 pt]
				& {\tilde T C} && C \\
				{\tilde T^2 C} && {\tilde T C'} && {C'.} \\
				& {\tilde T^2 C'} && {\tilde T C'}
				\arrow["\Phi", from=1-2, to=1-4]
				\arrow["{{\tilde T U}}"'
				, from=1-2, to=2-3]
				\arrow["{{\mathfrak{u}}}"', shorten <=4pt, shorten >=4pt, Rightarrow, from=1-4, to=2-3]
				\arrow["U", from=1-4, to=2-5]
				\arrow["{{\mu^\sharp_C}}", from=2-1, to=1-2]
				\arrow["{{\tilde T ^2 U}}"', from=2-1, to=3-2]
				\arrow["{{\Phi'}}"{description}, from=2-3, to=2-5]
				\arrow["{{\mathfrak{m}'}}"', shorten <=4pt, shorten >=4pt, Rightarrow, from=2-3, to=3-4]
				\arrow["{{\mu^\sharp_{C'}}}"
				, from=3-2, to=2-3]
				\arrow["{{\tilde T \Phi'}}"', from=3-2, to=3-4]
				\arrow["{{{\Phi'}}}"', from=3-4, to=2-5]
			\end{tikzcd} \] 
		\end{enumerate}
		A \emph{lax morphism} is defined analogously, reversing the direction of $\mathfrak{u}$. In particular, a (co)lax morphism $\braket{U, \mathfrak{u}}$ is a \emph{pseudomorphism} if $\mathfrak{u}$ is invertible, and it is a \emph{strict morphism} if $\mathfrak{u}$ is the identity.
		
		Let $\braket{U, \mathfrak{u}}, \braket{U', \mathfrak{u}'} \colon \braket{C, \Phi, \mathfrak{i}, \mathfrak{m}} \to \braket{C', \Phi', \mathfrak{i}', \mathfrak{m}'}$ be colax morphisms of colax algebras for $\tilde T$. A \emph{transformation} $\braket{U, \mathfrak{u}} \implies \braket{U', \mathfrak{u}'}$ is a natural transformation $\alpha \colon U \implies U'$ such that the following pastings coincide:
		\[ \begin{tikzcd}
			{\tilde T C} && {\tilde T C'} \\
			\\
			C && {C'}
			\arrow[""{name=0, anchor=center, inner sep=0}, "{{\tilde T U}}"{description}, curve={height=18pt}, from=1-1, to=1-3]
			\arrow[""{name=1, anchor=center, inner sep=0}, "{{\tilde T U'}}", curve={height=-18pt}, from=1-1, to=1-3]
			\arrow["\Phi"', from=1-1, to=3-1]
			\arrow["{{\Phi'}}", from=1-3, to=3-3]
			\arrow["{{\mathfrak{u}}}"', shorten <=15pt, shorten >=15pt, Rightarrow, from=3-1, to=1-3, shift right=3pt]
			\arrow["U"', from=3-1, to=3-3]
			\arrow["{{\tilde T \alpha}}", shorten <=4pt, shorten >=4pt, Rightarrow, from=0, to=1]
		\end{tikzcd} \qquad \begin{tikzcd}
			{\tilde T C} && {\tilde T C'} \\
			\\
			C && {C'}
			\arrow["{{\tilde T U'}}", from=1-1, to=1-3]
			\arrow["\Phi"', from=1-1, to=3-1]
			\arrow["{{\Phi'}}", from=1-3, to=3-3]
			\arrow["{{\mathfrak{u}'}}", shorten <=15pt, shorten >=15pt, Rightarrow, from=3-1, to=1-3, shift left = 3pt]
			\arrow[""{name=0, anchor=center, inner sep=0}, "{{U'}}"{description}, curve={height=-18pt}, from=3-1, to=3-3]
			\arrow[""{name=1, anchor=center, inner sep=0}, "U"', curve={height=18pt}, from=3-1, to=3-3]
			\arrow["\alpha", shorten <=5pt, shorten >=5pt, Rightarrow, from=1, to=0]
		\end{tikzcd} \]
		Transformations between lax morphisms are defined analogously. 
		
		
		Colax algebras, their morphisms and transformations form 2-categories:
		\begin{enumerate}
			\item $\CoLaxAlg_{co}(\tilde T)$, with colax morphisms;
			\item $\CoLaxAlg_{lax}(\tilde T)$, with lax morphisms;
			\item $\CoLaxAlg_{ps}(\tilde T)$, with pseudomorphisms;
			\item $\CoLaxAlg_{str}(\tilde T)$, with strict morphisms.
		\end{enumerate} 
	\end{definition}
	
	\begin{notation}
		As in Notation \ref{nota:omit_morphisms}, we will write $\CoLaxAlg_\star(\tilde T)$ for $\star \in \set{co , lax, ps, str}$ to refer to any of the above 2-categories.
	\end{notation}
	
	\begin{remark}
            Analogous to Example \ref{ex:relative_over_identity}, the pseudomonad $\braket{\tilde T, \eta^\sharp, \mu^\sharp}$ can be viewed as a relative pseudomonad over the identity $\id_\CAT$ in the sense of \cite{pseudomonads-fiore}.  Thus, there is another notion for the colax algebras of $\tilde T$ given by generalising Definition \ref{def:colax-algebra} to the relative pseudomonad setting, (see \cite{arkor-saville-slattery} or \cite[\S 4]{marmolejo-wood-pseudomonads}).  The resultant 2-category of colax algebras is, however, bi-equivalent to the 2-category $\CoLaxAlg(\tilde T)$ described above (cfr. \cite[Theorem 5.1]{marmolejo-wood-pseudomonads}).
	\end{remark}
	
	\begin{theorem}[Unrelativisation preserves colax algebras]\label{thm:algebras_preserved}
		Given a $J$-relative 2-monad $T$ on $\CAT$ satisfying the assumptions of Section \ref{sec:absolution}, for each $\star \in \set{co , lax, ps, str}$ the 2-categories $\CoLaxAlg^J_\star(T)$ and $\CoLaxAlg_\star(\tilde T)$ are isomorphic.
	\end{theorem}
	
	To prove Theorem \ref{thm:algebras_preserved}, we will make extensive use of the universal property of $\tilde T$ as a left oplax Kan extension. In particular, we will show that the datum of a colax algebra for $\braket{T , \eta, (-)^*}$ on a category $C$ coincides with that of a colax algebra for $\braket{\tilde T,\eta^\sharp,\mu^\sharp}$ on the same category $C$. Indeed -- recalling Example \ref{ex:extension-operators-define-lax-trans} and the discussion before Definition \ref{def:left-oplax-kan-extension} -- the former consists of a lax transformation $\funcat{ J - }{C} \implies \funcat{ T -}{C}$ with extra structure, while the latter consists of a functor $\tilde T C \to C$ with extra structure.  By Proposition \ref{prop:left-oplax-Kan-extension-in-CAT}, such lax transformations and such functors are in bijection, hence Theorem \ref{thm:algebras_preserved} will follow by showing that the rest of the structure is also interdefinable.  We present the construction for colax morphisms, but this is easily adapted.

	\subsection{ \texorpdfstring{From $T$-algebras to $\tilde T$-algebras}{From relative algebras to algebras for the pseudomonad} }To witness the isomorphisms in Theorem \ref{thm:algebras_preserved}, we first build a 2-functor $\Theta \colon \CoLaxAlg_{co}^J(T) \to \CoLaxAlg_{co}(\tilde T)$.
	
	\subsubsection{On algebras}\label{sssec:rel-to-abs_on-algebras}Let $\braket{ C, (-)^C, \Gamma, \Delta}$ be a colax algebra for the $J$-relative 2-monad $\braket{T, \eta, (-)^*}$. We define a colax algebra 
	\[ \Theta (\braket{ C, (-)^C, \Gamma, \Delta}) \coloneqq \braket{C, \Phi, \mathfrak{i}, \mathfrak{m} }\]
	for the pseudomonad $\braket{\tilde T , \eta^\sharp, \mu^\sharp}$ as follows.
	
	By Example \ref{ex:extension-operators-define-lax-trans}, recall that the family of extension operators defines a lax transformation $(-)^C \colon \funcat{ J - }{C} \implies \funcat{ T -}{C}$, so we let the structure functor $\Phi \colon \tilde T C \to C$ be the uniquely induced functor by Proposition \ref{prop:left-oplax-Kan-extension-in-CAT}. Explicitly, this means that $\Phi$ is defined
	\begin{enumerate}
		\item on objects by mapping $(b,h \colon J b \to C, \nu)\in\tilde T C$ to $h^C(\nu)$,
		\item and on arrows by mapping $(f,\alpha,\phi) \colon ( b , h , \nu) \to (b', h',\nu') \in \tilde T C$ to the composite
		\[ \begin{tikzcd}[column sep = 18 pt]
			{h^C(\nu)} & {h^CTf(\nu')} && {(h^C\eta_bJf)^C(\nu')} && {(hJf)^C(\nu')} & {h'^C(\nu') . }
			\arrow["{{h^C(\phi)}}", from=1-1, to=1-2]
			\arrow["{{(\Delta_{h, \eta_bJf})_{\nu'}}}", from=1-2, to=1-4]
			\arrow["{{(\Gamma_h*Jf)^C_{\nu'}}}", from=1-4, to=1-6]
			\arrow["{{\alpha^C_{\nu'}}}", from=1-6, to=1-7]
		\end{tikzcd} \]
	\end{enumerate}
	
	Next we define the counitor $\mathfrak{i} \colon \Phi \circ \eta_C^\sharp \implies \id_C$. First note that $\Phi \circ \eta_C^\sharp \colon C \to C$ is the functor
	\[ \begin{tikzcd}
		c & {c^C(\eta_{1_{\Bcat}}(\bullet))} \\
		{c'} & {c'^C(\eta_{1_{\Bcat}}(\bullet))}
		\arrow[maps to, from=1-1, to=1-2]
		\arrow["f"', from=1-1, to=2-1]
		\arrow["{f^C_{\eta_{1_{\Bcat}}(\bullet)}}", from=1-2, to=2-2]
		\arrow[maps to, from=2-1, to=2-2]
	\end{tikzcd} \]
	where we identify $c,c'\in C$ with functors $J 1_{\Bcat} = 1_{\CAT} \to C$ and hence $f \colon c \to c' \in C$ with a natural transformation $c \implies c'$. Therefore, we define the component of $\mathfrak{i}$ at an object $c \in C$ as the arrow 
	\[ \begin{tikzcd}
		{(\Phi\circ \eta_C^\sharp)(c)  = c^C(\eta_{1_{\Bcat}}(\bullet))} && c
		\arrow["{(\Gamma_c)_\bullet}", from=1-1, to=1-3]
	\end{tikzcd} \]
	in $C$. Naturality of $\mathfrak{i}$ follows immediately by that of the family $\set{ \Gamma }$.
	
	Finally, we define the coassociator $\mathfrak{m} \colon \Phi \circ \mu^\sharp_C \implies \Phi \circ \tilde T \Phi $. First note that the functor $\Phi \circ \mu^\sharp_C \colon \tilde T ^2 C \to C$ acts:
	\begin{enumerate}
		\item on objects, by mapping $(b, h \colon J b \to \tilde T C, \nu) \in \tilde T^2 C$ to $a^C Q^* (\nu) \in C$, where we let $(\ell, a \colon J \ell \to C, ( b , Q \colon J b \to T \ell, \nu)) \coloneqq \mathfrak{s}_{T,T}(b,h,\nu)$ as in \ref{ssec:structure-on-Bcat-to-CAT}; 
		\item on arrows, by mapping $(f , \alpha , \phi ) \colon ( b ,h , \nu) \to (b', h', \nu') \in \tilde T ^2 C$ to the composite 
		\[ \begin{tikzcd}[column sep = large]
			{a^CQ^*(\nu)} \\
			{a^CT \bar f (Q'^*\nu')} && {(a^C\eta_\ell J \bar f )^C(Q'^*\nu')} && {(a J \bar f )^C(Q'^*\nu')} \\
			&&&& {a'^CQ'^*(\nu') , }
			\arrow["{{a^C(\theta^*_{\nu'}\circ Q^*(\phi))}}", from=1-1, to=2-1]
			\arrow["{{(\Delta_{a, \eta J \bar f }*Q'^*)_{\nu'}}}", from=2-1, to=2-3]
			\arrow["{{((\Gamma_a*J\bar f )^C*Q'^*)_{\nu'}}}", from=2-3, to=2-5]
			\arrow["{{(\bar \alpha ^C*Q'^*)_{\nu'}}}", from=2-5, to=3-5]
		\end{tikzcd} \]
		where again we let $(\bar f , \bar \alpha, ( f, \theta, \phi)) \coloneqq \mathfrak{s}_{T,T} ( f, \alpha, \phi)$ as in \ref{ssec:structure-on-Bcat-to-CAT}.
	\end{enumerate}
	Moreover, note that for each object $(b , h \colon J b \to \tilde T C, \nu) \in \tilde T ^2 C$, the family of arrows
	\[ \begin{tikzcd}[column sep = 18 pt]
		{a^CQ(x) = a^CTc_x(\nu_x)} && {(a^C\eta_{\ell}Jc_x)^C(\nu_x)} && {(aJc_x)^C(\nu_x) = a_x^C(\nu_x)}
		\arrow["{{(\Delta_{a, \eta J c_x})_{\nu_x}}}", from=1-1, to=1-3]
		\arrow["{{(\Gamma_a*Jc_x)^C_{\nu_x}}}", from=1-3, to=1-5]
	\end{tikzcd} \]
	for $x \in Jb$ defines a natural transformation $\tau \colon a^C \circ Q \implies \Phi \circ h$, where as usual we denote $h(x)$ as $(Rx, a_x\colon J R x \to C, \nu_x) $ and hence $\Phi ( h ( x) ) = a_x^C(\nu_x)$. Therefore, we define the component of $\mathfrak{m}$ at $(b,h,\nu)$ as the composite
	\[ \begin{tikzcd}[column sep = 10pt]
		{(\Phi\circ \mu^\sharp_C)(b,h,\nu) = a^CQ^*(\nu)} && {(a^CQ)^C(\nu)} && {(\Phi h )^C(\nu) = (\Phi\circ \tilde T \Phi)(b,h,\nu)}
		\arrow["{{(\Delta_{a,Q})_\nu}}", from=1-1, to=1-3]
		\arrow["{{\tau^C_\nu}}", from=1-3, to=1-5]
	\end{tikzcd} \]
	in $C$. We omit the routine proof of the naturality of $\mathfrak{m}$. The proof of the fact that $\braket{C, \Phi, \mathfrak{i},\mathfrak{m}}$ is a colax algebra for $\tilde T$ is postponed to Lemma \ref{lem:rel-to-abs_is-algebra}.
	
	\subsubsection{On morphisms}\label{sssec:rel-to-abs_on-morphisms}
	Let $\braket{U, \upsilon}\colon \braket{C,(-)^C, \Gamma, \Delta} \to \braket{C',(-)^{C'}, \Gamma', \Delta'}$ be a colax morphism of colax algebras for $\braket{T, \eta, (-)^*}$. The same underlying functor $U \colon C \to C'$ yields a colax morphism
	\[ \Theta ( \braket{U, \upsilon} ) \coloneqq \braket{U, \mathfrak{u} } \colon \braket{C, \Phi, \mathfrak{i},\mathfrak{m} } \to \braket{C',\Phi', \mathfrak{i}',\mathfrak{m}' }  \]
	of colax algebras for $\braket{\tilde T , \eta^\sharp,\mu^\sharp}$ by taking the natural transformation $\mathfrak{u} \colon U \circ \Phi \implies \Phi' \circ \tilde T U$ defined at an object $(b,h\colon J b \to C,\nu) \in \tilde T C$ by the arrow
	\[ \begin{tikzcd}
		{(U\circ \Phi) ( b, h ,\nu) = U(h^C(\nu))} & {(Uh)^C(\nu) = (\Phi'\circ \tilde T U)(b,h,\nu)}
		\arrow["{{(\upsilon_h)_\nu}}", from=1-1, to=1-2]
	\end{tikzcd} \]
	in $C'$. Naturality of $\mathfrak{u}$ follows immediately by that of the family $\set{ \upsilon }$. A proof of the fact that $\braket{U, \mathfrak{u}}$ is a colax morphism is given in Lemma \ref{lem:rel-to-abs_is-morphism}. With this definition, $\Theta$ is clearly a functor $\CoLaxAlg_{co}^J(T) \to \CoLaxAlg_{co}(\tilde T)$. 
	
	\subsubsection{On transformations}Let $\braket{U, \upsilon}, \braket{U', \upsilon'}\colon \braket{C,(-)^C, \Gamma, \Delta} \to \braket{C',(-)^{C'}, \Gamma', \Delta}$ be colax morphisms of colax algebras for $\braket{T, \eta, (-)^*}$ and let $\alpha \colon \braket{U, \upsilon} \implies \braket{U', \upsilon'}$ be a transformation. Let us show that $\alpha$ is also a transformation $\braket{U, \mathfrak{u}}\implies \braket{U', \mathfrak{u}'}$, i.e.\ that the diagram of natural transformations
	\[ \begin{tikzcd}[column sep = large]
		{U \circ \Phi} & {U' \circ \Phi} \\
		{\Phi'\circ \tilde T U} & {\Phi'\circ \tilde T U '}
		\arrow["{\alpha * \Phi}", from=1-1, to=1-2]
		\arrow["{\mathfrak{u}}"', from=1-1, to=2-1]
		\arrow["{\mathfrak{u}'}", from=1-2, to=2-2]
		\arrow["{\Phi'*\tilde T \alpha}"', from=2-1, to=2-2]
	\end{tikzcd} \] 
	commutes. Instantiated at an object $(b,h\colon J b \to C, \nu) \in \tilde T C$, this yields the diagram
	\[ \begin{tikzcd}[column sep = large]
		{Uh^C(\nu)} & {U'h^C(\nu)} \\
		{(Uh)^{C'}(\nu)} & {(U'h)^{C'}(\nu)}
		\arrow["{(\alpha*h^C)_\nu}", from=1-1, to=1-2]
		\arrow["{(\upsilon_h)_\nu}"', from=1-1, to=2-1]
		\arrow["{(\upsilon'_h)_\nu}", from=1-2, to=2-2]
		\arrow["{(\alpha*h)^{C'}_\nu}"', from=2-1, to=2-2]
	\end{tikzcd} \] 
	in $C'$, which commutes since $\alpha$ is a transformation $ \braket{U, \upsilon} \implies \braket{U', \upsilon'}$. Therefore, we can extend $\Theta$ to a 2-functor $\CoLaxAlg_{co}^J(T) \to \CoLaxAlg_{co}(\tilde T)$ simply by setting:
	\[ \Theta ( \alpha ) \coloneqq \alpha \colon \braket{U, \mathfrak{u}}\implies \braket{U', \mathfrak{u}'}\]
	
	\begin{remark}\label{rem:theta_respects_lax_morphisms}
		It is evident how to define $\Theta$ on lax morphisms and their transformations, and moreover that it preserves pseudomorphisms and strict morphisms. Therefore, we have defined a 2-functor $\Theta \colon \CoLaxAlg_\star^J(T) \to \CoLaxAlg_\star(\tilde T)$ for each $\star \in \set{ co, lax, ps, str}$.
	\end{remark}
	
	\subsection{ \texorpdfstring{From $\tilde T$-algebras to $T$-algebras}{From algebras for the pseudomonad to relative algebras} }We now build a 2-functor $\Sigma \colon \CoLaxAlg_{co}(\tilde T )\to \CoLaxAlg_{co}^J(T)$. Although we will not make use of this intuition, $\Sigma$ can be thought of as being induced by $\zeta \colon T \implies \tilde T \circ S$ acting as a `morphism of relative pseudomonads' where we allow the root 2-functor to change -- a notion which does not yet appear in the literature.
	
	\subsubsection{On algebras}\label{sssec:abs-to-rel_on-algebras}Let $\braket{C, \Phi, \mathfrak{i}, \mathfrak{m}}$ be a colax algebra for the pseudomonad $\braket{\tilde T , \eta^\sharp, \mu^\sharp}$. We define a colax algebra
	\[ \Sigma ( \braket{C, \Phi, \mathfrak{i}, \mathfrak{m}} ) \coloneqq \braket{C, (-)^C, \Gamma,\Delta} \]
	for the $J$-relative 2-monad $\braket{T,\eta,(-)^*}$ as follows.
	
	We define the extension operator $(-)^C \colon \funcat{J-}{C} \implies \funcat{T-}{C}$ as the unique lax transformation corresponding to the functor $\Phi \colon \tilde T C \to C$, that is:
	\[ \begin{tikzcd}
		{\funcat{J-}{C}} && {\funcat{T-}{\tilde T C}} && {\funcat{T-}{C} .}
		\arrow["{\iota_{T,C}}", from=1-1, to=1-3]
		\arrow["{\Phi \circ-}", from=1-3, to=1-5]
	\end{tikzcd} \]
	
	Next we define a natural transformation $\Gamma_h \colon h^C \circ \eta_b \implies h$ for each functor $h \colon J b \to C$. First note that $h^C \circ \eta_b \colon J b \to C$ is the functor
	\[ \begin{tikzcd}
		x & {(b, h \colon J b \to C, \eta_b(x))} \\
		{x'} & {(b, h\colon J b \to C, \eta_b(x')) \ . }
		\arrow[maps to, from=1-1, to=1-2]
		\arrow["g"', from=1-1, to=2-1]
		\arrow["{(\id_b,\id,\eta_b(g))}", from=1-2, to=2-2]
		\arrow[maps to, from=2-1, to=2-2]
	\end{tikzcd} \] 
	Therefore, we define $\Gamma_h$ on an object $x \in Jb$ -- identified with a functor $1_{\CAT} \to Jb$ -- as the composite
	\[\begin{tikzcd}
		{\Phi (b,h\colon J b \to C, \eta_b(x))} && {\Phi(1_{\Bcat},hx \colon 1_{\CAT}\to C, \eta_{1_{\Bcat}}(\bullet))} && {h(x)}
		\arrow["{{\Phi(\dot x, \id, \id)}}", from=1-1, to=1-3]
		\arrow["{\mathfrak{i}_{h(x)}}", from=1-3, to=1-5]
	\end{tikzcd} \] 
	in $C$, where $\dot x \colon 1_{\Bcat} \to b$ is the unique 1-cell in $\Bcat$ such that $J \dot x = x$, by assumption of 2-fully-faithfulness of $J$. We omit the routine proofs of the naturality of each $\Gamma_h$ and of the family $\set{\Gamma}$.
	
	Finally, we define a natural transformation $\Delta_{h,k}\colon h^C \circ k^* \implies ( h^C \circ k )^C$ for each pair of functors $h\colon J b\to C$ and $k \colon J c \to T b$. First note that $h^C \circ k^*$ and $(h^C\circ k)^C $ are the functors $T c \to C$ respectively given by
	\[ \begin{tikzcd}
		\nu & {\Phi(b, h \colon J b \to C, k^*(\nu))} \\
		{\nu'} & {\Phi(b, h \colon J b \to C , k^*(\nu'))}
		\arrow[maps to, from=1-1, to=1-2]
		\arrow["\phi"', from=1-1, to=2-1]
		\arrow["{\Phi(\id,\id,k^*(\phi))}", from=1-2, to=2-2]
		\arrow[maps to, from=2-1, to=2-2]
	\end{tikzcd} \mbox{ and } \begin{tikzcd}
		\nu & {\Phi(c, h^Ck\colon J c \to C, \nu)} \\
		{\nu'} & {\Phi(c, h^Ck\colon J c \to C, \nu') \ .}
		\arrow[maps to, from=1-1, to=1-2]
		\arrow["\phi"', from=1-1, to=2-1]
		\arrow["{\Phi(\id,\id,\phi)}", from=1-2, to=2-2]
		\arrow[maps to, from=2-1, to=2-2]
	\end{tikzcd}\]
	Therefore, we define $\Delta_{h,k}$ on an object $\nu\in T c$ as the composite:
	\[ \begin{tikzcd}[column sep = small]
		{\Phi(b, h \colon J b \to C, k^*(\nu))} \\
		{\Phi(\ell, a \colon J \ell \to C, Q^*(\nu))} &[-19pt] {=(\Phi\circ \mu^\sharp_C)(c, \iota_b(h)k\colon J c \to \tilde T C, \nu)} \\
		& {(\Phi\circ\tilde T \Phi) (c, \iota_b(h)k\colon J c \to \tilde T C, \nu)} &[-22pt] {=\Phi ( c, h^Ck\colon J c \to C, \nu)}
		\arrow["{\Phi(t, \id,\id)}", from=1-1, to=2-1]
		\arrow["{\mathfrak{m}_{(c,\iota_b(h)k,\nu)}}", from=2-2, to=3-2]
	\end{tikzcd} \] 
	in $C$, where we let $(\ell, a \colon J \ell \to C, (c , Q \colon J c \to T \ell, \nu)) \coloneqq \mathfrak{s}_{T,T} ( c, \iota_b(h)k , \nu)$ as in \ref{ssec:structure-on-Bcat-to-CAT} and use $t \colon \ell \to b \in \Bcat$ to denote the canonical 1-cell induced being $b$ an oplax cocone on the constant diagram $\op{(Jc)}\to\Bcat$ of value $b$, of which $\ell$ is the oplax colimit. We omit the routine proofs of the naturality of each $\Delta_{h,k}$ and of the family $\set{\Delta}$. The proof of the fact that $\braket{C, (-)^C, \Gamma,\Delta}$ is a colax algebra for $T$ is postponed to Lemma \ref{lem:abs-to-rel_is-algebra}.
	
	\subsubsection{On morphisms}\label{sssec:abs-to-rel_on-morphisms}Let $\braket{U, \mathfrak{u}}\colon \braket{C,\Phi, \mathfrak{i}, \mathfrak{m}} \to \braket{C',\Phi', \mathfrak{i}', \mathfrak{m}'}$ be a colax morphism of colax algebras for $\braket{\tilde T, \eta^\sharp, \mu^\sharp}$. The same underlying functor $U \colon C \to C'$ yields a colax morphism
	\[ \Sigma ( \braket{U, \mathfrak{u}} ) \coloneqq \braket{U, \upsilon } \colon \braket{C, (-)^C, \Gamma, \Delta } \to \braket{C',(-)^{C'}, \Gamma', \Delta'}  \]
	of colax algebras for $\braket{ T , \eta ,(-)^*}$ by taking, for a functor $h \colon J b \to C$, the natural transformation $\upsilon_h \colon U h^C \implies (Uh)^{C'}$ defined at an object $\nu \in T b$ by the arrow
	\[ \begin{tikzcd}[column sep = 20pt]
		{Uh^C(\nu)} &[-27pt] {=(U\circ \Phi)(b,h\colon J b \to C, \nu)} \\
		& {(\Phi \circ \tilde T U )(b, h \colon J b \to C,\nu)} &[-29pt] {=\Phi (b, U h \colon J b \to C',\nu)=(Uh)^{C'}(\nu)}
		\arrow["{{\mathfrak{u}_{(b,h,\nu)}}}", from=1-2, to=2-2]
	\end{tikzcd} \] 
	in $C'$. Naturality of $\upsilon_h$ and of the family $\set{\upsilon}$ follow immediately by that of $\mathfrak{u}$. A proof of the fact that $\braket{U, \upsilon}$ is a colax morphism is given in Lemma \ref{lem:abs-to-rel_is-morphism}. With this definition, $\Sigma$ is clearly a functor $\CoLaxAlg_{co}(\tilde T) \to \CoLaxAlg_{co}^J(T)$. 
	
	\subsubsection{On transformations}Let $\braket{U, \mathfrak{u}}, \braket{U', \mathfrak{u}'} \colon \braket{C,\Phi, \mathfrak{i}, \mathfrak{m}} \to \braket{C',\Phi', \mathfrak{i}', \mathfrak{m}'}$ be colax morphisms of colax algebras for $\braket{\tilde T , \eta^\sharp, \mu^\sharp}$ and let $\alpha \colon \braket{U, \mathfrak{u}} \implies \braket{U', \mathfrak{u}'}$ be a transformation. Let us show that $\alpha$ is also a transformation $\braket{U, \upsilon}\implies \braket{U', \upsilon'}$, i.e.\ that the diagram of natural transformations
	\[ \begin{tikzcd}[column sep = large]
		{Uh^C} & {U'h^C} \\
		{(Uh)^{C'}} & {(U'h)^{C'}}
		\arrow["{\alpha*h^C}", from=1-1, to=1-2]
		\arrow["{\upsilon_h}"', from=1-1, to=2-1]
		\arrow["{\upsilon'_h}", from=1-2, to=2-2]
		\arrow["{(\alpha*h)^C}"', from=2-1, to=2-2]
	\end{tikzcd} \] 
	commutes for each functor $h \colon J b \to C$. Instantiated at an object $\nu \in T b$, this yields the diagram:
	\[ \begin{tikzcd}[column sep = 60pt]
		{(U\circ\Phi)(b,h,\nu)} & {(U'\circ\Phi)(b,h,\nu)} \\
		{(\Phi\circ\tilde T U)(b,h,\nu)} & {(\Phi\circ\tilde T U')(b,h,\nu)}
		\arrow["{(\alpha*\Phi)_{(b,h,\nu)}}", from=1-1, to=1-2]
		\arrow["{\mathfrak{u}_{(b,h,\nu)}}"', from=1-1, to=2-1]
		\arrow["{\mathfrak{u}'_{(b,h,\nu)}}", from=1-2, to=2-2]
		\arrow["{(\Phi*\tilde T \alpha)_{(b,h,\nu)}}"', from=2-1, to=2-2]
	\end{tikzcd}
	\] 
	which commutes since $\alpha$ is a transformation $ \braket{U, \mathfrak{u}} \implies \braket{U', \mathfrak{u}'}$. Therefore, we can extend $\Sigma$ to a 2-functor $\CoLaxAlg_{co}(\tilde T) \to \CoLaxAlg_{co}^J( T)$ simply by setting:
	\[ \Sigma ( \alpha ) \coloneqq \alpha \colon \braket{U, \upsilon}\implies \braket{U', \upsilon'}\]
	
	\begin{remark}\label{rem:sigma_respects_lax_morphisms}
		As for $\Theta$, it is evident how to define $\Sigma$ on lax morphisms and their transformations, and moreover that it preserves pseudomorphisms and strict morphisms. Therefore, we have defined a 2-functor $\Sigma \colon  \CoLaxAlg_\star(\tilde T)\to \CoLaxAlg_\star^J(T)$ for each $\star \in \set{ co , lax, ps, str}$
	\end{remark}
	
	\subsection{There and back again}Finally, we complete the proof of Theorem \ref{thm:algebras_preserved} by showing that $\Theta$ and $\Sigma$ are each other's inverse. First, consider the composite $\Sigma \circ \Theta$.
	
	\begin{lemma}
		Let $\braket{C, (-)^C, \Gamma,\Delta}$ be a colax algebra for $\braket{T, \eta, (-)^*}$. Then 
		\[ \Sigma\Theta ( \braket{C, (-)^C, \Gamma,\Delta} ) = ( \braket{C, (-)^C, \Gamma,\Delta} ).\] 
		\begin{proof}
			Denote $\Sigma\Theta ( \braket{C, (-)^C, \Gamma,\Delta} )$ as $\braket{C, (-)^{\bar C}, \bar\Gamma,\bar\Delta}$. Clearly $(-)^{\bar C} = (-)^C$ as they are both the unique lax transformation corresponding to the functor $\Phi \colon \tilde T C \to C$ defining $\braket{C , \Phi, \mathfrak{i}, \mathfrak{m}}\coloneqq \Theta (\braket{C, (-)^C, \Gamma,\Delta})$. 
			
			For a functor $h \colon J b \to C$, the natural transformation $\bar \Gamma_h$ is defined at an object $x\in J b$ as the composite
			\[ \begin{tikzcd}
				{\Phi(b,h, \eta_b(x))} && {\Phi(1_{\Bcat},hx, \eta_{1_{\Bcat}}(\bullet))} && {h(x)}
				\arrow["{{\Phi(\dot x, \id, \id)}}", from=1-1, to=1-3]
				\arrow["{{\mathfrak{i}_{h(x)}}}", from=1-3, to=1-5]
			\end{tikzcd} \] 
			which, using the description of $\Phi$ and $\mathfrak{i}$, is explicitly the composite:
			\[ \begin{tikzcd}[cramped, sep = 16pt]
				{h^C(\eta_b  (x) ) = h^CT\dot x(\eta_{1_{\Bcat}}(\bullet))} && {(h^C\eta_b x )^C(\eta_{1_{\Bcat}}(\bullet))} && {(hx)^C(\eta_{1_{\Bcat}}(\bullet))} & {h(x) .}
				\arrow["{(\Delta_{h,\eta_b x} )_{\eta_{1_{\Bcat}}(\bullet)}}", from=1-1, to=1-3]
				\arrow["{(\Gamma_h*x)^C_{\eta_{1_{\Bcat}}(\bullet)}}", from=1-3, to=1-5]
				\arrow["{{{{(\Gamma_{hx})_\bullet}}}}", from=1-5, to=1-6]
			\end{tikzcd} \] 
			Note then that the diagram:
			\[ \begin{tikzcd}[sep = 22pt]
				{h^C(\eta_b  (x) )= h^CT\dot x(\eta_{1_{\Bcat}}(\bullet))} && {(h^C\eta_b x )^C(\eta_{1_{\Bcat}}(\bullet))} && {(hx)^C(\eta_{1_{\Bcat}}(\bullet))} \\
				&& {(h^C\eta_b x )(\bullet)} && {hx(\bullet)}
				\arrow["{{{{(\Delta_{h,\eta_b x} * \eta_{1_{\Bcat}})_\bullet}}}}", from=1-1, to=1-3]
				\arrow[curve={height=12pt},equals,from=1-1, to=2-3]
				\arrow["{{{{((\Gamma_h*x)^C*\eta_{1_{\Bcat}})_\bullet}}}}", from=1-3, to=1-5]
				\arrow["{{(\Gamma_{h^C\eta_b x})_\bullet}}", from=1-3, to=2-3]
				\arrow["{{{{(\Gamma_{hx})_\bullet}}}}", from=1-5, to=2-5]
				\arrow["{{(\Gamma_h*x)_\bullet}}"', from=2-3, to=2-5]
			\end{tikzcd} \]
			commutes by naturality of the family $\set{\Gamma}$ with respect to $\Gamma_h*x$ and by axiom (b) of a colax algebra (see Definition \ref{def:colax-algebra}). Therefore, $(\bar\Gamma_h)_x = (\Gamma_h)_x$, which entails that $\bar \Gamma = \Gamma$.
			
			For a pair of functors $h \colon J b \to C, k \colon J c \to T b$, the natural transformation $\bar\Delta_{h,k}$ is defined at an object $\nu\in T c$ as the composite
			\[ \begin{tikzcd}
				{\Phi(b, h , k^*(\nu))} \\
				{\Phi(\ell, a , Q^*(\nu))} &[-31pt] { =(\Phi\circ \mu^\sharp_C)(c, \iota_b(h)k, \nu)} \\
				& {(\Phi\circ\tilde T \Phi) (c, \iota_b(h)k, \nu)} &[-35pt] {=\Phi ( c, h^Ck, \nu)}
				\arrow["{{{\Phi(t, \id,\id)}}}", from=1-1, to=2-1]
				\arrow["{{{\mathfrak{m}_{(c,\iota_b(h)k,\nu)}}}}", from=2-2, to=3-2]
			\end{tikzcd} \] 
			where $(\ell, a \colon J \ell \to C, (c , Q \colon J c \to T \ell, \nu)) \coloneqq \mathfrak{s}_{T,T} ( c, \iota_b(h)k , \nu)$ as in \ref{sssec:abs-to-rel_on-algebras} and where $t \colon \ell \to b \in \Bcat$ is the canonical 1-cell induced being $b$ an oplax cocone on the diagram of which $\ell$ is the oplax colimit. Using the explicit description of $\Phi$ and $\mathfrak{m}$, we can rewrite the previous as the composite:
			\[ \begin{tikzcd}[cramped, row sep = 12pt, column sep = 15 pt]
				{h^C(k^*(\nu))} \\
				{h^CTt(Q^*\nu )} && {(h^C\eta_bJt)^C(Q^*\nu )} && {(hJt)^C(Q^*\nu)} \\
				&&&& {a^C(Q^*\nu)} & {(a^CQ)^C(\nu)} & {(h^Ck)^C(\nu)}
				\arrow[equals, from=1-1, to=2-1]
				\arrow["{(\Delta_{h, \eta_bJt})_{Q^*\nu}}", from=2-1, to=2-3]
				\arrow["{(\Gamma_h*Jt)^C_{Q^*\nu}}", from=2-3, to=2-5]
				\arrow[equals, from=2-5, to=3-5]
				\arrow["{{{(\Delta_{a,Q})_\nu}}}", from=3-5, to=3-6]
				\arrow["{{{\tau^C_\nu}}}", from=3-6, to=3-7]
			\end{tikzcd} \] 
			Consider now the diagram:
			\[ \begin{tikzcd}[column sep = small, row sep = 13pt]
				{h^CTt(Q^*(\nu))} &&&& {h^C(k^*(\nu))} \\
				{(h^C\eta_bJt)^C(Q^*(\nu))} \\
				&& {((h^C\eta_bJt)^CQ)^C(\nu)} \\
				{(hJt)^C(Q^*(\nu))} \\
				{a^C(Q^*(\nu))} && {((hJt)^CQ)^C(\nu)} & {(h^C T t Q)^C(\nu)} \\
				{(a^CQ)^C(\nu)} \\
				{(h^Ck)^C(\nu)} &&&& {(h^Ck)^C(\nu)}
				\arrow["{{{(\Delta_{h, \eta_bJt}*Q^*)_\nu}}}"', from=1-1, to=2-1]
				\arrow[""{name=0, anchor=center, inner sep=0}, equals, from=1-5, to=1-1]
				\arrow["{{(\Delta_{h,k})_\nu}}", from=1-5, to=7-5]
				\arrow[""{name=1, anchor=center, inner sep=0}, "{{(\Delta_{h^C\eta_bJt, Q})_\nu}}"{pos=0.35}, from=2-1, to=3-3]
				\arrow["{{{((\Gamma_h*Jt)^C*Q^*)_\nu}}}"', from=2-1, to=4-1]
				\arrow["{{((\Gamma_h*Jt)^C*Q)^C_\nu}}"{description}, from=3-3, to=5-3]
				\arrow[equals, from=4-1, to=5-1]
				\arrow[""{name=2, anchor=center, inner sep=0}, "{{(\Delta_{hJt,Q})_\nu}}"{pos=0.35}, from=4-1, to=5-3]
				\arrow["{{{(\Delta_{a,Q})_\nu}}}"', from=5-1, to=6-1]
				\arrow["{{(\Delta_{h,\eta_bJt}*Q)^C_\nu}}"', from=5-4, to=3-3]
				\arrow[""{name=3, anchor=center, inner sep=0}, equals, from=5-4, to=7-5]
				\arrow[""{name=4, anchor=center, inner sep=0}, equals, from=6-1, to=5-3]
				\arrow["{{{\tau^C_\nu}}}"', from=6-1, to=7-1]
				\arrow[equals, from=7-1, to=7-5]
				\arrow["{{(i)}}"{description}, draw=none, from=0, to=3-3]
				\arrow["{{(ii)}}"{description}, draw=none, from=1, to=2]
				\arrow["{{(iii)}}"{description}, draw=none, from=2, to=4]
				\arrow["{{(iv)}}"{description}, draw=none, from=6-1, to=3]
			\end{tikzcd} \] 
			where (i) commutes by axiom (c) of a colax algebra, (ii) commutes by naturality of the family $\set{\Delta}$ with respect to $\Gamma_h*Jt$, and (iii) commutes since $h \circ J t = a$. To see why (iv) also commutes, note that by functoriality of $(-)^C$ it suffices to show that the diagram of natural transformations
			\[ \begin{tikzcd}[row sep = small]
				{(hJt)^CQ} && {(h^C\eta_bJt)^CQ} \\
				{a^CQ } \\
				{h^Ck} && {h^C T t Q }
				\arrow[equals, from=1-1, to=2-1]
				\arrow["{(\Gamma_h*Jt)^C*Q}"', from=1-3, to=1-1]
				\arrow["{\tau }"', from=2-1, to=3-1]
				\arrow[equals, from=3-1, to=3-3]
				\arrow["{\Delta_{h,\eta_bJt}*Q}"', from=3-3, to=1-3]
			\end{tikzcd} \] 
			commutes. Instantiated at $x\in J c$, this yields the outer border of the diagram 
			\[ \begin{tikzcd}[column sep = 5pt, row sep = 13pt]
				{(hJt)^CTc_x(\nu_x)} &&& {(h^C\eta_bJt)^C(Tc_x(\nu_x))} \\
				{a^CTc_x(\nu_x)} && {((h^C\eta_bJt)^C\eta_\ell Jc_x)^C(\nu_x)} \\
				{(a^C\eta_\ell J c_x)^C(\nu_x)} \\
				& {(h^C\eta_bJt  Jc_x)^C(\nu_x)} \\
				&& {(h^C T t \eta_\ell J c_x)^C (\nu_x)} \\
				{(a  J c_x)^C(\nu_x)} && {h^C (\nu_x)} & {h^C T t T c_x (\nu_x)}
				\arrow[equals, from=1-1, to=2-1]
				\arrow[""{name=0, anchor=center, inner sep=0}, "{((\Gamma_h*Jt)^C*Tc_x)_{\nu_x}}"', from=1-4, to=1-1]
				\arrow["{(\Delta_{h^C\eta_bJt, \eta_{\ell}Jc_x})_{\nu_x}}"', from=1-4, to=2-3]
				\arrow["{(\Delta_{a, \eta_{\ell}Jc_x})_\nu}"', from=2-1, to=3-1]
				\arrow[""{name=1, anchor=center, inner sep=0}, "{((\Gamma_h*Jt)^C*\eta_\ell J c_x)^C_{\nu_x}}"'{pos=0.3}, from=2-3, to=3-1]
				\arrow[""{name=2, anchor=center, inner sep=0}, "{(\Gamma_{h^C\eta_b J t }*Jc_x)^C_{\nu_x}}"{description}, from=2-3, to=4-2]
				\arrow["{(\Gamma_a*Jc_x)^C_{\nu_x}}"', from=3-1, to=6-1]
				\arrow[""{name=3, anchor=center, inner sep=0}, equals, from=4-2, to=5-3]
				\arrow["{(\Gamma_{h}*JtJc_x)^C_{\nu_x}}"{description}, from=4-2, to=6-1]
				\arrow["{(\Delta_{h,\eta_b J t }*\eta_\ell J c_x)^C_{\nu_x}}"{description}, from=5-3, to=2-3]
				\arrow[""{name=4, anchor=center, inner sep=0}, equals, from=6-1, to=6-3]
				\arrow[equals, from=6-3, to=6-4]
				\arrow[""{name=5, anchor=center, inner sep=0}, "{(\Delta_{h, \eta_b J t }*Tc_x)_{\nu_x}}"{description}, from=6-4, to=1-4]
				\arrow["{(\Delta_{h, Tt \eta_\ell Jc_x})_{\nu_x}}"'{pos=0.55}, from=6-4, to=5-3]
				\arrow["{(i')}"', draw=none, from=0, to=1]
				\arrow["{(iv')}"{description}, draw=none, from=2-3, to=5]
				\arrow["{(ii')}"', draw=none, from=1, to=4-2]
				\arrow["{(iii')}"{description}, draw=none, from=2, to=3]
				\arrow["{(v')}"{description}, draw=none, from=3, to=4]
			\end{tikzcd} \] 
			in $C$, which commutes since:
			\begin{enumerate}
				\item[(i')]commutes by naturality of the family $\set{ \Delta }$ with respect to $\Gamma_h*Jt$, using the fact that $h \circ J t = a$;
				\item[(ii')]commutes by naturality of the family $\set{ \Gamma}$ with respect to $\Gamma_h*Jt$, again using the fact that $h \circ J t = a$;
				\item[(iii')]commutes by axiom (b) of a colax algebra; 
				\item[(iv')]commutes by axiom (c) of a colax algebra;
				\item[(v')]commutes by axiom (a) of a colax algebra, using the fact that $T t \circ \eta_\ell = \eta_b \circ J t$ by naturality of $\eta$, that $t \circ c_x = \id_b$, and that $h \circ J t = a$.
			\end{enumerate}
			Summing up, we have shown that $(\bar\Delta_{h,k})_\nu$ coincides with $(\Delta_{h,k})_\nu$, which entails that $\bar\Delta = \Delta$, concluding the proof.
		\end{proof}
	\end{lemma}
	
	\begin{corollary}\label{cor:sigmatheta-identity}
		The composite $\Sigma \circ \Theta$ is the identity 2-functor on $\CoLaxAlg_\star^J(T)$.
		\begin{proof}
			It is immediate to see that $\Sigma \circ \Theta$ acts as the identity on morphisms of colax algebras for $\braket{T, \eta, (-)^*}$. As both $\Theta$ and $\Sigma$ act as the identity on transformations,  we conclude that $\Sigma\circ \Theta = \id$ as 2-functors on $\CoLaxAlg_\star^J(T)$.
		\end{proof}
	\end{corollary}
	
	On the other hand, consider now the composite $\Theta \circ \Sigma$.
	
	\begin{lemma}
		Let $\braket{C, \Phi, \mathfrak{i}, \mathfrak{m}}$ be a colax algebra for $\braket{\tilde T , \eta^\sharp, \mu^\sharp}$. Then
		\[ \Theta\Sigma \braket{C, \Phi, \mathfrak{i}, \mathfrak{m}} = \braket{C, \Phi, \mathfrak{i}, \mathfrak{m}}.\]
		\begin{proof}
			Denote $\Theta\Sigma \braket{C, \Phi, \mathfrak{i}, \mathfrak{m}}$ as $\braket{C, \bar \Phi, \bar{\mathfrak{i}}, \bar{\mathfrak{m}}}$. Clearly $\bar \Phi = \Phi$ as they are both the unique functor corresponding to the lax transformation $(-)^C \colon \funcat{J-}{C} \implies \funcat{T-}{C}$ defining $\braket{C, (-)^C, \Gamma,\Delta} \coloneqq \Sigma (\braket{C, \Phi, \mathfrak{i},\mathfrak{m}})$. Moreover, it is immediate to see that $\bar{\mathfrak{i}} = \mathfrak{i}$.
			
			It remains to show that $\bar{\mathfrak{m}} = \mathfrak{m}$. The natural transformation $\bar{\mathfrak{m}}$ is defined at an object $(b, h \colon J b \to \tilde T C, \nu) \in \tilde T^2 C$ as the composite:
			\[ \begin{tikzcd}[sep = 14 pt]
				{(\Phi\circ \mu^\sharp_C)(b,h,\nu) = a^CQ^*(\nu)} && {(a^CQ)^C(\nu)} && {(\Phi h )^C(\nu) = (\Phi\circ \tilde T \Phi)(b,h,\nu)}
				\arrow["{{{(\Delta_{a,Q})_\nu}}}", from=1-1, to=1-3]
				\arrow["{{{\tau^C_\nu}}}", from=1-3, to=1-5]
			\end{tikzcd} \]
			where we let $(\ell, a \colon J \ell \to C, (b, Q \colon J b \to T \ell, \nu)) \coloneqq \mathfrak{s}_{T,T} ( b, h , \nu)$ as in \ref{sssec:rel-to-abs_on-algebras}, and where $\tau \colon a^C Q \implies \Phi \circ h$ is the natural transformation defined also in \ref{sssec:rel-to-abs_on-algebras}. Unwrapping the definition of $(-)^C$, the previous composite is equal to
			\[ \begin{tikzcd}
				{\Phi \iota_\ell( a) (Q^*(\nu))} && {(\Phi \circ \iota_b(\Phi \iota_{\ell}(a)Q) )(\nu)} && {(\Phi\circ  \iota_b(\Phi h))(\nu)}
				\arrow["{{(\Delta_{a,Q})_\nu}}", from=1-1, to=1-3]
				\arrow["{(\Phi*\iota_b(\tau))_\nu}", from=1-3, to=1-5]
			\end{tikzcd} \] 
			and hence, unwrapping the definition of $\Delta_{a,Q}$, also equal to
			\[ \begin{tikzcd}[column sep = 20pt, row sep = 12pt]
				{\Phi(\ell, a , Q^*(\nu))} \\
				{\Phi(\bar \ell, \bar a , \bar Q^*(\nu))} &[-27pt] {=(\Phi\circ\mu^\sharp_C)(b,\iota_\ell(a)Q, \nu)} \\
				& {(\Phi\circ\tilde T \Phi)(b,\iota_\ell(a)Q, \nu)} &[-31pt] {=\Phi( b, a^CQ, \nu)} && {\Phi(b, \Phi h , \nu) , }
				\arrow["{\Phi(t, \id,\id)}", from=1-1, to=2-1]
				\arrow["{\mathfrak{m}_{(b, \iota_\ell (a) Q, \nu)}}", from=2-2, to=3-2]
				\arrow["{\Phi(\id, \tau, \id)}", from=3-3, to=3-5]
			\end{tikzcd} \] 
			where we let $(\bar \ell, \bar a \colon J \bar \ell \to C , (b, \bar Q \colon J b \to T \bar \ell, \nu ))\coloneqq \mathfrak{s}_{T,T}(b,\iota_\ell(a) Q , \nu )$ and where $t \colon \bar\ell \to \ell \in \Bcat$ is the canonical 1-cell induced being $\ell$ an oplax cocone on the diagram of which $\bar\ell$ is the oplax colimit. Note then that $\tau = \Phi * \tau'$ where $\tau ' \colon \iota_\ell(a) Q \implies h$ is the natural transformation defined at an object $x \in J b$ by the arrow
			\[ (c_x, \id, \id ) \colon \iota_\ell(a) (Q(x)) = ( \ell, a , Tc_x(\nu_x)) \to (Rx, a_x, \nu_x) = h(x)  \] 
			in $\tilde T C$. Therefore, $\Phi (\id,\tau,\id)$ coincides with $(\Phi \circ \tilde T \Phi) (\id,\tau',\id)$ by definition of $\tilde T \Phi$, and hence by naturality of $\mathfrak{m}$ we have that the diagram
			\[ \begin{tikzcd}[column sep = 18pt]
				{(\Phi\circ\mu^\sharp_C)(b,\iota_\ell(a)Q, \nu)} &[-25pt]& {(\Phi\circ \mu^\sharp_C)(b,h , \nu)} & {(\Phi\circ \tilde T \Phi)(b,h , \nu)} \\
				{(\Phi\circ\tilde T \Phi)(b,\iota_\ell(a)Q, \nu)} &[-25pt] {=\Phi( b, a^CQ, \nu)} && {\Phi(b, \Phi h , \nu)}
				\arrow["{(\Phi\circ \mu^\sharp_C)(\id, \tau', \id)}", from=1-1, to=1-3]
				\arrow["{\mathfrak{m}_{(b, \iota_\ell (a) Q, \nu)}}"', from=1-1, to=2-1]
				\arrow["{\mathfrak{m}_{(b, h, \nu)}}", from=1-3, to=1-4]
				\arrow[equals, from=1-4, to=2-4]
				\arrow["{\Phi(\id, \tau, \id)}"', from=2-2, to=2-4]
			\end{tikzcd} \] 
			commutes. Moreover, by the universal properties of oplax colimits (and by functoriality of $\Phi$), the diagram
			\[ \begin{tikzcd}
				{\Phi(\ell, a, Q^*(\nu))} \\
				{\Phi(\bar \ell, \bar a , \bar Q^*(\nu))} &[-30pt] {=(\Phi\circ\mu^\sharp_C)(b,\iota_\ell(a)Q, \nu)} && {(\Phi\circ \mu^\sharp_C)(b,h , \nu)}
				\arrow["{\Phi(t, \id,\id)}"', from=1-1, to=2-1]
				\arrow[curve={height=-12pt}, equals, from=1-1, to=2-4]
				\arrow["{(\Phi\circ \mu^\sharp_C)(\id, \tau', \id)}"', from=2-2, to=2-4]
			\end{tikzcd} \] 
			also commutes. Summing up, we have shown that $\bar{\mathfrak{m}}_{(b,h,\nu)}$ coincides with $\mathfrak{m}_{(b,h,\nu)}$, which means that $\bar{\mathfrak{m}} = \mathfrak{m}$, concluding the proof.
		\end{proof}
	\end{lemma}
	
	\begin{corollary}
		The composite $\Theta \circ \Sigma$ is the identity 2-functor on $\CoLaxAlg_\star(\tilde T)$.
		\begin{proof}
			Analogous to Corollary \ref{cor:sigmatheta-identity}.
		\end{proof}
	\end{corollary}
	
	\section{The ultracompletion}\label{sec:ultracompletion}
	We now arrive at the intended application of Theorem \ref{thm:algebras_preserved}. In the case of the relative ultrafilter 2-monad $\braket{J\beta, J\eta, (-)^\ast}$ of Section \ref{sec:weak-ultra}, the hypotheses of Section \ref{sec:absolution} are satisfied. In particular, since $\Set$ is discrete on 2-cells, (small) oplax colimits are simply conical colimits, which $\Set$ admits and the inclusion $J \colon \Set \hookrightarrow \CAT$ preserves (cf.\ Remark \ref{rem:all_small_oplax_colimits_suffices}). Therefore, we can deduce that weak ultracategories are the colax algebras for a universally obtained pseudomonad on $\CAT$, for which Rosolini suggested the name \emph{ultracompletion} in \cite{rosolini-ultracompletion}. Note that, in this section, by Remark \ref{rem:size-issues} we can replace $\CAT$ with the 2-category of locally small categories.
	
	\begin{corollary}\label{cor:ultracategories-universally}
		Weak ultracategories are the colax algebras for a pseudomonad on $\CAT$ whose underlying 2-endofunctor $\tilde{\beta} \colon \CAT \to \CAT$ is (up to isomorphism) the unique 2-functor such that $(\beta,\rho)$ is a right Kan extension and $(\tilde \beta,\zeta)$ is a left oplax Kan extension in the diagram
		\[\begin{tikzcd}
			{{\bf FinSet}} & \Set & \CAT \\
			\Set \\
			\CAT
			\arrow[hook, from=1-1, to=1-2]
			\arrow[hook, from=1-1, to=2-1]
			\arrow[hook, from=1-2, to=1-3]
			\arrow[""{name=0, anchor=center, inner sep=0}, "\beta"{description}, curve={height=12pt}, from=2-1, to=1-2]
			\arrow[hook, from=2-1, to=3-1]
			\arrow[""{name=1, anchor=center, inner sep=0}, "{\tilde{\beta}}"', curve={height=24pt}, from=3-1, to=1-3]
			\arrow["\rho"', shorten <=4pt, Rightarrow, from=0, to=1-1]
			\arrow["\zeta", shorten <=6pt, shorten >=6pt, Rightarrow, from=0, to=1]
		\end{tikzcd}\]
	\end{corollary}
	\begin{proof}
		By Proposition \ref{prop:left-oplax-Kan-extension-in-CAT} we know that $(\tilde \beta , \zeta)$ is the left oplax Kan extension of the composite $\beta \colon \Set \to \Set \hookrightarrow \CAT$ along the inclusion $\Set \hookrightarrow \CAT$. By the above, we can then apply Theorem \ref{thm:algebras_preserved} to the relative ultrafilter 2-monad, obtaining that $\mathbf{WUlt}_\star \iso \CoLaxAlg_\star(\tilde \beta)$ for each $\star \in \set{co, lax, ps, str}$.  Finally, we combine this with the well-known fact that the ultrafilter monad $\beta \colon \Set \to \Set$ is the codensity monad of the inclusion ${\bf FinSet} \hookrightarrow \Set$ (\cite{kennison-gildenhuys,manes,leinster-codensity}).
	\end{proof}
	
	\begin{remark}
		Our `ultracompletion' pseudomonad $\tilde \beta \colon \CAT \to \CAT$ strongly resembles the pseudomonad constructed in \cite{hamad} -- and hence also that of \cite{rosolini-ultracompletion} -- but is not identical.  This is because our weak ultracategories, i.e.\ the colax algebras for $\tilde \beta$, satisfy a weaker axiomatisation than Lurie's ultracategories (Definition \ref{def:luries-axioms}). The latter, in turn, are the \emph{normal} colax algebras for Hamad's pseudomonad, where a colax algebra $\braket{M,\Phi, \mathfrak{i},\mathfrak{m}}$ is normal if the counitor $\mathfrak{i}$ is a natural isomorphism.
	\end{remark}
	
	\subsection{Prime categories}
	One benefit of our general treatment involving arbitrary relative 2-monads acting on $\CAT$ is that we can apply the same results to analogues of ultracategories.  To conclude, we discuss a generalisation of ultracategories where the role of ultrafilters and ultraproducts is played by \emph{prime filters} and the \emph{prime product} construction from positive model theory \cite{moraschini-prime-products}. Our hope is that this may lead to an ordered generalisation of the ``Stone duality for first-order logic'' of \cite{makkai}, as a kind of \emph{Priestley} duality for first-order logic.
	
	As we have seen, ultracategories can be interpreted as a categorical analogue of compact Hausdorff spaces, which are precisely the \emph{ultrasets}, i.e.\ those ultracategories whose carrier category is a set (Remark \ref{rem:ultrasets}). In particular, among these are Stone spaces, which are the spaces of models of classical propositional theories by Stone duality for Boolean algebras.  In the ordered setting, compact Hausdorff spaces are replaced by \emph{compact ordered spaces}; among these, Priestley spaces correspond to spaces of models of coherent propositional theories through Priestley duality. In \cite{flagg}, the category of compact ordered spaces and continuous order-preserving maps is proved to be monadic over the category $\mathbf{Pos}$ of posets, similarly to Manes' theorem \cite{manes-comphaus} for compact Hausdorff spaces. 
	
	The monad on $\mathbf{Pos}$ witnessing the monadicity result is the \emph{prime upper filter monad} $\braket{ B , \eta, (-)^*}$, and it maps a poset $P$ to the poset $BP \coloneqq \operatorname{Spec} (\operatorname{Up} (P))$ of prime filters of the lattice $\operatorname{Up} (P)$ of upward-closed subsets of $P$. Moreover, $B$ is a 2-monad on $\mathbf{Pos}$ seen as a locally-thin 2-category via the pointwise order on its homsets.  In particular, the results of \cite{adamek-ultrafilters} show that $\braket{B, \eta, (-)^*}$ is the codensity 2-monad of the inclusion $\mathbf{FinPos} \hookrightarrow \Pos$, meaning that $B$ is the right Kan extension -- in the 2-categorical sense of \cite[\S 4.1]{kelly2005} -- of $\mathbf{FinPos} \hookrightarrow \Pos$ along itself. Therefore, as the inclusion $J \colon \mathbf{Pos} \hookrightarrow \CAT$ of posets into categories is clearly 2-fully faithful, by Example \ref{ex:Jff-and-monad-yields-relative} we obtain a relative 2-monad $\braket{ J B , J \eta, (-)^*}$ over $J \colon \mathbf{Pos}\to \CAT$, which we call the \emph{relative prime upper filter 2-monad}. This allows us to give the following definition.
	
	\begin{definition}
		A \emph{prime category} is a colax algebra for the relative prime upper filter 2-monad.
	\end{definition}
	
	\begin{example}
		Analogously to the case of ultracategories, the category $\mathbf{Pos}$ carries the archetypal structure of a prime category. A monotone function $h \colon P \to \mathbf{Pos}$ can be identified with an \emph{ordered system} in the sense of \cite{moraschini-prime-products}, i.e.\ a sequence of posets $\set{ M_p }_{p\in P}$ together with a family of monotone maps $\set{ f_{p,p'} \colon M_p \to M_{p'}}_{p \leq p'}$ satisfying:
		\begin{enumerate}
			\item $f_{p,p} = \id_{M_p}$ for each $p\in P$;
			\item $f_{p',p''}\circ f_{p,p'} = f_{p,p''}$ if $p \leq p' \leq p''$.
		\end{enumerate}
		The extension $h^{\mathbf{Pos}} \colon B P \to \mathbf{Pos}$ acts by sending a prime upper filter $\nu \in B P$ to the \emph{prime product} of the system $h$ with respect to $\nu$, as constructed in \cite[Definition 2.4]{moraschini-prime-products}:
		\[ h^{\mathbf{Pos}} (\nu) \coloneqq \colim_{U \in \op\nu } \left( \lim_{p \in U} M_p \right)\]
	\end{example}
	
	\begin{example}
		By \cite[Theorem 12]{flagg}, the \emph{strict} prime categories whose carrier category is a poset correspond to compact ordered spaces (cf.\ Example \ref{ex:Jff-and-monad-colax-algebras-embed} and Remark \ref{rem:ultrasets}).
	\end{example}
	
	Note that the inclusion $J \colon \mathbf{Pos} \hookrightarrow \CAT$ satisfies the hypotheses of Section \ref{sec:absolution}.  To see why $J$ preserves oplax colimits of diagrams indexed by a poset, let $F \colon I \to \mathbf{Pos}$ be such a diagram. Then $J \circ F \colon I \to \CAT$ admits an oplax colimit $L$ in $\CAT$, given by its Grothendieck construction; \todo{add ref} explicitly, this means that $L$ is the category having:
	\begin{enumerate}
		\item as objects, pairs $(i, x)$ of an element $i \in I$ and an element $x \in F(i)$; 
		\item an arrow $(i, x) \leq (i', x')$ if and only if $i \leq i'$ in $I$ and $F(i \leq i')(x) \leq x'$ in $F(x')$.
	\end{enumerate}
	In particular, $L$ itself is a poset. As $J$ is 2-fully faithful, it reflects all 2-colimits (among which are oplax colimits), meaning that $L$ is the oplax colimit of the diagram $F$ in $\mathbf{Pos}$. This yields the following.
	
	\begin{corollary}
		Prime categories are the colax algebras for a pseudomonad on $\CAT$ whose underlying 2-endofunctor $\tilde B \colon \CAT \to \CAT$ is (up to isomorphism) the unique 2-functor such that $(B,\rho)$ is a right Kan extension and $(\tilde B, \zeta)$ is a left oplax Kan extension in the diagram
		\[\begin{tikzcd}
			{{\mathbf {FinPos}}} & \Pos & \CAT \\
			\Pos \\
			\CAT
			\arrow[hook, from=1-1, to=1-2]
			\arrow[hook, from=1-1, to=2-1]
			\arrow[hook, from=1-2, to=1-3]
			\arrow[""{name=0, anchor=center, inner sep=0}, "B"{description}, curve={height=12pt}, from=2-1, to=1-2]
			\arrow[hook, from=2-1, to=3-1]
			\arrow[""{name=1, anchor=center, inner sep=0}, "{\tilde{B}}"', curve={height=24pt}, from=3-1, to=1-3]
			\arrow["\rho"', shorten <=4pt, Rightarrow, from=0, to=1-1]
			\arrow["\zeta", shorten <=6pt, shorten >=6pt, Rightarrow, from=0, to=1]
		\end{tikzcd}\]
		\begin{proof}
			Analogous to Corollary \ref{cor:ultracategories-universally}. 
		\end{proof}
	\end{corollary}

	\bibliographystyle{alpha}
	\bibliography{biblio}
	
	\newpage
	\input{appendix.tex}
	
\end{document}

%% file: appendix.tex
\appendix

\section{Oplax coends}\label{app:oplax-coends}

As we show in Section \ref{sec:kan-ext}, left oplax Kan extension in $\CAT$ can be calculated by means of \emph{oplax coends}. In this appendix, we recall the definitions of {(op)lax (co)ends} and we prove some of their properties relating them to left oplax Kan extensions.

\begin{definition}\label{def:oplax-cowedge}
	Fix a 2-functor $H \colon \op\Bcat \times \Bcat \to \Acat$. An \emph{oplax cowedge} on $H$, denoted by $\gamma \colon H \to A$, consists of:
	\begin{enumerate}
		\item an object $A \in \Acat$, called the \emph{vertex};
		\item for each object $b\in\Bcat$, a 1-cell $\gamma_b \colon H(b,b)\to A$ in $\Acat$;
		\item for each 1-cell $f \colon b' \to b \in \Bcat$, a 2-cell $\gamma_f \colon \gamma_b \circ H(\id, f) \implies \gamma_{b'} \circ H(f,\id)$,
	\end{enumerate}
	satisfying the following axioms:
	\begin{enumerate}
		\item[(a)]for each object $b \in\Bcat$, the 2-cell $\gamma_{\id_b}$ is the identity on $\gamma_b$;
		\item[(b)]for each composable pair of 1-cells $f \colon b'\to b$ and $g \colon b'' \to b'$ in $\Bcat$, the pasting on the left coincides with that on the right:
		\[ \begin{tikzcd}[row sep = 14pt, column sep = -8pt]
			& {H(b,b')} && {H(b,b)} \\
			\\
			{H(b'',b'')} && {H(b',b')} && {A,} \\
			\\
			& {H(b',b'')} && {H(b'',b'')}
			\arrow["{H(\id, f)}", from=1-2, to=1-4]
			\arrow["{H(f,\id)}"', from=1-2, to=3-3]
			\arrow["{\gamma_f}"', shorten <=9pt, shorten >=9pt, Rightarrow, from=1-4, to=3-3]
			\arrow["{\gamma_b}", from=1-4, to=3-5]
			\arrow["{H(\id,g)}", from=3-1, to=1-2]
			\arrow["{H(f,\id)}"', from=3-1, to=5-2]
			\arrow["{\gamma_{b'}}"', from=3-3, to=3-5]
			\arrow["{\gamma_g}"', shorten <=9pt, shorten >=9pt, Rightarrow, from=3-3, to=5-4]
			\arrow["{H(\id,g)}", from=5-2, to=3-3]
			\arrow["{H(g,\id)}"', from=5-2, to=5-4]
			\arrow["{\gamma_{b''}}"', from=5-4, to=3-5]
		\end{tikzcd}  \quad \begin{tikzcd}[row sep = 14pt, column sep = -8pt]
			& {H(b,b')} && {H(b,b)} \\
			\\
			{H(b'',b'')} && {\textcolor{white}{H(b',b')}} && {A.} \\
			\\
			& {H(b',b'')} && {H(b'',b'')}
			\arrow["{H(\id, f)}", from=1-2, to=1-4]
			\arrow["{\gamma_b}", from=1-4, to=3-5]
			\arrow["{\gamma_{gf}}"', shorten <=15pt, shorten >=15pt, Rightarrow, from=1-4, to=5-4]
			\arrow["{H(\id,g)}", from=3-1, to=1-2]
			\arrow["{H(\id,gf)}"', from=3-1, to=1-4]
			\arrow["{H(f,\id)}"', from=3-1, to=5-2]
			\arrow["{H(gf,\id)}", from=3-1, to=5-4]
			\arrow["{H(g,\id)}"', from=5-2, to=5-4]
			\arrow["{\gamma_{b''}}"', from=5-4, to=3-5]
		\end{tikzcd}\] 
		\item[(c)]for each parallel pair of 1-cells $f,g\colon b' \to b$ and each 2-cell $\sigma \colon f \implies g$ in $\Bcat$, the pasting on the left coincides with that on the right:
		\[ \begin{tikzcd}[row sep = 32pt, column sep = 18pt]
			& {H(b,b')} \\
			{H(b',b')} && {H(b,b),} \\
			& A
			\arrow["{H(g,\id)}"', from=1-2, to=2-1]
			\arrow[""{name=0, anchor=center, inner sep=0}, "{H(\id, f)}", curve={height=-20pt}, from=1-2, to=2-3]
			\arrow[""{name=1, anchor=center, inner sep=0}, "{H(\id,g)}"', curve={height=20pt}, from=1-2, to=2-3]
			\arrow[""{name=2, anchor=center, inner sep=0}, "{\gamma_{b'}}"', from=2-1, to=3-2]
			\arrow[""{name=3, anchor=center, inner sep=0}, "{\gamma_b}", from=2-3, to=3-2]
			\arrow["{H(\id,\sigma)}"'{pos=0.3}, shorten <=4pt, shorten >=4pt, Rightarrow, from=0, to=1]
			\arrow["{\gamma_g}"',shift right=3, shorten <=8pt, shorten >=8pt, Rightarrow, from=3, to=2]
		\end{tikzcd} \qquad \begin{tikzcd}[row sep = 32pt, column sep = 18pt]
			& {H(b,b')} \\
			{H(b',b')} && {H(b,b).} \\
			& A
			\arrow[""{name=0, anchor=center, inner sep=0}, "{H(g,\id)}"', curve={height=20pt}, from=1-2, to=2-1]
			\arrow[""{name=1, anchor=center, inner sep=0}, "{H(f,\id)}", curve={height=-20pt}, from=1-2, to=2-1]
			\arrow["{H(\id, f)}", from=1-2, to=2-3]
			\arrow[""{name=2, anchor=center, inner sep=0}, "{\gamma_{b'}}"', from=2-1, to=3-2]
			\arrow[""{name=3, anchor=center, inner sep=0}, "{\gamma_b}", from=2-3, to=3-2]
			\arrow["{H(\sigma,\id)}"'{pos=0.7}, shorten <=4pt, shorten >=4pt, Rightarrow, from=1, to=0]
			\arrow["{\gamma_f}"', shift right=3, shorten <=8pt, shorten >=8pt, Rightarrow, from=3, to=2]
		\end{tikzcd}\]
	\end{enumerate}
	
	Given two oplax cowedges $\gamma,\gamma' \colon H \to A$, a \emph{modification} $\Psi \colon \gamma \implies \gamma'$ is a family of 2-cells $\Psi_b \colon \gamma_b \implies \gamma'_b$ such that, for each 1-cell $f \colon b'\to b\in\Bcat$, the pasting on the left coincides with that on the right:
	\[ \begin{tikzcd}[row sep = 32pt, column sep = 18pt]
		& {H(b,b')} \\
		{H(b',b')} && {H(b,b),} \\
		& A
		\arrow[""{name=0, anchor=center, inner sep=0}, "{H(f,\id)}"', from=1-2, to=2-1]
		\arrow[""{name=1, anchor=center, inner sep=0}, "{H(\id, f)}", from=1-2, to=2-3]
		\arrow[""{name=2, anchor=center, inner sep=0}, "{\gamma_{b'}}", curve={height=-20pt}, from=2-1, to=3-2]
		\arrow[""{name=3, anchor=center, inner sep=0}, "{\gamma'_{b'}}"', curve={height=20pt}, from=2-1, to=3-2]
		\arrow["{\gamma_b}", from=2-3, to=3-2]
		\arrow["{\gamma_f}"', shift left=3, shorten <=8pt, shorten >=8pt, Rightarrow, from=1, to=0]
		\arrow["{\Psi_{b'}}"', shorten <=4pt, shorten >=4pt, Rightarrow, from=2, to=3]
	\end{tikzcd}\qquad \begin{tikzcd}[row sep = 32pt, column sep = 18pt]
		& {H(b,b')} \\
		{H(b',b')} && {H(b,b).} \\
		& A
		\arrow[""{name=0, anchor=center, inner sep=0}, "{H(f,\id)}"', from=1-2, to=2-1]
		\arrow[""{name=1, anchor=center, inner sep=0}, "{H(\id, f)}", from=1-2, to=2-3]
		\arrow["{\gamma'_{b'}}"', from=2-1, to=3-2]
		\arrow[""{name=2, anchor=center, inner sep=0}, "{\gamma_b}", curve={height=-20pt}, from=2-3, to=3-2]
		\arrow[""{name=3, anchor=center, inner sep=0}, "{\gamma'_b}"', curve={height=20pt}, from=2-3, to=3-2]
		\arrow["{\gamma'_f}"', shift left=3, shorten <=8pt, shorten >=8pt, Rightarrow, from=1, to=0]
		\arrow["{\Psi_b}"', shorten <=4pt, shorten >=4pt, Rightarrow, from=2, to=3]
	\end{tikzcd}   \]
	
	A \emph{lax cowedge} on $H$ is an oplax cowedge on the 2-functor $H' \colon \Bcat \times \op\Bcat\to \Acat$ defined by $(b,b')\mapsto H(b',b)$. Dually, an \emph{(op)lax wedge} on $H$ is an (op)lax cowedge in $\op\Acat$. Transformations thereof are defined in the obvious way.
\end{definition}

\begin{notation}
	For an oplax cowedge $\gamma \colon H \to A$ and a 1-cell $u \colon A \to A' \in \Acat$, we denote with $u\gamma$ the oplax cowedge $H\to A'$ defined by the families $\set{ u \circ \gamma_b }$ and $\set{ u * \gamma_f }$.
\end{notation}


\begin{definition}\label{def:oplax-coend}
	The \emph{oplax coend} of $H$ consists of an object $\oplend^{b}H(b,b) \in \Acat$ together with an oplax cowedge $\lambda \colon H \to \oplend^{b}H(b,b)$ with the folllowing universal properties.
	\begin{enumerate}
		\item For each oplax cowedge $\gamma \colon H \to A$, there exists a unique 1-cell $u \colon  \oplend^{b}H(b,b)\to A$ in $\Bcat$ such that $\gamma = u \lambda$, i.e. that for each $b\in \Bcat$ the diagram 
		\[ \begin{tikzcd}
			&& {\oplend^{b} H(b,b)} \\
			{H(b,b)} \\
			&& {A}
			\arrow["u", from=1-3, to=3-3]
			\arrow["{\lambda_b}", from=2-1, to=1-3]
			\arrow["{\gamma_b}"', from=2-1, to=3-3]
		\end{tikzcd} \] 
		commutes.
		\item For each pair of 1-cells $u,v \colon \oplend^{b }H(b,b) \to A$ and each modification $\Psi \colon u\lambda \implies v \lambda$, there exists a unique 2-cell $\psi\colon u \implies v$ in $\Acat$ such that $\Psi = \psi*\lambda$, i.e. that $\Psi_b = \psi * \lambda_b$ for each $b\in\Bcat$.
	\end{enumerate}
	If it exists, an oplax coend is clearly unique up to isomorphism.  Similar universal properties define \emph{lax coends} and \emph{(op)lax ends}.
\end{definition}

Fix a pair of 2-functors $J,T \colon \Bcat \to \Acat$. Recall by Definition \ref{def:left-oplax-kan-extension} that a 2-functor $\tilde T \colon \Acat \to \Acat$ is the \emph{left oplax Kan extension} of $T$ along $J$ if, for any pair $a,a' \in \Acat$, there is an isomorphism of categories 
\[ \Hom{\Acat}{\tilde T a }{a'} \iso \oplend_{b\in\Bcat} \funcat{\Hom\Acat{Jb}a}{\Hom\Acat{Tb}{a'}} .  \]
In these terms, we can easily derive the following.

\begin{proposition}\label{prop:app.weak-left-oplax-kan-extension}
	If $\tilde T \colon \Acat\to\Acat$ is the left oplax Kan extension of $T \colon \Bcat\to \Acat$ along $J \colon \Bcat\to \Acat$, then for each 2-functor $S\colon \Acat\to \Acat$ there is an isomorphism of categories:
	\[  \mathbf{Str}\funcat{\Acat}{\Acat} ( \tilde T , S ) \iso \mathbf{Oplax}\funcat{\Bcat}{\Acat} \left( T , S \circ J \right)  \]
	which is natural in $S$. 
	\begin{proof}
		The proof is formally identical to that of \cite[Theorem 4.38]{kelly2005}, up to considering oplax ends instead of ends. Indeed, we have the following chain of isomorphisms, all natural in $S$:
		\begin{align*}
			\mathbf{Str}\funcat{\Acat}{\Acat} ( \tilde T , S  ) & \iso \int_{a\in \Acat} \Hom\Acat{\tilde T a}{S a} & \mbox{(\cite[\S 2.10]{kelly2005})}\\
			& \iso \int_{a\in\Acat} \oplend_{b\in\Bcat} \funcat{ \Hom\Acat{Jb}{a} }{ \Hom\Acat{Tb}{Sa} } &  \\
			& \iso \oplend_{b\in\Bcat} \int_{a\in \Acat} \funcat{ \Hom\Acat{Jb}{a} }{ \Hom\Acat{Tb}{Sa} } & \mbox{(\cite[\S 3.6]{hirata2022noteslaxends})}\\
			& \iso \oplend_{b\in\Bcat} \Hom\Acat{Tb}{SJb} & \mbox{(\cite[\S 2.33]{kelly2005})}\\
			& \iso \mathbf{Oplax}\funcat{\Bcat}{\Acat} \left( T , S\circ J \right) .  & \mbox{(\cite[\S 3.1]{hirata2022noteslaxends})} 
		\end{align*}
		In particular, by naturality in $S$, the unique oplax transformation $\zeta \colon T \implies \tilde T \circ J$ corresponding to the strict transformation $\id \colon \tilde T \implies \tilde T$ exhibits $\tilde T$ as the weak left oplax Kan extension of $T$ along $J$.
	\end{proof}
\end{proposition}

Let us now restrict to $\CAT$-valued 2-functors.

\begin{proposition}\label{prop:left-oplax-kan-ext-determined-by-oplax-coends}
	Let $J,T \colon \Bcat \to \CAT$ be 2-functors. Up to isomorphism, the left oplax Kan extension $\tilde T \colon \CAT \to \CAT$ of $T$ along $J$ is defined by setting, for each category $C$:
	\[ \tilde T C \coloneqq \oplend^{b \in \Bcat} \funcat{Jb}{C} \times T b .\]
	\begin{proof}
		Let $\tilde T C$ be the oplax coend $\oplend^{b \in \Bcat} \funcat{Jb}{C} \times T b$, in which case the assignment $C \mapsto \tilde T C$ uniquely lifts to a 2-functor $\tilde T \colon \CAT \to \CAT$ by the universal properties of oplax coends. For each pair of categories $C$ and $D$ we then have the following chain of natural isomorphisms:
		\begin{align*}
			\funcat{ \tilde T C }{D} & \iso \funcat{\oplend^{b\in\Bcat}\funcat{Jb}{C} \times Tb }{ D } \\
			& \iso \oplend_{b\in\Bcat} \big[ \funcat{Jb}{C} \times Tb , D\big] & \mbox{(\cite[\S 3.2]{hirata2022noteslaxends})}\\
			& \iso \oplend_{b\in\Bcat} \big[ \funcat{Jb}C , \funcat{Tb}D \big] \\
			& \iso \mathbf{Lax}\funcat{\op\Bcat}{\CAT} \big( \funcat{J-}{C} , \funcat{T-}{D}\big) , & \mbox{(\cite[\S 3.1]{hirata2022noteslaxends})}
		\end{align*}
		thus exhibiting $\tilde T$ as the left oplax Kan extension of $T$ along $J$.
	\end{proof}
\end{proposition}

We conclude this appendix by finishing the proof of Proposition \ref{prop:left-oplax-Kan-extension-in-CAT}.

\begin{proposition}\label{prop:oplax-coend-in-CAT}
	For each category $C$, the category $\tilde T C$ of Definition \ref{def:Ttilde-of-C} is the oplax coend of the 2-functor 
	\[ H_C \colon \op\Bcat \times \Bcat \to \CAT \qquad H_C (b, b') \coloneqq \funcat{ J b }C \times T b'\]
	\begin{proof}
		First, we show that $\tilde T C$ is the vertex of an oplax cowedge $\lambda$ on $H_C$.
		\begin{enumerate}
			\item For each $b\in\Bcat$, let $\lambda_b \colon H_C(b,b) \to \tilde TC$ be the functor 
			\[\begin{tikzcd}
				{(h\colon J b \to C, \nu \in T b)} & {(b, h\colon J b \to C, \nu \in T b)} \\
				{(h'\colon J b \to C, \nu' \in T b)} & {(b, h'\colon J b \to C, \nu' \in T b).}
				\arrow[maps to, from=1-1, to=1-2]
				\arrow["{(\alpha, \phi)}"', from=1-1, to=2-1]
				\arrow["{(\id, \alpha,\phi)}", from=1-2, to=2-2]
				\arrow[maps to, from=2-1, to=2-2]
			\end{tikzcd}\]
			
			\item For each 1-cell $f \colon b' \to b\in \Bcat$, let $\lambda_f \colon \lambda_b \circ H_C(\id, f) \implies \lambda_{b'} \circ H_C(f,\id)$ be the natural transformation defined at an object $(h \colon J b \to C, \nu' \in T b')$ by the arrow
			\[ (f, \id, \id) \colon (b, h , T f(\nu')) \to (b', h\circ Jf , \nu') .\]
		\end{enumerate}
		To show that this assignment defines an oplax cowedge $\lambda \colon H_C \to \tilde T C$, we need to verify that it satisfies the three axioms of Definition \ref{def:oplax-cowedge}. Axiom (a) and (b) are immediate, so we only show axiom (c), i.e.\ that for $f,g \colon b' \to b$ and $\sigma \colon f \implies g$ in $\Bcat$, the pastings 
		\[ \begin{tikzcd}[row sep = 32pt, column sep = 18pt]
			& {H_C(b,b')} \\
			{H_C(b',b')} && {H_C(b,b),} \\
			& {\tilde T C}
			\arrow["{H_C(g,\id)}"', from=1-2, to=2-1]
			\arrow[""{name=0, anchor=center, inner sep=0}, "{H_C(\id, f)}", curve={height=-24pt}, from=1-2, to=2-3]
			\arrow[""{name=1, anchor=center, inner sep=0}, "{H_C(\id,g)}"', curve={height=24pt}, from=1-2, to=2-3]
			\arrow[""{name=2, anchor=center, inner sep=0}, "{\lambda_{b'}}"', from=2-1, to=3-2]
			\arrow[""{name=3, anchor=center, inner sep=0}, "{\lambda_b}", from=2-3, to=3-2]
			\arrow["{H_C(\id,\sigma)}"'{pos=0.3}, shorten <=4pt, shorten >=4pt, Rightarrow, from=0, to=1]
			\arrow["{\lambda_g}"',shift right=3, shorten <=8pt, shorten >=8pt, Rightarrow, from=3, to=2]
		\end{tikzcd} \quad \begin{tikzcd}[row sep = 32pt, column sep = 18pt]
			& {H_C(b,b')} \\
			{H_C(b',b')} && {H_C(b,b)} \\
			& {\tilde T C}
			\arrow[""{name=0, anchor=center, inner sep=0}, "{H_C(g,\id)}"', curve={height=24pt}, from=1-2, to=2-1]
			\arrow[""{name=1, anchor=center, inner sep=0}, "{H_C(f,\id)}", curve={height=-24pt}, from=1-2, to=2-1]
			\arrow["{H_C(\id, f)}", from=1-2, to=2-3]
			\arrow[""{name=2, anchor=center, inner sep=0}, "{\lambda_{b'}}"', from=2-1, to=3-2]
			\arrow[""{name=3, anchor=center, inner sep=0}, "{\lambda_b}", from=2-3, to=3-2]
			\arrow["{H_C(\sigma,\id)}"'{pos=0.7}, shorten <=4pt, shorten >=4pt, Rightarrow, from=1, to=0]
			\arrow["{\lambda_f}"', shift right=3, shorten <=8pt, shorten >=8pt, Rightarrow, from=3, to=2]
		\end{tikzcd}\]
		coincide. Given $(h \colon J b \to C, \nu' \in T b') \in H_C(b,b')$, the component at $(h,\nu')$ of the first pasting is represented by the triple $(g, \id, (T\sigma)_{\nu'})$, while that of the second pasting is represented by the triple $(f, h * J \sigma , \id)$. Evidently, $\sigma \colon f \implies g$ witnesses the equivalence of these two arrows as in Definition \ref{def:Ttilde-of-C}, so the two components coincide.
		
		Finally, we show that $\lambda$ exhibits $\tilde TC$ as the oplax coend of $H_C$. Let $\gamma \colon H_C \to A$ be another oplax cowedge. The unique functor $u \colon \tilde T C \to A$ such that $\gamma = u \lambda$ is evidently the functor defined:
		\begin{enumerate}
			\item on objects by mapping $(b, h \colon J b \to C, \nu \in T b) \in \tilde T C$ to $\gamma_b(h, \nu) \in A$;
			\item on arrows by mapping $(f, \alpha,\phi) \colon ( b, h , \nu) \to (b', h', \nu') \in \tilde T C$ to the composite 
			\[ \begin{tikzcd}[column sep = 34pt]
				{\gamma_b(h,\nu)} & {\gamma_b(h, Tf(\nu'))} & {\gamma_{b'}(h\circ Jf, \nu')} & {\gamma_{b'}(h',\nu').}
				\arrow["{\gamma_b(\id, \phi)}", from=1-1, to=1-2]
				\arrow["{(\gamma_f)_{(h,\nu')}}", from=1-2, to=1-3]
				\arrow["{\gamma_{b'}(\alpha,\id)}", from=1-3, to=1-4]
			\end{tikzcd} \] 
		\end{enumerate}
		Let $u,v \colon \tilde T C \to A$ be functors and let $\Psi \colon u \lambda \implies v\lambda$ be a modification of oplax cowedges. The unique natural transformation $\psi \colon u \implies v$ such that $\Psi = \psi *\lambda$ is evidently the natural transformation defined at an object $(b, h\colon J b \to C, \nu \in T b)\in\tilde T C$ by the arrow $(\Psi_b)_{(h,\nu)} \colon u(b,h,\nu) \to v(b,h,\nu)$.
	\end{proof}
\end{proposition}

\section{Omitted proofs from Section \ref{sec:algebras}  }\label{app:algebras}

In this appendix we give some of the routine proofs omitted from Section \ref{sec:algebras}. 

\begin{lemma}\label{lem:rel-to-abs_is-algebra}
	Given a colax algebra $\braket{C, (-)^C, \Gamma,\Delta}$ for the $J$-relative 2-monad $\braket{T, \eta, (-)^*}$, the tuple $\braket{C,\Phi,\mathfrak{i},\mathfrak{m}}$ defined in \ref{sssec:rel-to-abs_on-algebras} is a colax algebra for the pseudomonad $\braket{\tilde T , \eta^\sharp, \mu^\sharp}$.
	\begin{proof}
		We will verify (half of) axiom (a) of Definition \ref{def:colax-algebras-for-pseudomonad}, i.e.\ that the composite natural transformation
		\[ \begin{tikzcd}
			{\tilde T C} &&&& {\tilde T C} \\
			&& {\tilde T ^ 2 C} \\
			&& {\tilde T C} \\
			C &&&& C
			\arrow[""{name=0, anchor=center, inner sep=0}, equals, from=1-1, to=1-5]
			\arrow["{{\eta^\sharp_{\tilde T C}}}"{description}, from=1-1, to=2-3]
			\arrow["\Phi"', from=1-1, to=4-1]
			\arrow["\Phi", from=1-5, to=4-5]
			\arrow[""{name=1, anchor=center, inner sep=0}, "{{\mu^\sharp_C}}"{description}, from=2-3, to=1-5]
			\arrow["{{\tilde T \Phi}}"', from=2-3, to=3-3]
			\arrow[""{name=2, anchor=center, inner sep=0}, "\Phi"{description}, from=3-3, to=4-5]
			\arrow["{{\eta^\sharp_C}}"{description}, from=4-1, to=3-3]
			\arrow[""{name=3, anchor=center, inner sep=0}, equals, from=4-1, to=4-5]
			\arrow["{{\mathfrak{p}_C^{-1}}}", shorten <=3pt, shorten >=3pt, Rightarrow, from=0, to=2-3]
			\arrow["{{\mathfrak{m}}}", shorten <=9pt, shorten >=9pt, Rightarrow, from=1, to=2]
			\arrow["{{\mathfrak{i}}}", shorten <=2pt, shorten >=3pt, Rightarrow, from=3-3, to=3]
		\end{tikzcd} \] 
		coincides with the identity on $\Phi$. Fix an object $(b, h \colon J b \to C, \nu ) \in \tilde T C$. Note that, by definition, $\eta^\sharp_{\tilde T C} ( b , h , \nu ) = (1_{\Bcat}, (b,h,\nu)\colon 1_{\CAT} \to \tilde T C , \eta_{1_{\Bcat}}(\bullet))$. Hence, recalling the definition of $\mu^\sharp$, we have that $\mu^\sharp_C (1_{\Bcat}, (b,h,\nu), \eta_{1_{\Bcat}}(\bullet))$ is given by the triple $( \ell, a \colon J \ell \to C , Q^*(\eta_{1_{\Bcat}}(\bullet))) $, where $\ell\in \Bcat$ is the oplax colimit of the constant diagram $\op{(Jb)}\to\Bcat$ of value $1_{\Bcat}$, $a \colon J \ell \to C$ is uniquely induced as $C$ is an oplax cocone on the same diagram, and $Q \colon 1_{\CAT} \to T \ell$ is such that $Q(\bullet) = Q^*(\eta_{1_{\Bcat}(\bullet)} ) = T u (\nu)$, where $u \colon b \to \ell$ is the invertible arrow constructed in Proposition \ref{prop:structure-on-Bcat-to-CAT}. Denote with $\bar u \colon \ell \to b$ its inverse, so that $h \circ J \bar u = a$ and $T \bar u (Q(\bullet)) = \nu$.
		
		With this notation, we then need to show that the composite:
		\[ \begin{tikzcd}[column sep = 17pt, row sep =12pt]
			{\Phi(b,h , \nu)} && {(\Phi  \mu^\sharp_C  \eta^\sharp_{\tilde T C})(b,h , \nu)} && {(\Phi  \tilde T \Phi  \eta^\sharp_{\tilde T C})(b,h , \nu)} \\
			&&&& {(\Phi  \eta^\sharp_{C} \Phi)(b,h , \nu)} && {\Phi (b,h , \nu)}
			\arrow["{(\Phi*\mathfrak{p}_C^{-1})_{(b,h,\nu)}}", from=1-1, to=1-3]
			\arrow["{(\mathfrak{m}*\eta^\sharp_{\tilde T C})_{(b,h,\nu)}}", from=1-3, to=1-5]
			\arrow[equals, from=1-5, to=2-5]
			\arrow["{(\mathfrak{i}*\Phi)_{(b,h,\nu)}}", from=2-5, to=2-7]
		\end{tikzcd} \] 
		is the identity, which is also equal to the composite:
		\[ (*)\begin{tikzcd}[sep = 26pt]
			{\Phi(b,h , \nu)} &[-14pt]& {\Phi(\ell, a, Q(\bullet))} &[+10pt]& {\Phi(1_{\Bcat}, h^C\nu, \eta_{1_{\Bcat}}(\bullet))} &[-28pt]& {h^C(\nu).}
			\arrow["{\Phi(\bar u,\id,\id)}", from=1-1, to=1-3]
			\arrow["{\mathfrak{m}_{(1_{\Bcat}, (b,h,\nu), \eta_{1_{\Bcat}}(\bullet))}}", from=1-3, to=1-5]
			\arrow["{\mathfrak{i}_{h^C(\nu)}}", from=1-5, to=1-7]
		\end{tikzcd} \] 
		Unwapping definitions, $(*)$ coincides with the composite:
		\[ \begin{tikzcd}[column sep = small]
			{h^C(\nu) =} &[-29pt] {h^CT \bar u(Q(\bullet))} \\
			& {(h^C\eta_b J \bar u)^C(Q(\bullet))} \\
			& {(h J\bar u)^C(Q(\bullet)) = } &[-20pt] {a^CQ(\bullet)=} &[-44pt] {a^CQ^*(\eta_{1_{\Bcat}}(\bullet))} \\
			&&& {(a^CQ)^C(\eta_{1_{\Bcat}}(\bullet))} \\
			&&& {(\Phi\circ (b,h,\nu))^C(\eta_{1_{\Bcat}}(\bullet)) =} &[-19pt] {(h^C\nu)^C(\eta_{1_{\Bcat}}(\bullet))} 
			&[5pt] {h^C(\nu)}
			\arrow["{(\Delta_{h,\eta_b J \bar u}*Q)_\bullet}", from=1-2, to=2-2]
			\arrow["{((\Gamma_h*J\bar u)^C*Q)_\bullet}", from=2-2, to=3-2]
			\arrow["{(\Delta_{a,Q}* \eta_{1_{\Bcat}})_\bullet}", from=3-4, to=4-4]
			\arrow["{(\tau^C*\eta_{1_{\Bcat}})_\bullet}", from=4-4, to=5-4]
			\arrow["{(\Gamma_{h^C\nu})_\bullet}", from=5-5, to=5-6]
		\end{tikzcd} \] 
		Note then that, by naturality of the family $\set{ \Gamma }$ with respect to $\tau$ and by axiom (b) of a colax algebra for $\braket{T , \eta, (-)^*}$ (see Definition \ref{def:colax-algebra}), the diagram
		\[ \begin{tikzcd}
			{a^CQ^*(\eta_{1_{\Bcat}}(\bullet))} & {a^CQ(\bullet)} \\
			{(a^CQ)^C(\eta_{1_{\Bcat}}(\bullet))} && {a^CQ(\bullet)} \\
			{(\Phi\circ (b,h,\nu))^C(\eta_{1_{\Bcat}}(\bullet))} \\
			{(h^C\nu)^C(\eta_{1_{\Bcat}}(\bullet))} && {h^C(\nu)}
			\arrow["{{(\Delta_{a,Q}* \eta_{1_{\Bcat}})_\bullet}}"', from=1-1, to=2-1]
			\arrow[equals, from=1-2, to=1-1]
			\arrow[curve={height=-12pt}, equals, from=1-2, to=2-3]
			\arrow["{{(\Gamma_{a^CQ})_\bullet}}"{description}, from=2-1, to=2-3]
			\arrow["{{(\tau^C*\eta_{1_{\Bcat}})_\bullet}}"', from=2-1, to=3-1]
			\arrow["{{\tau_\bullet}}", from=2-3, to=4-3]
			\arrow[equals, from=3-1, to=4-1]
			\arrow["{{(\Gamma_{h^C\nu})_\bullet}}"', from=4-1, to=4-3]
		\end{tikzcd} \] 
		commutes, meaning that we can reduce to showing that the diagram
		\[ \begin{tikzcd}
			{h^CT\bar u(Q(\bullet))} & {h^C(\nu)} \\
			{(h^C\eta_b J \bar u)^C(Q(\bullet))} \\
			{(hJ\bar u)^C(Q(\bullet))} & {a^CQ(\bullet)}
			\arrow["{{(\Delta_{h,\eta_{b} J \bar u}*Q)_\bullet}}"', from=1-1, to=2-1]
			\arrow[equals, from=1-2, to=1-1]
			\arrow["{{((\Gamma_h*J\bar u)^C*Q)_\bullet}}"', from=2-1, to=3-1]
			\arrow[equals, from=3-1, to=3-2]
			\arrow["{\tau_\bullet}"', from=3-2, to=1-2]
		\end{tikzcd} \]
		commutes. Expanding, this is the outer border of the diagram
		\[ \begin{tikzcd}[column sep = 6pt, row sep = 22pt]
			{h^CT\bar u(Tu (\nu))} && {h^C(\nu)} & {(aJu)^C(\nu)} \\
			& {(h^C\eta_b )^C (\nu)} \\
			\\
			& {((h^C\eta_b J \bar u)^C\eta_{\ell}Ju)^C (\nu)} && {(a^C\eta_\ell Ju)^C(\nu)} \\
			{(h^C\eta_b J \bar u)^C(Tu (\nu))} && {((h  J \bar u)^C\eta_{\ell}Ju)^C (\nu)} \\
			\\
			{(hJ\bar u)^C(Tu(\nu))} &&& {a^C(Tu(\nu))}
			\arrow[equals, from=1-1, to=1-3]
			\arrow[""{name=0, anchor=center, inner sep=0}, "{{{(\Delta_{h,\eta_b J \bar u}*Tu)_\nu}}}"', from=1-1, to=5-1]
			\arrow[equals, from=1-3, to=1-4]
			\arrow["{{(\Gamma_h)^C_\nu}}"{description}, from=2-2, to=1-3]
			\arrow[""{name=1, anchor=center, inner sep=0}, "{{(\Gamma_{h^C\eta_b J \bar u}*Ju)^C_\nu}}"{description}, from=4-2, to=2-2]
			\arrow[""{name=2, anchor=center, inner sep=0}, "{{((\Gamma_h*J\bar u)^C*\eta_{\ell}Ju)^C_\nu}}"', from=4-2, to=5-3]
			\arrow[""{name=3, anchor=center, inner sep=0}, "{{(\Gamma_a*Ju)^C_\nu}}"', from=4-4, to=1-4]
			\arrow["{{(\Delta_{h^C\eta_b J \bar u, \eta_{\ell} J u})_\nu}}"{pos=0.6}, from=5-1, to=4-2]
			\arrow["{{{((\Gamma_h*J\bar u)^C*Tu)_\nu}}}"', from=5-1, to=7-1]
			\arrow[""{name=4, anchor=center, inner sep=0}, "{{(\Gamma_{hJ\bar u}*Ju)^C_\nu}}"{description}, from=5-3, to=1-3]
			\arrow[equals, from=5-3, to=4-4]
			\arrow[""{name=5, anchor=center, inner sep=0}, equals, from=7-1, to=7-4]
			\arrow["{{(\Delta_{a, \eta J u})_\nu}}"', from=7-4, to=4-4]
			\arrow["{{(i)}}"{description}, draw=none, from=0, to=1]
			\arrow["{{(ii)}}"{description}, draw=none, from=1, to=4]
			\arrow["{{(iv)}}"{description}, draw=none, from=2, to=5]
			\arrow["{{(iii)}}"{description}, draw=none, from=4, to=3]
		\end{tikzcd} \]
		where:
		\begin{enumerate}
			\item[(i)]commutes by axioms (a), (b) and (c) of a colax algebra for $\braket{T, \eta, (-)^*}$;
			\item[(ii)]commutes by naturality of the family $\set{\Gamma}$ with respect to $\Gamma_h*J\bar u$;
			\item[(iii)]commutes being $h \circ J \bar u = a$;
			\item[(iv)]commutes by naturality of the family $\set{\Delta}$ with respect to $\Gamma_h*J\bar u$.
		\end{enumerate}
		
		The other axioms can be shown with similar arguments.
	\end{proof}
\end{lemma}

\begin{lemma}\label{lem:rel-to-abs_is-morphism}
	Given a colax morphism $\braket{U, \upsilon} \colon \braket{C, (-)^C, \Gamma,\Delta} \to \braket{C', (-)^{C'}, \Gamma', \Delta'}$ of colax algebras for the $J$-relative 2-monad $\braket{T, \eta, (-)^*}$, the pair $\braket{U, \mathfrak{u}}$ defined in \ref{sssec:rel-to-abs_on-morphisms} is a colax morphism $\braket{C, \Phi, \mathfrak{i},\mathfrak{m}} \to \braket{C', \Phi', \mathfrak{i}', \mathfrak{m}'}$ of colax algebras for the pseudomonad $\braket{\tilde T , \eta^\sharp, \mu^\sharp}$.
	\begin{proof}
		We will verify axiom (a) of Definition \ref{def:colax-algebras-for-pseudomonad}, that is the following pastings coincide:
		\[ \begin{tikzcd}[row sep = 28pt]
			& {\tilde T C} && C \\
			C & {\tilde T C'} && {C',} \\
			{C'}
			\arrow["\Phi", from=1-2, to=1-4]
			\arrow["{{\tilde T U}}"{description}, from=1-2, to=2-2]
			\arrow["{{\mathfrak{u}}}"', shorten <=9pt, shorten >=9pt, Rightarrow, from=1-4, to=2-2]
			\arrow["U", from=1-4, to=2-4]
			\arrow["{{\eta_C^\sharp}}", from=2-1, to=1-2]
			\arrow["U"', from=2-1, to=3-1]
			\arrow["{\Phi'}"{description}, from=2-2, to=2-4]
			\arrow["{{\eta_{C'}^\sharp}}", from=3-1, to=2-2]
			\arrow[""{name=0, anchor=center, inner sep=0}, curve={height=20pt}, equals, from=3-1, to=2-4]
			\arrow["{{\mathfrak{i}'}}", shorten <=4pt, shorten >=4pt, Rightarrow, from=2-2, to=0]
		\end{tikzcd} \qquad \begin{tikzcd}[row sep = 28pt]
			& {\tilde T C} && C \\
			C &&& {C'.} \\
			{C'}
			\arrow["\Phi", from=1-2, to=1-4]
			\arrow["U", from=1-4, to=2-4]
			\arrow["{{\eta_C^\sharp}}", from=2-1, to=1-2]
			\arrow[""{name=0, anchor=center, inner sep=0}, curve={height=20pt}, equals, from=2-1, to=1-4]
			\arrow["U"', from=2-1, to=3-1]
			\arrow[curve={height=20pt}, equals, from=3-1, to=2-4]
			\arrow["{{\mathfrak{i}}}"', shorten <=4pt, shorten >=4pt, Rightarrow, from=1-2, to=0]
		\end{tikzcd}\]
		Fix an object $c \in C$. Showing that the above equality holds amounts to showing that the diagram:
		\[ \begin{tikzcd}
			{(U \circ \Phi \circ \eta^\sharp_C)(c)} && {(\Phi'\circ \tilde T U \circ \eta^\sharp_C)(c)} \\
			{U(c)} && {(\Phi'\circ \eta^\sharp_{C'}\circ U)(c)}
			\arrow["{(\mathfrak{u}*\eta^\sharp_C)_c}", from=1-1, to=1-3]
			\arrow["{(U*\mathfrak{i})_c}"', from=1-1, to=2-1]
			\arrow[equals, from=1-3, to=2-3]
			\arrow["{(\mathfrak{i}'*U)_c}", from=2-3, to=2-1]
		\end{tikzcd} \] 
		commutes. Unwrapping the definitions, this coincides with the diagram:
		\[\begin{tikzcd}
			{U(c^C(\eta_{1_{\Bcat}}(\bullet)))} && {(Uc)^{C'}(\eta_{1_{\Bcat}}(\bullet))} \\
			& {U(c)}
			\arrow["{(\upsilon_c)_{\eta_{1_{\Bcat}}(\bullet)}}", from=1-1, to=1-3]
			\arrow["{U((\Gamma_c)_\bullet)}"', from=1-1, to=2-2]
			\arrow["{(\Gamma'_{Uc})_\bullet}", from=1-3, to=2-2]
		\end{tikzcd}\]
		which commutes by axiom (a) of a colax morphism of colax algebras for $\braket{T , \eta, (-)^*}$. Axiom (b) can be shown similarly.
	\end{proof}
\end{lemma}

\begin{lemma}\label{lem:abs-to-rel_is-algebra}
	Given a colax algebra $\braket{C,\Phi,\mathfrak{i},\mathfrak{m}}$ for the pseudomonad $\braket{\tilde T, \eta^\sharp, \mu^\sharp}$, the tuple $\braket{C, (-)^C, \Gamma,\Delta}$ defined in \ref{sssec:abs-to-rel_on-algebras} is a colax algebra for the $J$-relative 2-monad $\braket{T, \eta, (-)^*}$.
	\begin{proof}
		We will verify axiom (a) of Definition \ref{def:colax-algebra}, i.e.\ that for each functor $h \colon J b \to C$, the diagram:
		\[ \begin{tikzcd}
			{h^C\eta^*_b} && {(h^C\eta_b)^C} \\
			& {h^C}
			\arrow["{\Delta_{h,\eta_b}}", from=1-1, to=1-3]
			\arrow[equals, from=1-1, to=2-2]
			\arrow["{(\Gamma_h)^C}", from=1-3, to=2-2]
		\end{tikzcd} \] 
		commutes. Let $\nu \in T b$. Unwrapping the definitions of $\Delta$ and of $(-)^C$, this means showing that the composite:
		\[ (*)\begin{tikzcd}[row sep = small, column sep = 7pt]
			{\Phi(b, h , \eta^*_b(\nu))} &[-4pt]& {\Phi(\ell, a , Q^*(\nu))} \\
			&& {(\Phi\mu^\sharp_C)(b, \iota_b(h)\eta_b  , \nu)} &[+15pt]& {(\Phi\tilde T \Phi)(b, \iota_b(h)\eta_b  , \nu)} \\
			&&&& {\Phi(b,h^C\eta_b,\nu)} &[-3pt]& {\Phi(b, h,\nu)}
			\arrow["{\Phi(t,\id,\id)}", from=1-1, to=1-3]
			\arrow[equals, from=1-3, to=2-3]
			\arrow["{\mathfrak{m}_{(b, \iota_b(h)\eta_b,\nu)}}", from=2-3, to=2-5]
			\arrow[equals, from=2-5, to=3-5]
			\arrow["{\Phi(\id,\Gamma_h,\id)}", from=3-5, to=3-7]
		\end{tikzcd} \] 
		is the identity, where $(\ell, a \colon J \ell \to C, (b , Q \colon J b \to T \ell, \nu)) \coloneqq \mathfrak{s}_{T,T} ( b, \iota_b(h)\eta_b , \nu)$ as in \ref{sssec:abs-to-rel_on-algebras} and where $t \colon \ell \to b \in \Bcat$ is the canonical 1-cell induced being $b$ an oplax cocone on the diagram of which $\ell$ is the oplax colimit. Consider then the diagram
		\[ \begin{tikzcd}
			{\Phi(b, h , \eta^*_b(\nu))} && {\Phi(\ell, a , Q^*(\nu))} \\
			&& {(\Phi\circ\mu^\sharp_C)(b, \iota_b(h)\eta_b  , \nu)} \\
			&& {(\Phi\circ\mu^\sharp_C)(b, \eta^\sharp_C h   , \eta^*_b(\nu) ) } \\
			{\Phi(b, h , \eta^*_b(\nu))} && {\Phi(\bar \ell, \bar a, \bar Q^*\eta_b^*(\nu)) ,}
			\arrow["{\Phi(t,\id,\id)}", from=1-1, to=1-3]
			\arrow[equals, from=1-1, to=4-1]
			\arrow[equals, from=1-3, to=2-3]
			\arrow["{(\Phi\circ \mu^\sharp_C)(\id, \sigma, \id)}", from=2-3, to=3-3]
			\arrow["{\Phi(\bar v,\id,\id)}"', from=4-1, to=4-3]
			\arrow[equals, from=4-3, to=3-3]
		\end{tikzcd} \] 
		where the arrow 
		\[ (\bar v , \id, \id) \colon ( b, h , \eta_b^*(\nu)) \to ( \bar \ell , \bar a , \bar Q ^* \eta_b^*(\nu)) \coloneqq (\mu^\sharp_C \circ \tilde T \eta^\sharp_C)(b, h , \eta_b^*(\nu)) \]
		is the component at $( b, h , \eta_b^*(\nu))$ of the right unit $\mathfrak{q}_C$, meaning that $\bar v \colon \bar \ell \to b$ is the inverse of the canonical invertible 1-cell $v \colon b \to \bar \ell$ defined in Proposition \ref{prop:structure-on-Bcat-to-CAT}, and where $\sigma \colon \iota_b(h) \eta_b \implies \eta_C^\sharp h$ is the natural transformation defined on $x \in J b$ by the arrow
		\[ (\dot x, \id, \id) \colon \iota_b(h)(\eta_b(x)) = (b, h , \eta_b(x)) \to (1_{\Bcat}, hx , \eta_{1_{\Bcat}}(\bullet)) = \eta_C^\sharp (h(x)). \]
		Note that, by the universal properties of oplax colimits (and by functoriality of $\Phi$), the above diagram commutes. Moreover, the diagram
		\[ \begin{tikzcd}[column sep = small]
			{(\Phi\mu^\sharp_C)(b, \iota_b(h)\eta_b  , \nu)} &[+14pt]& {(\Phi\tilde T \Phi)(b, \iota_b(h)\eta_b  , \nu)} &[-8pt]& {\Phi(b,h^C\eta_b,\nu)} \\
			\\
			{(\Phi\mu^\sharp_C)(b, \eta^\sharp_C h   , \eta^*_b(\nu) )} &[+14pt]& {(\Phi \tilde T \Phi )(b, \eta^\sharp_C h   , \eta^*_b(\nu) )} && {(\Phi \tilde T \Phi  \tilde T \eta^\sharp_C )(b,  h   , \eta^*_b(\nu) )}
			\arrow["{\mathfrak{m}_{(b, \iota_b(h)\eta_b,\nu)}}", from=1-1, to=1-3]
			\arrow["{(\Phi \mu^\sharp_C)(\id, \sigma, \id)}"', from=1-1, to=3-1]
			\arrow[equals, from=1-3, to=1-5]
			\arrow["{(\Phi\tilde T \Phi)(\id, \sigma, \id)}"', from=1-3, to=3-3]
			\arrow["{\Phi(\id, \Phi * \sigma, \id)}", from=1-5, to=3-5]
			\arrow["{\mathfrak{m}_{(b, \eta_C^\sharp h,\eta_b^*(\nu))}}"', from=3-1, to=3-3]
			\arrow[equals, from=3-3, to=3-5]
		\end{tikzcd} \] 
		clearly commutes, by naturality of $\mathfrak{m}$ (for the left-hand square) and by definition of $\tilde T \Phi$ (for the right-hand square). Finally, note that the diagram
		\[ \begin{tikzcd}
			{\Phi(b,h^C\eta_b,\nu)} && {\Phi(b,h,\nu)} \\
			\\
			{(\Phi\circ \tilde T \Phi \circ \tilde T \eta^\sharp_C )(b,  h   , \eta^*_b(\nu) )} && {\Phi(b,h,\eta_b^*(\nu))}
			\arrow["{\Phi(\id,\Gamma_h,\id)}", from=1-1, to=1-3]
			\arrow["{\Phi(\id, \Phi * \sigma, \id)}"', from=1-1, to=3-1]
			\arrow[equals, from=1-3, to=3-3]
			\arrow["{\Phi((\tilde T \mathfrak{i})_{(b,h,\eta_b^*(\nu))})}"', from=3-1, to=3-3]
		\end{tikzcd} \]
		also commutes. This follows by functoriality of $\Phi$ as the diagram of natural transformations
		\[ \begin{tikzcd}
			{\Phi\circ \iota_b(h)\circ \eta_b} && h \\
			& {\Phi\circ \eta_C^\sharp \circ h}
			\arrow["{\Gamma_h}", from=1-1, to=1-3]
			\arrow["{\Phi*\sigma}"', from=1-1, to=2-2]
			\arrow["{\mathfrak{i}*h}"', from=2-2, to=1-3]
		\end{tikzcd} \] 
		commutes, i.e.\ the diagram 
		\[ \begin{tikzcd}
			{\Phi(b,h,\eta_b(x))} && h \\
			& {\Phi(1_{\Bcat}, hx, \eta_{1_{\Bcat}}(\bullet))}
			\arrow["{(\Gamma_h)_x}", from=1-1, to=1-3]
			\arrow["{\Phi(\dot x,\id,\id)}"', from=1-1, to=2-2]
			\arrow["{\mathfrak{i}_{h(x)}}"', from=2-2, to=1-3]
		\end{tikzcd} \] 
		commutes for each $x\in J b$, which is exactly the definition of $(\Gamma_h)_x$. Summing up, pasting commutative squares, we have shown that $(*)$ is the identity if and only if the composite
		\[ \begin{tikzcd}[row sep = small]
			{\Phi(b, h , \eta^*_b(\nu))} \\
			\\
			{\Phi(\ell', a', Q'^*\eta_b^*(\nu))} &[-30pt] {=(\Phi\mu^\sharp_C)(b, \eta^\sharp_C h   , \eta^*_b(\nu) )} \\
			\\
			& {(\Phi \tilde T \Phi )(b, \eta^\sharp_C h   , \eta^*_b(\nu) )} &[-34pt] {=(\Phi \tilde T \Phi  \tilde T \eta^\sharp_C )(b,  h   , \eta^*_b(\nu) )} \\
			\\
			&& {\Phi(b, h,\nu)}
			\arrow["{{\Phi(v,\id,\id)}}", from=1-1, to=3-1]
			\arrow["{{\mathfrak{m}_{(b, \eta_C^\sharp h,\eta_b^*(\nu))}}}", from=3-2, to=5-2]
			\arrow["{{\Phi((\tilde T \mathfrak{i})_{(b,h,\eta_b^*(\nu))})}}", from=5-3, to=7-3]        
		\end{tikzcd} \] 
		is also the identity. To conclude, note that this latter composite is the component at $(b, h, \eta_b^*(\nu))$ of the composite natural transformation represented in the diagram:
		\[ \begin{tikzcd}
			&&& {\tilde T C} \\
			{\tilde T C} && {\tilde T ^2 C} && C \\
			&&& {\tilde T C}
			\arrow["\Phi", from=1-4, to=2-5]
			\arrow["{{{{\mathfrak{m}}}}}"', shorten <=6pt, shorten >=6pt, Rightarrow, from=1-4, to=3-4]
			\arrow[""{name=0, anchor=center, inner sep=0}, curve={height=-20pt}, equals, from=2-1, to=1-4]
			\arrow["{{{{\tilde T \eta^\sharp_C}}}}"{description}, from=2-1, to=2-3]
			\arrow[""{name=1, anchor=center, inner sep=0}, curve={height=20pt}, equals, from=2-1, to=3-4]
			\arrow["{{{{\mu^\sharp_C}}}}"{description}, from=2-3, to=1-4]
			\arrow["{{{{\tilde T \Phi}}}}"{description}, from=2-3, to=3-4]
			\arrow["\Phi"', from=3-4, to=2-5]
			\arrow["{{{{\mathfrak{q}_C}}}}"', shorten <=4pt, Rightarrow, from=0, to=2-3]
			\arrow["{{{{\tilde T \mathfrak{i}}}}}"', shorten <=4pt, shorten >=4pt, Rightarrow, from=2-3, to=1]
		\end{tikzcd} \] 
		which, by axiom (a) of a colax algebra in Definition \ref{def:colax-algebras-for-pseudomonad}, is indeed the identity natural transformation. The other axioms can be shown with similar arguments.
	\end{proof}
\end{lemma}

\begin{lemma}\label{lem:abs-to-rel_is-morphism}
	Given a colax morphism $\braket{U, \mathfrak{u}} \colon \braket{C, \Phi, \mathfrak{i},\mathfrak{m}} \to \braket{C', \Phi', \mathfrak{i}', \mathfrak{m}'}$ of colax algebras for the pseudomonad $\braket{\tilde T, \eta^\sharp, \mu^\sharp}$, the pair $\braket{U, \upsilon }$ defined in \ref{sssec:abs-to-rel_on-morphisms} is a colax morphism $\braket{C, (-)^C, \Gamma,\Delta} \to \braket{C', (-)^{C'}, \Gamma',\Delta'}$ of colax algebras for the $J$-relative 2-monad $\braket{ T , \eta , (-)^*}$.
	\begin{proof}
		We will verify axiom (a) of Definition \ref{def:colax-algebra}, i.e.\ that for each functor $h \colon J b \to C$, the diagram
		\[ \begin{tikzcd}
			{Uh^C\eta_b} && {(Uh)^{C'}\eta_b} \\
			& Uh
			\arrow["{\upsilon_h*\eta_b}", from=1-1, to=1-3]
			\arrow["{U*\Gamma_h}"', from=1-1, to=2-2]
			\arrow["{\Gamma'_{Uh}}", from=1-3, to=2-2]
		\end{tikzcd} \] 
		commutes. Let $x\in J b$. Explicitly, this means showing that the diagram
		\[ \begin{tikzcd}
			{U\Phi(b,h,\eta_b(x))} & {(\Phi'\circ \tilde T U)(b,Uh,\eta_b(x))} & {\Phi'(b,Uh,\eta_b(x))} \\
			{U\Phi(1_{\Bcat}, hx, \eta_{1_{\Bcat}}(\bullet))} && {\Phi'(1_{\Bcat}, Uhx, \eta_{1_{\Bcat}}(\bullet))} \\
			& {Uh(x)}
			\arrow["{\mathfrak{u}_{(b,h,\eta_b(x))}}", from=1-1, to=1-2]
			\arrow["{U\Phi(\dot x , \id, \id)}"', from=1-1, to=2-1]
			\arrow[equals, from=1-2, to=1-3]
			\arrow["{\Phi'(\dot x , \id, \id)}", from=1-3, to=2-3]
			\arrow["{U(\mathfrak{i}_{h(x)})}"', from=2-1, to=3-2]
			\arrow["{\mathfrak{i}_{Uh(x)}}", from=2-3, to=3-2]
		\end{tikzcd} \]
		commutes. This follows by noting that this diagram factors as
		\[ \begin{tikzcd}
			{U\Phi(b,h,\eta_b(x))} & {(\Phi'\circ \tilde T U)(b,Uh,\eta_b(x))} & {\Phi'(b,Uh,\eta_b(x))} \\
			\\
			{U\Phi(1_{\Bcat}, hx, \eta_{1_{\Bcat}}(\bullet))} & {(\Phi'\circ \tilde T U)(1_{\Bcat}, h x , \eta_{1_{\Bcat}}(\bullet))} & {\Phi'(1_{\Bcat}, Uhx, \eta_{1_{\Bcat}}(\bullet))} \\
			& {Uh(x)}
			\arrow["{\mathfrak{u}_{(b,h,\eta_b(x))}}", from=1-1, to=1-2]
			\arrow[""{name=0, anchor=center, inner sep=0}, "{U\Phi(\dot x , \id, \id)}"', from=1-1, to=3-1]
			\arrow[equals, from=1-2, to=1-3]
			\arrow[""{name=1, anchor=center, inner sep=0}, "{(\Phi'\circ \tilde T U )(\dot x , \id, \id)}"{description}, from=1-2, to=3-2]
			\arrow[""{name=2, anchor=center, inner sep=0}, "{\Phi'(\dot x , \id, \id)}", from=1-3, to=3-3]
			\arrow["{\mathfrak{u}_{(1_{\Bcat}, h x , \eta_{1_{\Bcat}}(\bullet))}}", from=3-1, to=3-2]
			\arrow["{U(\mathfrak{i}_{h(x)})}"', from=3-1, to=4-2]
			\arrow[equals, from=3-2, to=3-3]
			\arrow["{(iii)}"{description}, draw=none, from=3-2, to=4-2]
			\arrow["{\mathfrak{i}'_{Uh(x)}}", from=3-3, to=4-2]
			\arrow["{(i)}"{description}, draw=none, from=0, to=1]
			\arrow["{(ii)}"{description}, draw=none, from=1, to=2]
		\end{tikzcd} \]
		where (i) commutes by naturality of $\mathfrak{u}$, (ii) commutes by definition of $\tilde T U$, and (iii) commutes by axiom (a) of a colax morphism in Definition \ref{def:colax-algebras-for-pseudomonad}. Axiom (b) can be shown similarly.
	\end{proof}
\end{lemma}